\documentclass[smallextended]{svjour3}       
\smartqed 

\usepackage[utf8]{inputenc}
\usepackage{cmap}
\usepackage[T1]{fontenc}
\usepackage{lmodern}
\usepackage[square, numbers, sort]{natbib}
\usepackage[colorlinks=true,hypertexnames=false]{hyperref}
\usepackage{amssymb, amsmath, amsfonts}
\usepackage{authblk, dsfont, mathrsfs, color, bm, mathtools}
\usepackage{graphicx, transparent}
\usepackage{caption, subcaption}
\usepackage{enumitem} 

\allowdisplaybreaks

\graphicspath{{figs/}}
\usepackage[dvipsnames]{xcolor}
\hypersetup{colorlinks, linkcolor={red!80!black}, citecolor={blue!80!black}, urlcolor={ForestGreen}}

\definecolor{gold}{rgb}{0.85,0.65,0}

\let\emptyset\varnothing
\newcommand{\set}[1]{\left\{#1\right\}}

\def\C{{\mathbb{C}}}

\def\H{{\mathbb{H}}}

\def\K{{\mathbb{K}}}

\def\P{{\mathbb{P}}}

\def\R{{\mathbb{R}}}

\def\bA{{\bm{A}}}
\def\bB{{\bm{B}}}
\def\bC{{\bm{C}}}

\def\bF{{\bm{F}}}

\def\bH{{\bm{H}}}
\def\bI{{\bm{I}}}

\def\cF{{\cal F}}

\def\cH{{\cal H}}
\def\cI{{\cal I}}

\def\cK{{\cal K}}
\def\cL{{\cal L}}
\def\cM{{\cal M}}

\def\cO{{\cal O}}

\def\cR{{\cal R}}
\def\cS{{\cal S}}
\def\cT{{\cal T}}
\def\cU{{\cal U}}
\def\cV{{\cal V}}
\def\cW{{\cal W}}
\def\cX{{\cal X}}

\def\cZ{{\cal Z}}

\DeclareMathOperator{\rank}{rank}

\DeclareMathOperator{\rint}{rint}
\DeclareMathOperator{\cl}{cl}

\DeclareMathOperator{\dom}{dom}

\DeclareMathOperator{\epi}{epi}

\DeclareMathOperator{\conv}{conv}

\DeclareMathOperator{\rec}{rec}

\DeclareMathOperator{\sign}{sign}
\DeclareMathOperator{\relx}{relax}

\DeclareMathOperator{\Proj}{Proj}

\newcommand{\ones}{{\bm{1}}}

\newcommand{\Sbar}{\cS(\Delta, \K)}

\newcommand{\x}{{\bm x}}
\newcommand{\z}{{\bm z}}

\newcommand{\f}{{\bm f}}
\newcommand{\g}{{\bm g}}
\newcommand{\e}{{\bm e}}
\newcommand{\w}{{\bm w}}
\newcommand{\s}{{\bm s}}
\renewcommand{\a}{{\bm a}}

\renewcommand{\u}{{\bm u}}

\renewcommand{\t}{{\bm t}}
\newcommand{\ga}{{\bm \alpha}}
\newcommand{\gb}{{\bm \beta}}
\newcommand{\gd}{{\bm \delta}}
\newcommand{\gm}{{\bm \mu}}

\begin{document}

\title{Constrained Optimization of Rank-One Functions with Indicator Variables}

\author{Soroosh Shafiee \\
Fatma K{\i}l{\i}n\c{c}-Karzan
}

\authorrunning{Soroosh Shafiee, Fatma K{\i}l{\i}n\c{c}-Karzan} 

\institute{
	Soroosh Shafiee, Operations Research and Information Engineering, Cornell University \\ \email{shafiee@cornell.edu} \\
	Fatma K{\i}l{\i}n\c{c}-Karzan, Tepper School of Business, Carnegie Mellon University \\ \email{fkilinc@andrew.cmu.edu} 
}

\maketitle

\begin{abstract}
Optimization problems involving minimization of a rank-one convex function over constraints modeling restrictions on the support of the decision variables emerge in various machine learning applications. These problems are often modeled with indicator variables for identifying the support of the continuous variables. In this paper we investigate compact extended formulations for such problems through perspective reformulation techniques. In contrast to the majority of previous work that relies on support function arguments and disjunctive programming techniques to provide convex hull results, we propose a constructive approach that exploits a hidden conic structure induced by perspective functions. To this end, we first establish a convex hull result for a general conic mixed-binary set in which each conic constraint involves a linear function of independent continuous variables and a set of binary variables. We then demonstrate that extended representations of sets associated with epigraphs of rank-one convex functions over constraints modeling indicator relations naturally admit such a conic representation. This enables us to systematically give perspective formulations for the convex hull descriptions of these sets with nonlinear separable or non-separable objective functions, sign constraints on continuous variables, and combinatorial constraints on indicator variables. We illustrate the efficacy of our results on sparse nonnegative logistic regression problems.

\keywords{Mixed-integer nonlinear optimization, indicator variables, perspective function, combinatorial constraints, convex hull}
\end{abstract}

\section{Introduction}\label{sec:intro}

We consider specific classes of the general mixed-binary convex optimization problem with indicator variables

\begin{align}
	\label{eq:mixed_binary_convex}
	\min_{\x, \z} \set{H(\x):~ (\x, \z) \in \cX \times \cZ,~ x_i (1 - z_i) = 0, ~ \forall i \in [d]},
\end{align}
where $H: \R^m \to \R$ is a convex function, $\cX \times \cZ \subseteq \R^d \times \set{0, 1}^d$ denotes the feasible set, and $[d] := \set{1, \dots, d}$. Each binary variable $z_i$ determines whether a continuous variable $x_i$ is zero or not by requiring $x_i = 0$ when $z_i = 0$ and allowing $x_i$ to take any value when $z_i = 1$. The sets $\cX$ and $\cZ$ model restrictions on continuous and binary variables, respectively.

The most commonly studied constraints defining $\cZ$ are cardinality, hierarchy and multicolinearity constraints. For example, the cardinality set,  i.e., $\cZ = \{ \z \in \{0,1\}^d: \ones^\top \z \leq \kappa \}$, is widely used in the best subset selection problem in statistics \citep{bertsimas2016best,bertsimas2020sparse}. Several other restrictions for $\cZ$ also arise in statistical problems; e.g., hierarchy constraints of the form of $z_i \leq z_j$ are used in~\citep{bien2013lasso,hazimeh2020learning}, 
constraints on a subset of indicator variables to prevent multicollinearity are proposed in~\citep{bertsimas2016or},
constraints linking the indicator variables associated with regression coefficients in the same group in group variable selection are explored in~\citep{hazimeh2023grouped,huang2012selective},
and constraints for cycle prevention in causal graphical variable selection problems in regression problems are examined in \citep{kucukyavuz2020consistent,manzour2021integer}. 

Various other constraints on both continuous and discrete variables emerge as means to enforce priors from human experts to enhance interpretability and/or improve prediction performance, see e.g.,~\citep{cozad2014learning}. Among these, nonnegativity constraints on continuous variables naturally appear in a range of applications, including signal recovery~\citep{atamturk2021sparse}, portfolio selection~\citep{atamturk2020supermodularity,han2021compact}, healthcare and criminal justice \citep{rudin2018optimized}, chemical process simulation~\citep{cozad2014learning,cozad2015combined}, and spectral decomposition~\citep{behdin2021archetypal}.

Problem~\eqref{eq:mixed_binary_convex} is NP-hard~\citep{natarajan1995sparse}, and as a result in the literature its convex surrogates like lasso~\citep{tibshirani1996regression,hastie2015statistical} as well as branch-and-bound based exact methods~\citep{bertsimas2016best} have been studied. In this paper, we study the epigraph associated with~\eqref{eq:mixed_binary_convex}, i.e., the set
\begin{align*}
	\set{(\tau, \x, \z) \in \R \times \cX \times \cZ :~ H(\x) \leq \tau,~ x_i (1 - z_i) = 0, ~ \forall i \in [d]},
\end{align*}
and provide its closed convex hull characterization.
The primary challenge in this class of problems stems from the complementary constraints between $x_i$ and $z_i$. Whenever an a priori bound on the magnitude of the continuous variables $|x_i|$ is known, these constraints can be linearized via the big-M method. However, finding a suitable bound on the continuous variables to set the big-M parameter can be very challenging, and the resulting big-M formulations are known to be weak. Dating back to \citet{ceria1999convex}, the perspective functions have played a significant role in offering big-M free reformulations of mixed-binary convex optimization problems. In particular, \citet{frangioni2006perspective} introduce perspective cuts based on a linearization of perspective functions. \citet{akturk2009strong} then show that perspective relaxations can be viewed as implicitly including all (infinitely many) perspective cuts. \citet{gunluk2010perspective} examine a \emph{separable} structure where the objective function of~\eqref{eq:mixed_binary_convex} constitutes a sum of univariate functions, taking the form $H(x) = \sum_{i \in [d]} h_i(x_i)$ for some lower semicontinuous and convex functions $h_i: \R \to \R$, and 
study the mixed-binary set
\begin{align}
	\label{eq:H}
	\cH := \set{(\tau, \x, \z) \in \R \times \cX \times \cZ :~ 
	\begin{array}{l}
		\sum_{i \in [d]} h_i(x_i) \leq \tau, \\[1ex]
		x_i (1 - z_i) = 0, ~ \forall i \in [d]
	\end{array}
	}
\end{align}
when $\cX = \R^d$ and  $\cZ = \set{0,1}^d$. Using the perspective reformulation technique,  \cite{gunluk2010perspective} presents an \emph{ideal} formulation of the closed convex hull of $\cH$ ($\cl\conv(\cH)$) in the original space of variables.
\citet{xie2020scalable} revisit the separable structure and give a perspective formulation of $\cl\conv(\cH)$ when $\cX = \R^d, \cZ = \{z \in \set{0,1}^d: \ones^\top \z \leq \kappa \}$ for some integer $1 \leq \kappa \leq d$, and the functions $h_i$ are convex and quadratic. \citet{bacci2019new} extend these findings to convex differentiable functions under certain constraint qualification conditions. More recently, under this separable function assumption, using a support function argument, \citet{wei2022ideal} further generalize \cite{xie2020scalable} and provide an ideal perspective formulation for $\cl\conv(\cH)$ when $\cX = \R^d, \cZ \subseteq \set{0,1}^n$ modeling combinatorial relations, and the functions $h_i: \R \to \R$ are lower semicontinuous and convex. 

In a number of important applications, including portfolio optimization in finance \citep{bienstock1996computational,wei2022convex}, sparse classification and regression in machine learning \citep{atamturk2020safe,bertsimas2016best,bertsimas2021sparse,bertsimas2020sparse,deza2022safe,hazimeh2020fast,hazimeh2021sparse,xie2020scalable} and  their nonnegative variants~\citep{atamturk2020supermodularity,atamturk2021sparse}, and outlier detection in statistics \citep{gomez2021outlier}, the objective function $H$ constitutes a finite sum of \emph{rank-one convex functions}, taking the form $H(\x) = \sum_{j \in [N]} h_j(\a_j^\top \x)$ for some convex functions $h_j$ and vectors $\a_j \in \mathbb R^d$. 
For example, in the case of sparse least squares regression, $N$ denotes the number of samples and $h_j(\a_j^\top \x) = (\a_j^\top \x - b_j)^2$, where $\a_j$ represents the feature or input vector and $b_j$ represents the label or output; or  in the case of logistics regression we have $h_j(\a_j^\top \x) =\log (1+\exp(-b_j\a_j^\top \x))$.
As a result, a growing stream of research \citep{jeon2017quadratic,atamturk2018strong,frangioni2020decompositions,atamturk2021sparse,liu2022graph,wei2022convex} studies~\eqref{eq:mixed_binary_convex} when $H(\x)$ is a \emph{non-separable} quadratic function.  
More recently, a number of papers \citep{atamturk2019rank,atamturk2020supermodularity,han2021compact,wei2020convexification,wei2022ideal} offer strong perspective-based relaxations for rank-one non-separable objective functions by analyzing the closed convex hull of the mixed-binary~set
\begin{align}
	\label{eq:T}
	\cT := \set{(\tau, \x, \z) \in \R \times \cX \times \cZ :~ 
	\begin{array}{l}
		h(\a^\top \x) \leq \tau, \\ [1ex]
		x_i (1 - z_i) = 0, ~ \forall i \in [d]
	\end{array}
	}
\end{align}
under various assumptions on $h(\cdot)$, $\cX$, and $ \cZ$. 
In particular, \citet{atamturk2019rank} examine rank-one convex quadratic functions when $\cX = \R^d$ and $\cZ = \set{0,1}^d$. \citet{wei2020convexification} extend these findings by allowing constraints on the binary variables, that is, $\cZ \subseteq \set{0,1}^d$. Following up on this,  \citet{wei2022ideal} give a perspective formulation of the closed convex hull of $\cT$ ($\cl \conv(\cT)$) in the original space of variables when $\cX = \R^d, \cZ \subset \set{0,1}^d$, and the function $h:\R \to \R$ is lower semicontinuous and convex. Nonetheless, the formulations in \citep{wei2020convexification,wei2022ideal} require including a nonlinear convex inequality for each facet of $\conv(\cZ \backslash \set{{\bm 0}})$. Note that $\conv(\cZ \backslash \set{{\bm 0}})$ itself may be complicated and may require an exponential number of inequalities for its description even when $\cZ$ is a simple boolean set itself.  
In a separate thread, when $\cZ=\set{0,1}^d$, \citet{atamturk2020supermodularity} examine $\cl\conv(\cT)$ when $h$ is a convex quadratic function and there are additional nonnegativity requirements on some of the continuous variables, i.e., $\cX = \{\x \in \R^d:~ x_i \geq 0, ~ \forall i \in \cI \}$ for some $\cI \subseteq [d]$, and propose classes of valid inequalities for $\cl\conv(\cT)$. However, these inequalities given in the original problem space require a cutting-surface based implementation, which may result in numerical issues. To address such issues, \citet{han2021compact} present compact extended formulations for $\cl\conv(\cT)$ when $\cZ=\set{0,1}^d$. These extended formulations are not only applicable to general lower semicontinuous and convex function $h$, but are also easier to implement as they can be embedded within standard branch-and-bound based integer programming solvers.

In this paper, by linking the perspective formulations to conic programming, we propose a conic-programming based approach to address constrained optimization of rank-one functions with indicator variables. Our approach generalizes the existing results by studying simultaneously both sign restrictions on continuous variables and combinatorial constraints on binary variables. Specifically, we provide perspective formulations for $\cl\conv(\cH)$ and $\cl\conv(\cT)$ when $\cX \!=\! \set{\x \in \R^d: x_i \geq 0, \, \forall i \in \cI}$ for some $\cI \subseteq [d]$, $\cZ \subseteq \set{0,1}^d$ and all functions are proper, lower semicontinuous, and convex.
The crux of our approach relies on understanding the recessive and rounding directions associated with these sets involving complementarity constraints. Based on this understanding, we reduce the given complementarity constraints to fewer and much simpler to handle complementarity relations in a lifted space. For example, in the case of the set $\cT$ when $\cX$ does not contain any sign restrictions, following this approach we analyze a set involving a \emph{single} complementarity constraint with a single new binary variable. In this way, our approach provides an effective way to arrive at compact extended formulations for $\cl\conv(\cH)$ and $\cl\conv(\cT)$ in a more direct manner while simultaneously handling arbitrary $\cZ \subseteq \set{0,1}^d$ and arbitrary sign restrictions included in $\cX$. In contrast, the proof techniques in \citep{atamturk2019rank,atamturk2020supermodularity,wei2020convexification,han2021compact,wei2022ideal} are based on support function arguments and disjunctive programming methods, which have resulted in addressing restrictions on the sets $\cX$ and $\cZ$ in separate studies through different approaches.
The key contributions of this paper and its organization are summarized below.

\begin{itemize}[label=$\diamond$]
	\item We derive the necessary tools for the study of the separable and nonseparable sets in Section~\ref{sec:Preparatory}. 
	We begin by studying perspective functions in Section~\ref{sec:perspective}. Specifically, we analyze the epigraph of perspective functions in Lemmas~\ref{lem:epi:projective}~and~\ref{lem:epi:perspective}, and establish that such epigraphs are indeed convex cones. We conclude this subsection by introducing a mixed-binary set that will serve as the primary substructure for the sets $\cH$ and~$\cT$. Using the conic representation of perspective functions from Lemmas~\ref{lem:epi:projective}~and~\ref{lem:epi:perspective}, we show that this set is an instance of conic mixed-binary sets of the following form
	\begin{align}
		\label{eq:S:bar}
		\Sbar = \left\{ (\ga, \gd) \in \R^{m} \times \Delta :  \bA_i \ga_i + \bB_i \delta_i \in \K_i, ~ \forall i \in [p] \right\},
	\end{align}
	where $\Delta \subseteq \{ 0, 1 \}^n$, $\K = \times_{i \in [d]} \, \K_i$, $\K_i$ is convex cone containing the origin for every $i \in [p]$, and the matrices $\bA_i, \bB_i$ have appropriate dimensions. We use the notation $\ga = (\ga_1, \dots, \ga_p)$ to denote the vector $\ga \in \R^m$ in terms of its subvectors $\ga_i \in \R^{m_i}, i \in [p]$, where $\sum_{i \in [p]} m_i = m$.
	Motivated by this observation, we study and characterize the convex hulls of $\Sbar$ in Section~\ref{sec:convex_hull}. We show in Proposition~\ref{prop:conv} that $\conv(\Sbar) = \cS (\conv(\Delta), \K)$ as long as $\Delta \subseteq \set{0,1}^n$ and $\K_i$ is a convex cone containing the origin for all $i \in [d]$. We then establish a simple condition in Theorem~\ref{thm:clconv} under which $\cl\conv(\Sbar) = \cS (\conv(\Delta), \cl(\K))$ holds. 
	These results highlight that the complexity of the convex hull characterizations of sets of the form $\Sbar$ is solely determined by the complexity of the characterization of $\conv(\Delta)$.
	
	\item In Section~\ref{sec:applications}, we discuss the constrained optimization of both separable and rank-one functions with indicator variables through the lens of conic mixed-binary sets. 
	Specifically, we derive compact extended formulations for $\cl\conv(\cH)$ and $\cl\conv(\cT)$ when $\cX = \{ \x \in \R^d : x_i \geq 0, \forall i \in \cI \}$ for some index set $\cI \subseteq [d]$ and $\cZ \subseteq \set{0,1}^d$. 
	
	\begin{itemize}
		\item
		In Section~\ref{sec:applications:separable}, we investigate the separable case and establish $\cl\conv(\cH)$ in Theorem~\ref{thm:separable}. This result extends \citet[Theorem~3]{wei2022ideal} to proper (rather than real-valued) lower semicontinuous convex functions and by allowing for nonnegativity restrictions on $\x$.
	
		\item In Section~\ref{sec:applications:rank1:nosign}, we explore the non-separable case and examine the set $\cT$ when $\cX = \R^d$ and $\cZ \subseteq \set{0,1}^d$. We establish an extended description for $\cl\conv(\cT)$ in Theorem~\ref{thm:connected}. Our description requires a single new binary variable $w$ and relies on the convex hull description of an associated set $\Delta_1$ involving $w$ and $\z$ variables. 
		This therefore reduces the complexity of characterizing $\cl\conv(\cT)$ in this setting to understanding the complexity of $\conv(\Delta_1)$. We show that $\conv(\Delta_1)$ admits simple descriptions in several cases of interest such as when $\cZ$ is defined by a cardinality constraint or by weak or strong hierarchy constraints. In this setting, $\cl\conv(\cT)$ in the original space was first given in \citep[Theorem~1]{wei2022ideal}.
		In contrast to our result, \citet[Theorem~1]{wei2022ideal} provide ideal descriptions for $\cl\conv(\cT)$ that rely on explicit linear inequality description of the set $\conv(\cZ \backslash \set{\bm 0})$ and adding a new nonlinear convex constraint based on every facet of $\conv(\cZ \backslash \set{\bm 0})$. 
		Our extended formulation, however, can immediately take advantage of any relaxation of $\Delta_1$ and opens up ways to benefit from the long-line of research on convex hull descriptions of binary sets and related advanced techniques in optimization software. 
		More recently, for $\cZ = \set{0,1}^d$ and $\cI = \emptyset$, \citet[Proposition~3]{han2021compact} provide an extended formulation involving $2d$ new continuous variables, that are subsequently projected out in \citep[Proposition~4]{han2021compact} to recover \citep[Theorem~1]{wei2022ideal}. In contrast, our extended formulation works for any boolean set $\cZ \subseteq \set{0,1}^d$ and relies on only one additional binary variable. 
	
		\item In Section~\ref{sec:applications:rank1:nonnegative}, we continue to explore the non-separable setting of $\cT$ when $\cX = \{ \x \in \R^d: x_i \geq 0, \forall i \in \cI \}$ for some index set $\cI \subseteq [d]$ and $\cZ \subseteq \set{0,1}^d$. In Theorem~\ref{thm:nonnegative}, we establish an extended formulation for $\cl\conv(\cT)$ using $d$ new binary variables and $d$ new continuous variables. 
		Such an extended formulation for $\cl\conv(\cT)$ when $\cZ$ contains combinatorial constraints and  $\cX$ contains sign restrictions has not been provided in the literature before.
		When $\cZ= \set{0,1}^d$, we recover \citep[Proposition~3]{han2021compact}. Moreover, our result extends \citep{han2021compact} by allowing combinatorial constraints on binary variables, i.e., $\cZ\subset \set{0,1}^d$, and extended-valued functions. Note that the proof techniques from \cite{han2021compact} rely on explicit disjunctive programming arguments and the Fourier-Motzkin elimination method, and therefore, they cannot be easily adapted to handle combinatorial constraints on $\cZ$ (see Remark~\ref{rmk:han}).
	\end{itemize}
	
	\item Finally, in Section~\ref{sec:numerical}, we compare the numerical performance of formulations discussed in Section~\ref{sec:applications} on sparse nonnegative logistic regression with hierarchy constraints. We observe that our new relaxations are of high quality in terms of leading to both good quality continuous relaxation bounds and also significant improvements in the branch and bound performance. 

\end{itemize}

\paragraph{Notation.} 
We use $\overline \R$ to denote the extended real numbers. The indicator function $\mathds{1}_\cI(i) = 1$ if $i \in \cI$ and $=0$ otherwise. Given a positive integer $n$, we let $[n]:=\set{1,\ldots,n}$. 
We use boldface letters to denote vectors and matrices. 
We let $\bm 0$ and $\ones$ denote the vectors with all zeros and ones, respectively, while $\e_i$ denote the $i$th unit basis vector. Given a vector $\x \in \R^n$, we define $\sign(\x)$ as vector in $\R^n$ whose $i$th elements is $+1$ if $x_i > 0$, $-1$ if $x_i < 0$, and $0$ if $x_i = 0$.
For a set $\cS \subseteq \R^m$, we denote by $\rint(\cS), \rec(\cS), \cl(\cS), \conv(\cS), \cl\conv(\cS)$ its relative interior, recessive directions, closure, convex hull and closed convex hull, respectively. Given a boolean set $\cZ$, we denote its continuous relaxation by $\relx(\cZ)$, and for a boolean set $\Delta$ involving binary variables $\w$, the set $\relx_{\w}(\Delta)$ refers to partial continuous relaxation of $\Delta$ obtained by removing the integrality restriction on only the variables $\w$.

\section{Technical Tools}\label{sec:Preparatory}
In this section we build theoretical tools necessary for our study of separable and nonseparable sets.
In Section~\ref{sec:perspective}, we begin by examining perspective functions, recognizing that their epigraphs are convex cones, and relating their epigraphs back to the sets $\cH$ and $\cT$ of our interest as well as the conic mixed-binary~sets $\Sbar$. We then study the convex hull characterization of $\Sbar$ in Section~\ref{sec:convex_hull} and conclude with technical tools to handle a simple linking constraint in taking convex hulls and the closure operation in an extended~space.

\subsection{Perspective function and its closure}\label{sec:perspective} 

Perspective function plays an important role in our analysis.
For a proper lower semicontinuous and convex function $h: \R^d \to \overline \R$ with $h(\bm 0) = 0$, we define its \emph{perspective function} $h^+: \R^d \times \R_+ \to \overline \R$~as
\begin{align*}
	h^+(\x, w) := 
	\begin{cases}
		w \, h(\x / w), & \qquad \text{if } w > 0, \\
		0, & \qquad \text{if } w = 0 \text{~and~} \x = \bm 0, \\
		+\infty, & \qquad \text{otherwise.}
	\end{cases} 
\end{align*} 
The epigraph of $h^+$ is given by 
\[
\epi(h^+) := \set{ (t, \x, w) \!\in\! \R \times \R^d \times \R_+:~ t\ge h^+(\x, w) }.
\]
Note that our definition of perspective function $h^+$ is almost matching with the classical definition of the perspective function given in \citep{hiriart2004fundamentals,combettes2018perspective} as
\[	\widetilde{h}(\x, w) := 
\begin{cases}
	w \, h(\x / w), & \qquad \text{if } w > 0, \\
	+\infty, & \qquad \text{otherwise,}
\end{cases}
\]
where the main distinction between $h^+$ and $\widetilde{h}$ is that $h^+(\bm 0, 0)=0$ whereas $\widetilde{h}(\bm 0, 0)=+\infty$. 
We first establish that the epigraph of the perspective function $h^+$ is a convex cone under standard assumptions. 
\begin{lemma}
	\label{lem:epi:projective}
	Let $h : \R^d \!\to\! \overline \R$ be a proper lower semicontinuous and convex function with $h(\bm 0) = 0$. Then, $\epi(h^+)$ is a convex cone containing the origin.
\end{lemma} 
\begin{proof}
	Since $h$ is assumed to be proper, lower semicontinuous and convex, the perspective function $h^+$ is proper and convex. This follows from \citep[p.~35]{rockafellar1970convex}.
	Therefore, the set $\epi(h^+)$ is convex. Moreover, $(\bm 0, 0) \in \epi(h^+)$ as $h^+(\bm 0, 0) = 0$. 
	Additionally, $\epi(h^+)$ is indeed a cone, i.e., for any $\lambda > 0$ and $(t, \x, w) \in \epi(h^+)$, we have $(\lambda t, \lambda \x, \lambda w) \!\in\! \epi(h^+)$. This is because for any $w > 0$ and any $(t, \x, w) \in \epi(h^+)$, we have $w \, h(\x / w) \leq t$ or $\lambda w \, h(\lambda \x / (\lambda w)) \leq \lambda t$. We thus conclude that $\epi(h^+)$ is a cone containing the origin.
	\qed
\end{proof}

While $\epi(h^+)$ is a convex cone under standard assumptions, it is not closed. Therefore, we also study the closure of the perspective function. 
Recall that for a proper lower semicontinuous and convex function $h: \R^d \to \overline \R$ with $h(\bm 0) = 0$, the \emph{closure of the perspective function} $h^\pi: \R^d \times \R_+ \to \overline \R$ is defined~as 
\begin{align*}
	h^\pi(\x, w) := 
	\begin{cases}
		w \, h(\x / w), & \qquad \text{if } w > 0, \\
		\lim_{s \downarrow 0} s \, h(\x / s), & \qquad \text{if } w = 0,\\
		+\infty, & \qquad \text{otherwise.}
	\end{cases} 
\end{align*}
It is well known that the closure of $\widetilde{h}$, and consequently $h^+$, is given by $h^\pi$; see for example \citep[Proposition~2.2.2]{hiriart2004fundamentals}. 
The epigraph of $h^\pi$ is given by $\epi(h^\pi) \!=\! \{ (t, \x, w) \!\in\! \R \times \R^d \times \R_+ \!\!:\! t\ge h^\pi(\x, w) \}$. We analyze the cone generated by the perspective function and its closure in the next lemma. 

\begin{lemma}
	\label{lem:epi:perspective}
	Let $h : \R^d \to \overline \R$ be a proper lower semicontinuous and convex function such that $h(\bm 0) = 0$. Then, $\cl(\epi(h^+)) = \epi(h^\pi)$. If, additionally, $\epi(h)$ does not contain a line, $\epi(h^\pi)$ is pointed.
\end{lemma}

\begin{proof} 
	Note that the perspective function $h^+(\x, w)$ coincides with its closure $h^\pi(\x, w)$ for $w > 0$ and $(\bm 0, 0) \in \dom(h^\pi)$. Then, by lower semicontinuity of $h^\pi$ (see \citep[p.~37 and Theorem~13.3]{rockafellar1970convex}), we conclude $\cl(h^+) = h^\pi$.
	Therefore, we have $\cl(\epi(h^+)) = \epi(h^\pi)$. 
	
	To see that $\epi(h^\pi)$ is pointed, suppose both $(-\bar t, -\bar \x, -\bar w),(\bar t, \bar \x, \bar w) \!\in\! \epi(h^\pi)$. Then, we must have $\bar w=0$. The closure of the perspective function at $(\x, 0)$ satisfies
	\begin{align*}
		h^\pi(\x, 0) = \lim_{w \downarrow 0} w \, h(\x/w) = \sup \set{\g^\top \x: \g \in \dom(h^*)},
	\end{align*}
	where $h^*$ denotes the conjugate of $h$ and the last equality follows from \citep[p.~37 and Theorem~13.3]{rockafellar1970convex}. Thus, as both $(-\bar t, -\bar \x, 0),(\bar t, \bar \x, 0) \in \epi(h^\pi)$, we have 
	\begin{align*}
		\begin{cases}
			h^\pi(\bar \x, 0) = \sup \set{\g^\top \bar \x: \g \in \dom(h^*)} \leq \bar t \\
			h^\pi(-\bar \x, 0) = \sup \set{-\g^\top \bar \x: \g \in \dom(h^*)} \leq - \bar t,
		\end{cases}
	\end{align*}
	which enforces that $\g^\top \bar \x = \bar t$ for all $\g \in \dom(h^*)$. Hence, for any $\gamma \in \R$, the function $h$ satisfies
	\begin{align*}
		h(\gamma \bar \x) 
		= \sup_{\g \in \dom(h^*)} \gamma \g^\top \bar \x - h^*(\g) 
		= \sup_{\g \in \dom(h^*)} \gamma \bar t - h^*(\g) 
		= \gamma \bar t - h(\bm 0) 
		= \gamma \bar t,
	\end{align*}
	where the first equality holds because $h = h^{**}$ (as $h$ is a proper lower semicontinuous and convex function and thus~\citep[Theorem~12.2]{rockafellar1970convex} applies), the second equality follows from the observations that $\g^\top \bar \x = \bar t$, the third equality follows from the definition of biconjugate function and the relation $h(\bm 0) = h^{**}(\bm 0) = 0$. 
	If $(\bar t, \bar \x) \neq \bm 0$, we have $(\gamma \bar t, \gamma \bar \x) \in \epi(h)$ for any $\gamma \in \R$, which contradicts our assumption that $\epi(h)$ does not contain a line. Hence, we have shown that $(-\bar t, -\bar \x, -\bar w),(\bar t, \bar \x, \bar w) \in \epi(h^\pi)$ only if $(\bar t, \bar \x, \bar w) = \bm 0$, implying that $\epi(h^\pi)$ is pointed.
	This then completes the proof.
	\qed
\end{proof}

Lemma~\ref{lem:epi:perspective} extends \citep[Proposition 4]{ramachandra2021robust} to non-differentiable proper functions.
For univariate functions, the requirement that $\epi(h)$ does not contain a line means that $h$ is a nonlinear function.
We next recall that whenever $\epi(h)$ admits a conic representation so will $\epi(h^\pi)$ under a minor condition.
\begin{remark}
	Suppose that the function $h:\R^d \to \overline \R$ is lower semicontinuous and convex, and its epigraph admits the conic representation
	\begin{align*}
		\epi(h) = \set{ (t, \x) \in \R \times \R^d:~ \exists \u \text{~s.t.~} \bH_t t + \bH_x \x + \bH_u \u + \bH_0 \in \H}
	\end{align*}
	for some appropriate matrices $\bH_t, \bH_x, \bH_u$ and $\bH_0$, and a regular cone $\H$.
	Provided that $\bH_u \u \in \H$ implies that $\bH_u \u = \bm 0$, \citet[Proposition~2.3.2]{ben2021lectures} show that the epigraph of $h^\pi$ admits the conic representation 
	\begin{align*}
		\epi(h^\pi) \!=\! \set{ (t, \x, w) \in \R \times \R^d \times \R_+ \!:\! \exists \u \text{~s.t.~} \bH_t t + \bH_x \x + \bH_u \u + \bH_0 w \in \H}.
	\end{align*}
	While \citep{ben2021lectures} present this conic representation only for conic quadratic representable functions, the result and its proof immediately extend to regular cones $\H$ as discussed in \citep[Section~2.3.7]{ben2021lectures}.
	\qed
\end{remark}

We end this subsection by introducing the mixed-binary set 
\begin{align}
	\label{eq:W}
	\cW \!:=\! \set{(\gb, \bm \gamma, \gd) \!\in\! \R^q \times \R^p \times \Delta: 
		\begin{array}{l}
			h_i(\beta_{i,1}) \leq \gamma_i, ~ \forall i \in [p],  \\ [1ex]
			\beta_{i,1} (1 - \delta_i) = 0, ~ \bC_i \gb_i \in \C_i, \,\forall i \in [p]\!
		\end{array}
	}\!,
\end{align}
where $\Delta \!\subseteq\! \set{0,1}^n$, $h_i : \R \!\to\! \overline \R$ is a proper lower semicontinuous convex function with $h_i(0) = 0$, $\C_i$ is a closed convex cone, and $\beta_{i,1}$ denotes the first component of the subvector $\gb_i$.
The set $\cW$ will serve as the primary substructure for characterizing the closed convex hulls of the sets $\cH$ and $\cT$ in Section~\ref{sec:applications}. 
Note that by the definition of the perspective function, the set $\cW$ can be reformulated~as 
\begin{align*}
	\cW = \set{(\gb, \bm \gamma, \gd) \in \R^q \times \R^p \times \Delta:~ 
		\begin{array}{l}
			h^+_i(\beta_{i,1}, \delta_i) \leq \gamma_i, ~\forall i \in [p],  \\ [1ex]
			\bC_i \gb_i \in \C_i, ~\forall i \in [p]
		\end{array}
	}.
\end{align*}
We next show that $\cW$ is an instance of the conic mixed-binary set~\eqref{eq:S:bar}. 

\begin{lemma}
	\label{lem:separableW}
	Consider the mixed-binary set $\cW$ as defined in~\eqref{eq:W}. Then,  for all $i \in [p]$, by letting
	\begin{align}
	\label{eq:param}
	\ga_i := (\gb_i, \gamma_i), ~ \bA_i := 
	\begin{bmatrix}
		\bm 0^\top & 1 \\[0.3ex] \e_1^\top & 0 \\[0.3ex] \bm 0^\top & 0 \\[0.3ex] \bC_i & 0
	\end{bmatrix}, ~
	\bB_i :=
	\begin{bmatrix}
		0 \\ 0 \\ 1 \\ \bm 0
	\end{bmatrix}, ~
	\K_i := \epi(h_i^+) \times \C_i,
\end{align}
and $\K = \times_{i \in [p]}\, \K_i$, we arrive at $\cW =\Sbar$.
\end{lemma}

\begin{proof}
Note that for every $i \in [d]$ the constraint $h^+_i(\beta_{i,1}, \delta_i) \leq \gamma_i$ is simply $(\gamma_i, \beta_{i,1}, \delta_i) \in \epi(h^+_i)$. As $h_i$ is a proper lower semicontinuous convex function with $h_i(0)=0$, by Lemma~\ref{lem:epi:projective}, $\epi(h^+_i)$ is a convex cone containing the origin. Thus, letting $\ga_i := (\gb_i, \gamma_i)$ for all $i \in [p]$, we can express the requirements  $(\gamma_i, \beta_{i,1}, \delta_i) \in \epi(h^+)$ and $\bC_i \gb_i \in \C_i$ in the conic form $\bA_i \ga_i + \bB_i \delta_i \in \K_i$~with the data given in \eqref{eq:param}.
Hence, we have $\cW =\Sbar$ with $\K = \times_{i \in [p]}\, \K_i$.
\qed
\end{proof} 
In the next subsection, we first characterize the closed convex hull of $\Sbar$, and then provide a description for $\cl\conv(\cW)$ using its conic mixed-binary representation.

\subsection{Conic binary sets}\label{sec:convex_hull}

Our results in this subsection and also later on rely on  Carath\'{e}odory's theorem. Carath\'{e}odory's theorem states that if a point lies in the convex hull of a finite-dimensional set, it can be written as a convex combination of a finite number of points from the original set.
We first establish a description for the convex hull of $\Sbar$. 

\begin{proposition}
	\label{prop:conv}
	Consider the set $\Sbar$ defined as in~\eqref{eq:S:bar}, where $\Delta \subseteq \set{0,1}^n$, $\K = \times_{i \in [d]} \, \K_i$, and each $\K_i$ is a convex cone containing the origin for every $i \in [p]$. Then, we have $\conv(\Sbar) = \cS(\conv(\Delta), \K)$.
\end{proposition}

\begin{proof} 
	We will proceed by showing that $\conv(\Sbar) \subseteq \cS(\conv(\Delta), \K)$ and $\conv(\Sbar) \supseteq \cS(\conv(\Delta), \K)$. The first direction trivially holds as $\K$ is convex and $\conv(A \cap B) \subseteq \conv(A) \cap \conv(B)$. Therefore, we focus on establishing $\conv(\Sbar) \supseteq \cS(\conv(\Delta), \K)$.
	Consider $(\bar \ga, \bar \gd) \in \cS(\conv(\Delta), \K)$. Then, as $\bar \gd \in \conv(\Delta)$, by Carath\'{e}odory's theorem, we have $\bar \gd = \sum_{k \in [q]} \lambda_k \gd^k$ for some finite $q$, $\gd^k \in \Delta$, and $\lambda_k > 0$ with $\sum_{k \in [q]} \lambda_k = 1$. For each $k\in[q]$, we construct the vector $\ga^k$ with subvectors $\ga^k_i$ as follows 
	\begin{align*}
		\ga^k_i := 
		\left\{
		\begin{array}{ll}
			\bar \ga_i, & \text{~if~} \bar \delta_i = 0 \\ [1ex]
			{\bar \ga_i}/{\bar \delta_i}, & \text{~if~} \bar \delta_i \neq 0 \text{~and~} \delta^k_i = 1 \\ [1ex]
			\bm 0,  & \text{~if~} \bar \delta_i \neq 0 \text{~and~} \delta^k_i = 0
		\end{array} .
		\right.
	\end{align*}
	Next, we prove that $(\ga^k, \gd^k) \!\in\! \Sbar$ by showing that $\bA_i \ga^k_i + \bB_i \delta^k_i \!\in\! \K_i$ for all $k \in [q]$ and $i \in [p]$. 
	Consider any index $i\in[p]$. 
	Note that as $(\bar \ga, \bar \gd) \in \cS(\conv(\Delta), \K)$, we have $\bA_i \bar \ga_i + \bB_i \bar \delta_i \in \K_i$.
	If $\bar \delta_i = 0$, then $\bA_i \bar \ga_i \in \K_i$. Also, we have $\delta^k_i = 0$ for all $k \in [q]$ because $\bar \delta_i = \sum_{k \in [q]} \lambda_k \delta^k_i = 0$, $\lambda_k > 0$, and $\delta^k_i \in \set{0,1}$ for every $k \in [q]$. Thus, the point $(\ga^k_i, \delta^k_i)$ satisfies $\bA_i \ga^k_i + \bB_i \delta^k_i = \bA_i \bar \ga_i \in \K_i$ for every $k\in[q]$.
	If $\bar \delta_i \neq 0$, then $\bA_i (\bar \ga_i / \bar \delta_i) + \bB_i \in \K_i$ as $\K_i$ is a cone.
	Moreover, as $\bar \delta_i \neq 0$, there exists at least one index $\hat k \in [q]$ with $\delta^{\hat k}_i = 1$. 
	For any $k\in[q]$ such that $\delta^k_i = 1$, by definition $(\ga^k_i, \delta^k_i)=(\bar \ga_i / \bar \delta_i,1)$ and thus it satisfies $\bA_i \ga^k_i + \bB_i \delta^k_i = \bA_i (\bar \ga_i / \bar \delta_i) + \bB_i \in \K_i$.
	Finally, for any $k\in[q]$ such that $\delta^k_i = 0$, as $\bar \delta_i \neq 0$, by definition $(\ga^k_i, \delta^k_i)=(\bm 0,0)$ and thus it satisfies $\bA_i \ga^k_i + \bB_i \delta^k_i = \bm 0 \in \K_i$ as $\K_i$ contains the origin. All in all, we have shown that $\bA_i \ga^k_i + \bB_i \delta^k_i \in \K_i$ for all $k \in [q]$ and $i \in [p]$, which implies that $(\ga^k, \gd^k) \in \Sbar$ for all $k \in [q]$.
	
	We next prove that $(\bar \ga, \bar \gd) = \sum_{k \in [q]} \lambda_k (\ga^k, \gd^k)$. Recall that $\bar \gd = \sum_{k \in [q]} \lambda_k \gd^k$ where $\gd^k\in\Delta$ for all $k\in[q]$. Then, $\bar \delta_i = \sum_{k \in [q]} \lambda_k \delta^k_i = \sum_{k \in [q]: \delta^k_i = 1} \lambda_k$ holds for all $i\in[p]$. Moreover, for any $i \in [p]$ with $\bar \delta_i \neq 0$, from the definition of $\ga^k$, we deduce that 
	\begin{align*}
		\sum_{k \in [q]} \lambda_k \ga^k_i 
		= \sum_{k \in [q]: \delta^k_i = 1} \lambda_k \ga^k_i + \sum_{k \in [q]: \delta^k_i = 0} \lambda_k \ga^k_i 
		= \frac{\bar \ga_i}{\bar \delta_i} \sum_{k \in [q]: \delta^k_i = 1} \lambda_k = \bar \ga_i,
	\end{align*}
	where the last equality follows from $\bar \delta_i  = \sum_{k \in [q]: \delta^k_i = 1} \lambda_k$.
	Furthermore, for any $i \in [p]$ with $\bar \delta_i = 0$, by the definition of $\ga^k$ we have $\sum_{k \in [q]} \lambda_k \ga^k_i = \sum_{k \in [q]} \lambda_k \bar \ga_i = \bar \ga_i$. Thus, we conclude that  $(\bar \ga, \bar \gd) = \sum_{k \in [q]} \lambda_k (\ga^k, \gd^k)$, which shows $(\bar \ga, \bar \gd) \in \conv(\Sbar)$, as desired. 
	\qed
\end{proof}

The main result of this section depends on the following conditions.
\begin{assumption}
	\label{asmp:clconv}
	There exist a point $\gd \in \Delta$ with $\delta_i = 1$ for every $i \in [p]$. For any $(\ga_i, \delta_i) \in \R^{m_i} \times \R_{++}$, if $\bA_i \ga_i + \bB_i \delta_i \in \cl(\K_i)$, then $\bA_i \ga_i + \bB_i \delta_i \in \K_i$. Additionally, $\bB_i \in \K_i$ for every $i \in [p]$.
\end{assumption}

The first condition in Assumption~\ref{asmp:clconv} holds without loss of generality. Suppose that there exist an index $i \in [d]$ such that for any $\gd \in \Delta$, $\delta_i = 0$. In this case, the continuous subvector $\ga_i$ has no link with the binary value $\delta_i$. Hence, the convex hull of $\Sbar$ can be computed from the convex hull of a lower dimensional set obtained by eliminating the subvector $\ga_i$ from $\ga$ and the value $\delta_i$ from $\gd$. The elimination process continues until the first condition is met. In addition, the second condition implies that $\K_i$ and its closure coincides for any $\delta_i > 0$. All closed cones $\K_i$ and also all  cones that arise as the epigraph of the perspective function $h^+$ of any proper lower semicontinuous and convex function $h$ with $h(\bm 0) = 0$ satisfy this second condition; see the proof of Lemma~\ref{lem:epi:perspective}.

The next theorem gives the closed convex hull characterization of $\Sbar$. 

\begin{theorem}
	\label{thm:clconv}
	Under assumptions of Proposition~\ref{prop:conv} and Assumption~\ref{asmp:clconv}, we have 
	$$\cl\conv(\Sbar) = \cS(\conv(\Delta), \cl(\K)).$$
\end{theorem}

\begin{proof}
	We proceed by showing that $\cl\conv(\Sbar) \subseteq \cS(\conv(\Delta), \cl(\K))$ and $\cl\conv(\Sbar) \supseteq \cS(\conv(\Delta), \cl(\K))$. The first direction trivially holds as $\conv(\Delta)$ is a bounded polytope, $\K$ is a convex set, and $\cl\conv(A \cap B) \subseteq \cl\conv(A) \cap \cl\conv(B)$. Therefore, in the sequel we focus on $\cl\conv(\Sbar) \supseteq \cS(\conv(\Delta), \cl(\K))$.
	Consider any $(\bar \ga, \bar \gd) \in \cS(\conv(\Delta), \cl(\K))$. Then as $\bar \gd \in \conv(\Delta)$, by Carath\'{e}odory's theorem, we have $\bar \gd = \sum_{k \in [q]} \lambda_k \gd^k$ for some finite $q$, $\gd^k \in \Delta$, and $\lambda_k > 0$ with $\sum_{k \in [q]} \lambda_k = 1$. 
	Moreover, we also have $\bA_i \bar \ga_i + \bB_i \bar \delta_i \in \cl(\K_i)$. For all $i\in[p]$, we let $\hat \ga_i := \bar \ga_i \cdot \sign(\bar \delta_i)$ and $\tilde \ga_i := \bar \ga_i \cdot (1 - \sign(\bar \delta_i))$ so that $(\bar \ga_i, \bar \delta_i) = (\hat \ga_i, \bar \delta_i) + (\tilde \ga_i, 0)$, and we define the vectors $\hat \ga$ and $\tilde \ga$ based on the subvector $\hat \ga_i$ and $\tilde \ga_i$, respectively.
	
	We first prove that $(\hat \ga, \bar \gd) \in \conv(\Sbar)$. Recall that $\bar \gd \in \conv(\Delta)$. Consider any index $i \in [p]$. If $\bar \delta_i = 0$, then $\hat \ga_i = \bm 0$, which implies that $\bA_i \hat \ga_i + \bB_i \bar \delta_i = \bm 0 \in \K_i$ as $\K_i$ contains the origin. Moreover, if $\bar \delta_i \neq 0$, then $\hat \ga_i = \bar \ga_i$ and $\bA_i \hat \ga_i + \bB_i \bar \delta_i = \bA_i \bar \ga_i + \bB_i \bar \delta_i \in \cl(\K_i)$. Then, as $\bar \delta_i \neq 0$, by the second condition in Assumption~\ref{asmp:clconv}, we conclude that $\bA_i \hat \ga_i + \bB_i \bar \delta_i \in \K_i$.
	Hence, $\bA_i \hat \ga_i + \bB_i \bar \delta_i \in \K_i$ for all $i \in [p]$, and we have $(\hat \ga, \bar \gd) \in \cS(\conv(\Delta), \K)$. Then, the claim follows from Proposition~\ref{prop:conv} as $\cS(\conv(\Delta), \K) = \conv(\Sbar)$.
	
	We next show that $(\tilde \ga, \bm 0) \in \cS(\conv(\Delta), \cl(\K))$. 
	Consider any index $i \in [p]$. 
	Recall that $\tilde \ga_i = \bar \ga_i \cdot (1 - \sign(\bar \delta_i))$. 
	If $\bar \delta_i \neq 0$, then $\tilde \a_i = \bm 0$ and we have $\bA_i \tilde \ga_i = \bm 0 \in \cl(\K_i) $ as $\K_i$ contains the origin. Additionally, if $\bar \delta_i = 0$, then $\tilde \ga_i = \bar \ga_i$ and we have $\bA_i \tilde \ga_i = \bA_i \bar \ga_i + \bB_i \bar \delta_i \in \cl(\K_i)$ (as $(\bar \ga, \bar \gd) \in \cS(\conv(\Delta), \cl(\K))$ and $(\bar \ga, \bar \gd) = (\hat \ga, \gd) + (\tilde \ga, \bm 0)$). Thus, these two observations together imply that $(\tilde \ga, \bm 0) \in \cS(\conv(\Delta), \cl(\K))$.
	
	Next, we write $(\tilde \ga, \bm 0)$ as the sum of limit points in $\Sbar$. By the first condition in Assumption~\ref{asmp:clconv}, for any $j \in [d]$, there exists a vector $\tilde \gd^j \in \Delta$ with $\tilde \delta^j_j = 1$. Using these binary vectors and introducing the vectors $\tilde \ga^j, j \in [d],$ with subvectors $\tilde \ga^j_j := \tilde \ga_j$ and $\tilde \ga^j_i := \bm 0$ for $i \neq j$, we have 
	\begin{align*}
		\textstyle (\tilde \ga, \bm 0) = \lim_{\varepsilon \downarrow 0} \sum_{i \in [d]} \varepsilon (\tilde \ga^j / \varepsilon, \tilde \gd^j).
	\end{align*}
	Note that the points $(\tilde \ga^j / \varepsilon, \tilde \gd^j) \in \Sbar$ for any $\varepsilon > 0$ and $j \in [d]$. First consider any $j\in[d]$ such that $\bar \delta_j \neq 0$. Then, from $\tilde \ga_j = \bar \ga_j \cdot (1 - \sign(\bar \delta_j))$ we deduce that $\tilde \ga^j_j = \tilde \ga_j= \bm 0$. As by definition $\tilde \ga^j_i =\bm 0$ for all $i\neq j$, we conclude that for any $j\in[d]$ such that $\bar \delta_j \neq 0$ we must have $\tilde \ga^j =\bm 0$. 
	Then, for all $i \in [p]$, $\bA_i (\tilde \ga^j_i / \varepsilon) + \bB_i \tilde \delta^j_i = \bB_i \tilde \delta^j_i \in \K_i$ where the last relation follows from $\bB_i\in\K_i$ (see Assumption~\ref{asmp:clconv}). Thus, we conclude $(\tilde \ga^j / \varepsilon, \tilde \gd^j) \in \Sbar$ for all $j\in[d]$ such that $\bar \delta_j \neq 0$. 
	Now consider any $j\in[d]$ such that $\bar \delta_j = 0$. Then, by definition we have $\tilde \ga^j_j = \bar \ga_j$ and $\ga^j_i = \bm 0$ for $i \neq j$. 
	Hence, we have $\bA_j \tilde \ga_j^j = \bA_j \bar \ga_j + \bB_j \bar \delta_j \in \cl(\K_j)$ (as $(\bar \ga, \bar \gd) \in \cS(\conv(\Delta), \cl(\K))$) and $\bB_j \varepsilon \in \cl(\K_j)$ (as  $\bB_j \in \K_j$ by Assumption~\ref{asmp:clconv}). Since $\cl(\K_j)$ is a convex cone, we deduce that $\bA_j \tilde \ga_j^j + \bB_j \varepsilon \in \cl(\K_j)$. Since $\varepsilon>0$ and thus by applying the second condition in Assumption~\ref{asmp:clconv} we conclude $\bA_i (\tilde \ga^j_j / \varepsilon) + \bB_j \tilde \delta^j_j \in \K_j$. 
	Moreover, for $i \neq j$, we have $\bA_i (\tilde \ga^j_i / \varepsilon) + \bB_i \tilde \delta^j_i = \bB_i \tilde \delta^j_i\in \K_i$ where the last relation follows from $\bB_i\in\K_i$ (see Assumption~\ref{asmp:clconv}). Altogether, these show that $(\tilde \ga^j / \varepsilon, \tilde \gd^j) \in \Sbar$ for all $j\in[d]$.

	Using these relations, we may thus write
	\begin{align*}
		\textstyle
		(\bar \ga, \bar \gd) 
		= \lim_{\varepsilon \downarrow 0} ~ \sum_{k \in [q]} (1 - \varepsilon d) (\hat \ga, \bar \gd) + \sum_{i \in [d]} \varepsilon (\tilde \ga^j / \varepsilon, \tilde \gd^j).
	\end{align*}
	Then, $(\bar \ga, \bar \gd)$ is written as the limit of convex combinations of points from $\Sbar$. Thus, $(\bar \ga, \bar \gd) \in \cl\conv(\Sbar)$, as desired. 
	\qed
\end{proof}

In Section~\ref{sec:applications} we will see that the set $\cW$ plays a crucial role in characterizing the closed convex hulls of $\cH$ and $\cT$. Thus, we next present the characterization of $\cl\conv(\cW)$ using its conic representation discussed in Lemma~\ref{lem:separableW} and Theorem~\ref{thm:clconv}.

\begin{proposition}
	\label{prop:separable}
	Consider the mixed-binary set $\cW$ as defined in~\eqref{eq:W}.
	If for any $i \in [p]$ there exists a point $\gd \in \Delta$ with $\delta_i = 1$, then
	\begin{align*}
		\cl\conv(\cW) = \set{(\gb, \bm \gamma, \gd) \in \R^q \times \R^p \times \conv(\Delta) : 
			\begin{array}{l}
				h_i^\pi(\beta_{i,1}, \delta_i) \leq \gamma_i, ~\forall i \in [p], \\ [1ex]
				\bC_i \gb_i \in \C_i, ~\forall i \in [p]
			\end{array}
		}.
	\end{align*}
\end{proposition}
\begin{proof}
	Recall that Lemma~\ref{lem:separableW} states that $\cW = \Sbar$ with the parameters specified in~\eqref{eq:param}.
	By Lemma~\ref{lem:epi:projective}, each $\epi(h^+_i)$ is a convex cone containing the origin. Hence, $\K_i := \epi(h_i^+) \times \C_i$ is also a convex cone containing the origin (as $\C_i$ is a closed convex cone as well by assumption).
	Therefore, the requirements of Proposition~\ref{prop:conv} are met. In the following we verify the conditions in Assumption~\ref{asmp:clconv}. 
	The first condition holds as we have assumed that for any $i \in [p]$ there exists a $\gd \in \Delta$ with $\delta_i = 1$. 
	For all $i \in [p]$, by Lemma~\ref{lem:epi:perspective}, we have $\cl(\K_i) = \epi(h_i^\pi) \times \C_i$ as $\C_i$ is assumed to be closed. Thus, the second condition of Assumption~\ref{asmp:clconv} easily follows from the definitions of the perspective function and its closure.
	Finally, notice that $h^+_i(0, 1) = h(0) = 0$, implying that $(0,0,1) \in \epi(h^+_i)$. Additionally, $\bm 0 \in \C_i$ as $\C_i$ is a closed convex cone. Hence, $\bB_i \in \K_i$ for all $i \in [p]$, and the last condition of Assumption~\ref{asmp:clconv} also holds. 
	Therefore, from Theorem~\ref{thm:clconv} we deduce that $\cl\conv(\cW)$ is obtained by replacing the set $\Delta$ and the cones $\K_i$ with $\conv(\Delta)$ and $\cl(\K_i)$, respectively. Finally, as $\K_i = \epi(h_i^+) \times \C_i$,  from Lemma~\ref{lem:epi:perspective} we deduce that $\cl(\K_i) = \epi(h_i^\pi) \times \C_i$, which then concludes the~proof.
	\qed
\end{proof}

We finally conclude by introducing a lemma, which illustrates how to characterize the convex hull of a set that is obtained by adding a simple linking constraint to the direct product of sets and a sufficient condition for taking the closure of a set in an extended formulation form. This technical result allows us to study separable structures such as the one arising in $\cH$ in a simplified manner as well as taking the closure of convex hulls given in a lifted space.
The proof of this lemma is presented in Appendix~\ref{appendix:A}.

\begin{lemma}
	\label{lem:tools}
	The following holds:
	\begin{enumerate}[label=(\roman*)]
		\item \label{lem:separable} Let $\cV = \{ (\tau, \gm) : \exists \t \text{ s.t. } (\gm, \t) \in \cU, \ones^\top \t = \tau \}$ for some set $\cU$. Then 
		$$\conv(\cV) = \{ (\tau, \gm) : \exists \t \text{~s.t.~} (\gm, \t) \in \conv(\cU), \ones^\top \t = \tau \}.$$
		
		\item \label{lem:cl} Let $\cV \!=\! \{ \gm : \exists \bm \eta \in \cU \text{ s.t. } \bA \bm \eta = \gm, ~ \bB \bm \eta = \bm b \}$ for some non-empty convex set $\cU$, matrices $\bA$ and $\bB$, and vector $\bm b$. Suppose that there exists a point $\bm \eta^\star \in \rint(\cU)$ satisfying the condition $\bB \bm \eta^\star = \bm b$. If additionally $\bm \eta \!=\! \bm 0$ is the only vector in the set $\rec(\{\bm \eta \in \cl(\cU): \bA \bm \eta = \bm 0, \bB \bm \eta = \bm b\})$, then 
		$$ \cl(\cV) = \{ \gm : \exists \bm \eta \in \cl(\cU) \text{ s.t. } \bA \bm \eta = \gm, \bB \bm \eta = \bm b \}. $$
	\end{enumerate}
\end{lemma}

\section{Applications}\label{sec:applications}
In this section we analyze the convex hull of the sets $\cH$ and $\cT$ defined in \eqref{eq:H} and \eqref{eq:T}, respectively.
We assume that all univariate functions vanish at zero. This in fact holds without loss of any generality if zero is in the domain of the functions. In this case, we can always define a new function by subtracting the constant term $h(0)$ from~$h$. 

We will employ a simple yet powerful proof strategy to characterize the closed convex hull of the sets $\cH$ and $\cT$. In the first step, we identify the recessive directions of the given set, and study the associated sets that are augmented by adding these recessive directions. In particular, such augmented sets admit extended formulations with new binary variables and complementary restrictions that model the recessive directions. In the second step, we demonstrate that the original complementary restrictions can be eliminated from these new sets in the extended space. 
Finally, using perspective functions, we transform these new sets into the form of the set $\cW$ and leverage Proposition~\ref{prop:separable} along with Lemma~\ref{lem:tools}.

\subsection{Separable functions} \label{sec:applications:separable}
As a warm up for the subsequent sections, we start by analyzing the separable function case in this subsection. 
Recall the mixed-binary set
\begin{align*}
	\cH = \set{(\tau, \x, \z) \in \R \times \cX \times \cZ :~ 
	\begin{array}{l}
		\sum_{i \in [d]} h_i(x_i) \leq \tau, \\[1ex]
		x_i (1 - z_i) = 0, ~ \forall i \in [d]
	\end{array}
	}.
\end{align*}  
This type of set arises as substructure in a number of applications such as Markowitz portfolio selection~\citep{frangioni2006perspective}, network design~\citep{gunluk2010perspective}, sparse learning \citep{xie2020scalable,bertsimas2021sparse}, and low-rank regression~\citep{bertsimas2021new}. The result of this section relies on the following assumption.
\begin{assumption}
	\label{asmp:separable}
	The set $\cX = \{\x \in \R^d :~ x_i \geq 0, ~ \forall i \in \cI\}$. For any $i \in [d]$, there exists a point $\z \in \cZ \subseteq \set{0,1}^d$ such that $z_i = 1$. The function $h_i: \R \to \overline \R$ is proper, lower semicontinuous and convex with $h_i(0) = 0$ for all $i \in [d]$.
\end{assumption}

\begin{theorem}
	\label{thm:separable}
	Under Assumption~\ref{asmp:separable}, the set $\cH$ defined in~\eqref{eq:H} satisfies
	\begin{align*}
		\textstyle
		\cl\conv(\cH) = \set{ (\tau, \x, \z) \in \R \times \cX \times \conv(\cZ):~ \sum_{i \in [d]} h^\pi_i(x_i, z_i) \leq \tau}.
	\end{align*}
\end{theorem}

Theorem~\ref{thm:separable} extends \citep[Thereom~3]{wei2022ideal} by allowing sign-constrained continuous variables and proper functions. 
\citet{wei2022ideal} prove the convex hull result by showing that the support function of $\cH$ and the set presented in Theorem~\ref{thm:separable} coincide when $\cI = \emptyset$ and $h$ is real-valued.
In contrast, our proof is constructive and make use of the hidden conic structure introduced by the perspective function. As a by product, it allows us to easily include nonnegativity constraints on the set $\cX$. 

We demonstrate the importance of the requirement imposed on the binary set $\cZ$ in Assumption~\ref{asmp:separable} with an example. If $\cZ = \set{\bm 0}$, then the set $\cH$ simplifies to $\cH = \set{(\tau, \bm 0, \bm 0): \tau \geq 0}$. Thus, $\cl\conv(\cH) = \cH$. In contrast, Theorem~\ref{thm:separable} with no restriction on $\cZ$ would suggest that the set $\{(\tau, \x, \bm 0): \sum_{i \in [d]} h_i^\pi(x_i, 0) \leq \tau \}$ gives the closed convex hull of $\cH$, which may not be correct in this simple case. The requirement on the binary set $\cZ$, however, holds without loss of generality.
If there exist an index $i \in [d]$ such that for any $\z \in \cZ$, $z_i = 0$, then $x_i = 0$ due to the logical constraint $x_i(1 - z_i) = 0$. Thus, we can eliminate $x_i$ and $z_i$ from the set $\cH$ and compute $\cl\conv(\cH)$ from a lower dimensional set.

\begin{proof} \!\emph{of Theorem~\ref{thm:separable}~}
	We first introduce the auxiliary mixed-binary set
	\begin{align*}
		\overline \cH :=
		\set{ (\x, \z, \t) \in \cX \times \cZ \times \R^d : ~
			\begin{array}{l}
				h_i(x_i) \leq t_i, ~ \forall i \in [d], \\ [1ex]
				x_i (1 - z_i) = 0, ~ \forall i \in [d]
			\end{array}
		}.
	\end{align*}
	By letting $\gb_i := (x_i), \bm \gamma := \t, \gd := \z, \Delta := \cZ, \bC_i = 1, \C_i = \R_+$, the set $\overline \cH$ can be represented as an instance of the set $\cW$ defined as in~\eqref{eq:W}. Hence, applying Proposition~\ref{prop:separable} yields
	\begin{align*}
		\cl\conv(\overline \cH) =
		\set{ (\x, \z, \t) \in \cX \times \cZ \times \R^d : h_i^\pi(x_i, z_i) \leq t_i, ~ \forall i \in [d]
		}.
	\end{align*}
	In the following we characterize $\cl\conv(\cH)$ in terms of $\cl\conv(\overline \cH)$. Notice that $\cH \!=\! \{ (\tau, \x, \z) \!:\! \exists \t \text{ s.t. } (\x, \z, \t) \in \overline \cH, \ones^\top \t = \tau \}$.
	Therefore, applying Lemma~\ref{lem:tools}\,\ref{lem:separable} yields $\conv(\cH) \!=\! \{ (\tau, \x, \z): \exists \t \text{ s.t. } (\x, \z, \t) \in \conv(\overline \cH), \, \ones^\top \t = \tau \}$. By letting 
	\begin{align*}
		\begin{array}{c}
			\gm := (\tau, \x, \z), ~ \bm \eta := (\tau, \x, \z, \t), ~ \cU := \R \times \conv(\overline \cH), \\ [1ex]
			\bA := [\bI_{2d+1}, ~ \bm 0], ~ \bB := (-1, \bm 0, \bm 0, \ones)^\top, ~ \bm b := 0,
		\end{array}
	\end{align*}
	we can apply Lemma~\ref{lem:tools}\,\ref{lem:cl} to conclude that
	\begin{align}
		\label{eq:clconv:H}
		\cl\conv(\cH) = \{ (\tau, \x, \z): \exists \t \text{ s.t. } (\tau, \x, \z, \t) \in \R \times \cl\conv(\overline \cH), \, \ones^\top \t = \tau \}.
	\end{align}  
	This holds because the first requirement of Lemma~\ref{lem:tools}\,\ref{lem:cl} is trivially satisfied as the variable $\tau$ is free to choose from the set $\cU$. In addition, $\bA \bm \eta = \bm 0$ if $\bm \eta \in \{(0, \bm 0, \bm 0, t) : \ones^\top \t = 0 \} = \set{\bm 0}$, where the equality holds because, by definition of $\cU$, $h_i^\pi(0, 0) = 0 \leq t_i$ enforcing $\t$ to be the vector of all zeros. 
	The proof concludes by using the relation~\eqref{eq:clconv:H} and then applying Fourier-Motzkin elimination~\citep{dantzig1973fourier} to project out $\t$.
	\qed
\end{proof}

We conclude this section by a remark for the case of \emph{totally unimodular} binary sets. The set $\cZ = \{ \z \in \{0, 1\}^d: \bF \z \leq \f \}$ is totally unimodular if the matrix $\bF$ is totally unimodular and the vector $\f$ is integer-valued. Recall that every square submatrix of a totally unimodular matrix has determinant $0$, $+1$ or $-1$. Examples of totally unimodular sets include cardinality constraint set, in which $\cZ = \{ \z \in \{0,1\}^d: \ones^\top \z \leq \kappa \}$ for some $\kappa \in [d]$, weak hierarchy set, where $\cZ = \{ \z \in \{0,1\}^d: z_d \leq \sum_{i \in [d-1]} z_i \}$, and strong hierarchy set, where $\cZ = \{ \z \in \{0,1\}^d: z_d \leq z_i, ~ \forall i \in [d-1] \}$.

\begin{remark}\label{rem:disjoint:TU}
	Suppose that the set $\cZ = \{ \z \in \{0, 1\}^d: \bF \z \leq \f \}$ is totally unimodular. Then, under Assumptions~\ref{asmp:separable}, we have
	\begin{align*}
		\textstyle
		\conv(\cH) = \set{ (\tau, \x, \z) \in \R \times \cX \times [0, 1]^d:~ \sum_{i \in [d]} h^\pi_i(x_i, z_i) \leq \tau, ~ \bF \z \leq \f}.
	\end{align*}
\end{remark}

\subsection{Rank-1 functions with $\cX = \R^d$} 
\label{sec:applications:rank1:nosign}
Recall the mixed-binary set
\begin{align*}
	\cT = \set{ (\tau, \x, \z) \in \R \times \cX \times \cZ:~ h (\a^\top \x ) \leq \tau, ~ x_i (1 - z_i) = 0, ~ \forall i \in [d] },
\end{align*} 
where we assume that the vector $\a \in \R^d$ satisfies $a_i \neq 0$ for all $i \in [d]$. 
In this section we will consider the case that $\cX = \R^d$.
This type of set appears as a substructure in sparse regression~\citep{bertsimas2016best,xie2020scalable} and sparse classification~\citep{bertsimas2021sparse,wei2022ideal}.

For a given set $\cZ \subseteq \set{0,1}^d$, we construct a graph $G_{\cZ} = (V, E)$, where $V = [d]$ denotes its nodes, $E$ denotes its edge, and $\{i, j\} \in E$ if and only if $i \neq j$ and there exits a vector $\z \in \cZ$ with $z_i = z_j = 1$.
The structure of $\cZ$, as represented by its associated graph $G_{\cZ}$, plays an important role in the description of $\conv(\cT)$.

\subsubsection{Connected graph}
\label{sec:applications:rank1:connected}
Our result relies on the following assumption.

\begin{assumption}
	\label{asmp:connected}
	The set $\cX = \R^d$. The set $\cZ \subseteq \set{0,1}^d$ satisfies $\cZ \neq \set{ \bm 0 }$. The graph $G_{\cZ}$ associated with $\cZ$ is connected. The vector $\a \in \R^d$ satisfies $a_i \neq 0$ for all $i \in [d]$. The function $h: \R \to \overline \R$ is proper, lower semicontinuous, and convex with $h(0) = 0$.
\end{assumption} 

\begin{theorem}
	\label{thm:connected}
	Under Assumption~\ref{asmp:connected}, the set $\cT$ defined in~\eqref{eq:T} satisfies 
	\begin{align*}
		\cl\conv(\cT) = \set{ (\tau, \x, \z) \in \R \times \R^d \times \R^d:~ \exists w \in \R \text{~s.t.~}~
		\begin{array}{l}
			h^\pi(\a^\top \x, w) \leq \tau, \\ [1ex]
			(w, \z) \in \conv(\Delta_1)
		\end{array} 
		},
	\end{align*}
	where $\Delta_1 = \{(w, \z) \in \set{0,1} \times \cZ:~ w \leq \ones^\top \z\}$.
\end{theorem}

Theorem~\ref{thm:connected} provides a compact extended formulation for $\cl\conv(\cT)$ by introducing a \emph{single} binary variable $w$ and  finitely many \emph{linear} constraints describing $\conv(\Delta_1)$. The variable $w$ embeds the entire complexity of describing $\cl\conv(\cT)$ into the complexity of characterizing $\conv(\Delta_1)$, a polyhedral set.
Another advantage of Theorem~\ref{thm:connected} is that it does not require a complete description of $\conv(\Delta_1)$ to be known in advance when solving the mixed-binary problem with a solver; one can simply provide the binary formulation for $\Delta_1$ to the solver instead of $\conv(\Delta_1)$. Moreover, as the set $\Delta_1$ is defined by linear inequalities and today's commercial optimization solvers have advanced features to dynamically generate effective cuts for binary (or mixed-integer) sets defined by linear inequalities, solvers easily and automatically generate strong cuts that approximate the convex hulls of such sets as $\Delta_1$. This is in particular crucial when $\conv(\Delta_1)$ is difficult to characterize completely. In such cases, our formulation enables the users to rely on the optimization software to generate effective cuts on the fly.

In contrast to our compact extended formulation, \citet[Thereom~1]{wei2022ideal} give an ideal formulation in the original space of variables for $\cl\conv(\cT)$. This ideal formulation, however, relies on an explicit formulation for $\conv(\cZ \backslash \set{\bm 0})$. Specifically, \citep[Proposition~1]{wei2022ideal} shows that
\begin{align*}
	\conv(\cZ \backslash \set{0}) = \conv(\cZ) \cap \set{\z \in \set{0,1}^d:~ \bm \gamma^\top \z \geq 1, ~ \forall \bm \gamma \in \cF}
\end{align*}
for some finite set $\cF \subset \R^d$. Using this observation, \citet[Theorem~1]{wei2022ideal} proves that 
\begin{align}
	\label{eq:wei}
	\cl\conv(\cT) = \set{ (\tau, \x, \z) \in \R \times \R^d \times \R^d: 
	\begin{array}{l}
		h(\a^\top \x) \leq \tau, ~\z \in \conv(\cZ), \\ [1ex]
		h^\pi(\a^\top \x, \bm \gamma^\top \z) \leq \tau, ~ \forall \bm \gamma \in \cF
	\end{array} 
	}.
\end{align}  
This description of $\cl\conv(\cT)$ requires access to $\cF$, i.e., the explicit inequality description of $\conv(\cZ \backslash \set{0})$ which may not be easy to attain. 
Furthermore, the set $\cF$ can be very complex with possibly exponentially many (in the dimension $d$) inequalities and then the set in~\eqref{eq:wei} will have exponentially many \emph{nonlinear} convex  inequalities. Nonlinear convex constraints as opposed to linear ones are often more expensive to handle by the optimization solvers, and this makes the ideal formulation given in~\eqref{eq:wei} quite impractical from a computational standpoint whenever $|\cF|$ is large. Furthermore, to the best of our knowledge, the cut generation strategies for the nonlinear constraints as opposed to the linear ones are rather limited in today's optimization solvers. Thus, the nonlinear inequalities included in formulation~\eqref{eq:wei} would have a rather limited effect on the solvers in terms of generating further cuts on the fly beyond what the user specifies.

To prove Theorem~\ref{thm:connected}, we first make an observation about the recessive direction of $\cT$. 
\begin{lemma}
	\label{lem:rec:T:connected}
	Under Assumption~\ref{asmp:connected}, if $(\tau, \x, \z) \!\in\! \cl\conv (\cT)$, then for any $\bar \x \!\in\! \R^n$ satisfying $\a^\top \bar \x = 0$ we have $(\tau, \x + \bar \x, \z) \in \cl\conv(\cT)$.
\end{lemma}
\begin{proof} 
	As $G_{\cZ}=(V,E)$ is a connected graph, there exists a path of length $k \geq d$ visiting all nodes in the graph. Let $i_1, \dots, i_{k+1}$ be the nodes we visit in such a path. Then, there is an edge between node $i_l$ to node $i_{l+1}$, and hence, by definition of $G_{\cZ}$, there exists a binary vector $\z^l \in \cZ$ such that $z^l_{i_l} = z^l_{i_{l+1}} = 1$ for any $l \in [k]$.
	Additionally, let $\cM_0 := \emptyset$ and $\cM_j$ be the set of nodes we visit after $j-1$ steps in a such path for any $j \in [k+1]$. Note that each $\cM_j$ is a set containing only unique values from the set $[d]$. For instance, $\cM_1 = \set{i_1}$ and $\cM_{k+1} = [d]$ as we visit all nodes after $k$ steps.
	
	Based on $\bar \x$, we construct the vectors $\x^l, l \in [k],$ as follows
	\begin{align*}
		\x^l := \frac{\sum_{j \in \cM_{l}} a_j \bar x_j}{a_{i_l}} \e_{i_l} - \frac{\sum_{j \in \cM_{l}} a_j \bar x_j}{a_{i_{l+1}}} \e_{i_{l+1}},
	\end{align*}
	where $\e_i \in \R^d$ denotes the $i$th unit basis of $\R^d$.
	By construction, for every $l \in [k]$, $\x^l$ satisfies $\a^\top \x^l = 0$, and the support of $\x^l$ is covered by the binary vector $\z^l \in \cZ$. Note also that 
	\begin{align*}
		\sum_{l \in [k]} \x^l 
		&= \sum_{l \in [k]} \left( \frac{\sum_{j \in \cM_{l}} a_j \bar x_j}{a_{i_l}} \e_{i_l} - \frac{\sum_{j \in \cM_{l}} a_j \bar x_j}{a_{i_{l+1}}} \e_{i_{l+1}} \right) \\
		&=\sum_{l \in [k]} \left( \frac{\sum_{j \in \cM_{l}} a_j \bar x_j}{a_{i_l}} \e_{i_l} - \frac{\sum_{j \in \cM_{l}} a_j \bar x_j}{a_{i_{l+1}}} \e_{i_{l+1}} \right) - \frac{\sum_{j \in \cM_{0}} a_j \bar x_j}{a_{i_1}} \e_{i_1} \\
		& \qquad + \frac{\sum_{j \in \cM_{k+1}} a_j \bar x_j}{a_{i_{k+1}}} \e_{i_{k+1}} \\
		&= \sum_{l \in [k+1]} \frac{\sum_{j \in \cM_{l}} a_j \bar x_j - \sum_{j \in \cM_{l-1}} a_j \bar x_j}{a_{i_l}} \e_{i_l}
		= \sum_{i \in [d]} \bar x_i \e_i = \bar \x
	\end{align*}
	where the first equality follows from the construction of $\x^l$, the second equality holds because $\sum_{j \in \cM_{0}} a_j \bar x_j =  0$ (as $\cM_{0} = \emptyset$) and $\sum_{j \in \cM_{k+1}} a_j \bar x_j = \a^\top \bar \x = 0$ (as $\cM_{k+1} = [d]$). The third equality holds by rearranging the summation, while the final equality follows from the fact that for any $l\in[k+1]$ the expression $\sum_{j \in \cM_{l}} a_j \bar x_j - \sum_{j \in \cM_{l-1}} a_j \bar x_j = a_{i_l} \bar x_{i_l}$ if we visit node $i_l$ for the first time after $l-1$ step and $=0$ otherwise, and the fact that the path we consider visits all nodes. Hence, we conclude $\sum_{l \in [k]} \x^l = \bar \x$.
	
	As $h(\a^\top \x^l)=h(0)=0$, and the support of $\lambda \x^l$ and $\z^l$ are the same, we conclude that $(0, \lambda \x^l, \z^l) \in \cT$ for every $\lambda \in \R$ and $l \in [k]$. Since $\cl\conv(\cT)$ is both closed and convex, we have
	\begin{align*}
		(\tau, \x + \bar \x, \z) = \lim_{\lambda \to +\infty} \left[\frac{\lambda -k}{\lambda} (\tau, \x, \z) + \sum_{l \in [k]} \frac{1}{\lambda}(0, \lambda \x^l, \z^l)\right] &\in \cl\conv(\cT),
	\end{align*} 
	where the equality holds as $\sum_{l \in [k]} \x^l = \bar \x$. Thus, the claim follows.
	\qed
\end{proof}

Inspired by the set $\cT$, we then introduce the mixed-binary set 
\begin{align}
	\label{eq:T:bar}
	\overline \cT := \set{(\tau, \x, \z, s, w) \!\in\! \R \!\times\! \R^d \!\times\! \R^d \!\times\! \R \!\times\! \R: 
		\begin{array}{l}
			h(s) \leq \tau, ~ \a^\top \x = s, \\ [1ex]
			s (1-w) = 0, ~ (w, \z) \in \Delta_1 
		\end{array}
	},
\end{align}
where $\Delta_1$ is defined as in Theorem~\ref{thm:connected}.
Note that $\cT$ admits the representation
\begin{align*}
	\cT = \set{(\tau, \x, \z): \exists s,w \text{~s.t~} (\tau, \x, \z, s, w) \in \overline \cT, ~ x_i(1-z_i) = 0,~ \forall i \in [d] }.
\end{align*}
However, we establish that $\cl\conv(\cT)$ can be obtained solely from $\overline \cT$ using the result of Lemma~\ref{lem:rec:T:connected}. In particular, all individual complementary relations between the continuous variables $x_i$ and the binary variables $z_i$ can be dropped in the description of $\cl\conv(\cT)$. This results in a considerably simpler representation involving only a \emph{single} complementary relation (between the continuous variable $s$ and the binary variable $w$) given in $\overline{\cT}$. In fact, we will see that the variable $w$ essentially models the logical constraint $\z = \bm 0 \implies \a^\top \x = 0$.

\begin{proposition}
	\label{prop:connected}
	Under Assumption~\ref{asmp:connected}, we have
	\begin{align*}
		\cl\conv(\cT) = \cl\conv \left(\set{(\tau, \x, \z): \exists s,  w \text{~s.t.~} (\tau, \x, \z, s, w) \in \overline \cT
		} \right),
	\end{align*}
	where $\overline \cT$ is as defined in~\eqref{eq:T:bar}.
\end{proposition}

\begin{proof}
	By Lemma~\ref{lem:rec:T:connected}, the set
	\begin{align*}
		\cR := \set{(\bar \tau, \bar \x, \bar \z) \in \R \times \R^d \times \R^d:~ \bar \tau = 0, ~ \bar \z = \bm 0, ~ \a^\top \bar \x = 0}
	\end{align*}
	is contained in $\rec(\cT)$. Thus, $\cl\conv(\cT) = \cl\conv(\cT + \cR)$.
	Define the set
	\begin{align*}
		\cT' := \set{(\tau, \x, \z)\in \R \times \R^d \times \cZ:~ h(\a^\top \x) \leq \tau, ~ \z = \bm 0 \implies \a^\top \x = 0}.
	\end{align*}
	We first prove that $\cT' = \cT + \cR$ by showing that $\cT + \cR \subseteq \cT'$ and $\cT + \cR \supseteq \cT'$.
	
	$\bm{(\subseteq)}$ This is immediate as $\cT \subseteq \cT'$ and by definition of $\cT'$ for any $\bar \x$ such that $\a^\top \bar \x = 0$ we have $(0, \bar \x, \bm 0)$ is a recessive direction in $\cT'$.
	
	$\bm{(\supseteq)}$ Consider any $(\tau, \x, \z) \in \cT'$. Then, $h(\a^\top \x) \leq \tau$ and $\z \in \cZ$.
	If $\z = \bm 0$, then as $(\tau, \x, \z) \in \cT'$ we have $\a^\top \x = 0$ which implies $h(\a^\top \x) = h(0) = 0 \leq \tau$. Then, $(\tau, \x, \z) = (\tau, \bm 0, \z) + (0, \x, \bm 0) \in \cT + \cR$ (as $(\tau, \bm 0, \z) \in \cT$ always holds and also $\a^\top \x = 0$ implies $(0, \x, \bm 0) \in \cR$). 
	Else, there exists $i\in[d]$ such that  $z_i = 1$. Define $\bar \x := \x - \frac{\a^\top \x}{a_i} \e_i$. Then,  $\a^\top \bar \x = 0$ implying $\a^\top (\x - \bar \x)= \a^\top \x$ and $h(\a^\top(\x - \bar \x)) = h(\a^\top \x) \leq \tau$. Moreover, since $\x - \bar \x = \frac{\a^\top \x}{a_i} \e_i$, we have $(x_i - \bar x_i) (1 - z_i) = 0$ satisfied for all $i \in [d]$. 
	Therefore, $(\tau, \x - \bar \x, \z) \in \cT$ from which we conclude that $(\tau, \x, \z) = (\tau, \x - \bar \x, \z) + (0, \bar \x, \bm 0) \in \cT + \cR$ (as $\a^\top \bar \x = 0$ implies $(0, \bar\x, \bm 0) \in \cR$).
	
	Thus, we showed that $\cT' = \cT + \cR$. 
	Now, note that $(\tau, \x, \z) \in \cT'$ if and only if there exists $w \in \set{0,1}$ such that $w \leq \ones^\top \z$ and $(\a^\top \x) (1 - w) = 0$. This easily holds because there exists $w \in \set{0, 1}$ satisfying $(\a^\top \x) (1 - w) = 0$ and $w \leq \ones^\top \z$ if and only if $\z = 0 \implies \a^\top \x = 0$.
	Therefore, $(\tau, \x, \z) \in \cT'$ if and only if there exists $w \in \set{0,1}$ and $s \in \R$ such that $(\tau, \x, s, \z, w) \in \overline \cT$.  Put it differently, we have $\cT' = \Proj_{\tau, \x, \z} (\overline \cT)$, which implies that
	\begin{align*}
		\cl\conv(\cT) 
		= \cl\conv(\cT + \cR) 
		= \cl\conv(\cT')
		= \cl\conv(\Proj_{\tau, \x, \z} (\overline \cT)).
	\end{align*}
	Hence, the claim follows.
	\qed
\end{proof}

Based on Lemma~\ref{lem:epi:perspective} and Proposition~\ref{prop:connected}, we are ready to prove Theorem~\ref{thm:connected}.

\begin{proof} \!\emph{of Theorem~\ref{thm:connected}~}
	Recall the definition of $\overline \cT$ and $\Delta_1$. By letting 
	\begin{align*}
		\gb_i := (s, \x), ~ \bm \gamma := \tau, ~ \gd := (w, \z), ~ \Delta := \Delta_1, ~ \bC_i = [-1, \a^\top], ~ \C_i = \set{0},
	\end{align*}
	we can represent the set $\overline \cT$ as an instance of the set $\cW$ defined as in~\eqref{eq:W}. Then, Proposition~\ref{prop:separable} yields
	\begin{align*}
		\cl\conv(\overline \cT) \!=\!
		\set{ (\tau, \x, \z, s, w) \!\in\! \R \!\times\! \R^d \!\times\! \R^d \!\times\! \R \!\times\! \R \!: 
		\begin{array}{l}
			h^\pi(s, w) \leq \tau, ~ \a^\top \x = s, \\ [1ex]
			(w, \z) \in \conv(\Delta_1)
		\end{array}
		}.
	\end{align*}
	In the following we characterize $\cl\conv(\cT)$ in terms of $\cl\conv(\overline \cT)$. From Proposition~\ref{prop:connected}, we deduce
	\begin{align*}
		\cl\conv(\cT) 
		= \cl \big(\conv(\Proj_{\tau, \x, \z} (\overline \cT))\big)
		= \cl \big(\Proj_{\tau, \x, \z}(\conv (\overline \cT))\big).
	\end{align*}
	By letting 
	\begin{align*}
		\begin{array}{c}
			\gm := (\tau, \x, \z), ~ \bm \eta := (\tau, \x, \z, s, w), ~ \cU := \conv(\overline \cT), \\ [1ex]
			\bA := [\bI_{2d+1}, ~ \bm 0], ~ \bB := \bm 0^\top, ~ \bm b := 0,
		\end{array}
	\end{align*}
	we observe that $\Proj_{\tau, \x, \z}(\conv (\overline \cT)) = \{ \gm : \exists \bm \eta \in \cU \text{ s.t. } \bA \bm \eta = \gm, ~ \bB \bm \eta = \bm b \}$ as in Lemma~\ref{lem:tools}\,\ref{lem:cl}. 
	Note also that the first requirement of Lemma~\ref{lem:tools}\,\ref{lem:cl} is trivially satisfied as the matrix $\bB$ equals zero. In addition, $\bm \eta \in \cl(\cU)$ satisfies $\bA \bm \eta = \bm 0$ if $\bm \eta \in \{(0, \bm 0, \bm 0, s, w) : s = \a^\top\x=0, (w, \bm 0) \in \conv(\Delta_1) \} = \set{\bm 0}$, where the equality holds because, by definition of $\Delta_1$ we have $0 \leq w \leq \ones^\top \z$ which enforces $w=0$ as $\z=\bm 0$. 
	Thus, we can apply Lemma~\ref{lem:tools}\,\ref{lem:cl} to conclude that
	\begin{align}
		\label{eq:clconv:T}
		\cl\conv(\cT) = \{ (\tau, \x, \z): \exists s, w \text{ s.t. } (\tau, \x, \z, s, w) \in \cl\conv(\overline \cT) \}.
	\end{align}  
	This completes the proof.
	\qed
\end{proof}

It is important to note that the description of $\conv(\Delta_1)$ may not be readily available from $\conv(\cZ)$ even if $\cZ$ is an integral or a totally unimodular set. 
\begin{example}
Consider the set $\cZ = \{ \z \in \{0,1\}^2 : z_1 = z_2 \}$, which is integral and totally unimodular. By definition, the resulting $\Delta_1$ is given by $\Delta_1 = \set{ (w,\z) \in \{0, 1\}^{3} : w \leq z_1+z_2,~ z_1 = z_2}$. Furthermore, $\conv(\Delta_1)=\{(w,\z)\in[0,1]^{3}: w \leq z_1,~ z_1 = z_2\}$. We thus observe that the continuous relaxation of $\Delta_1$ and its convex hull $\conv(\Delta_1)$ are different.
For example, the point $z_1 = z_2 = 0.5$ and $w = 1$ is in the continuous relaxation of $\Delta_1$ but it is not in $\conv(\Delta_1)$. 
\qed
\end{example}

In the sequel we examine the description of $\conv(\Delta_1)$ for some simple integral sets $\cZ$ of interest. The proofs of these results are provided in Appendix~\ref{appendix:A}.
We start with the case when $\cZ$ is defined by a cardinality constraint, which leads to an immediate totally unimodular representation of $\Delta_1$.

\begin{lemma}
	\label{lem:conv:Z:cardinality}
	Suppose $\cZ=\{ \z \in \{0,1\}^d :~\ones^\top \z \leq \kappa \}$ for some $\kappa \in [d]$. Then 
	\begin{align*}
		\conv(\Delta_1) = \set{(w,\z)\in[0,1]^{1+d}:~w\leq \ones^\top \z, ~\ones^\top \z \leq \kappa}.
	\end{align*}
\end{lemma} 
We next study the case of weak hierarchy constraints; in this case, $\Delta_1$ also admits a totally unimodular representation. 

\begin{lemma}
	\label{lem:conv:Z:weak:hierarchy}
	Suppose $\cZ=\{\z\in\{0,1\}^d:~z_d\leq\sum_{i\in[d-1]} z_i\}$. Then
	\begin{align*}
		\textstyle
		\conv(\Delta_1)= \set{(w,\z) \in[0,1]^{1+d} :~w \leq \sum_{i\in[d-1]} z_i, ~z_d \leq \sum_{i\in[d-1]} z_i}.
	\end{align*}
\end{lemma} 
For general $\cZ$, in the same spirit of \citep{wei2022ideal}, we can take advantage of the fact that the set $\conv(\cZ \backslash \set{\bm 0})$ is a polytope to give an explicit description of $\conv(\Delta_1)$ based on the description of $\conv(\cZ \backslash \set{\bm 0})$. In particular, suppose that $\conv(\cZ \backslash \set{\bm 0})$ admits the following representation
\begin{align}
	\label{eq:Z:0}
	\conv(\cZ\backslash\set{\bm 0}) = \set{ \z \in \R^d : \bF^0 \z \geq \bm 0, ~ \begin{array}{l} \z^\top \bm f^+_k \geq 1, ~ \forall k \in \cK, \\ [1ex] \z^\top \bm f_l^- \leq 1, ~ \forall l \in \cL
	\end{array} }.
\end{align}
Note that the representation of $\conv(\cZ\backslash\set{\bm 0})$ in \eqref{eq:Z:0} is without loss of generality as we can always scale each inequality to have right hand side value in $\{-1,0,1\}$.
\begin{lemma}
	\label{lem:conv:Delta:1}
	Given the representation $\conv(\cZ\backslash\set{\bm 0})$  in~\eqref{eq:Z:0}, we have
	\begin{align*}
		\conv(\Delta_1)= \set{ (w, \z) \in \R^{1+d} : \bF^0 \z \geq \bm 0, ~ \begin{array}{l} \z^\top \bm f^+_k \geq w, ~ \forall k \in \cK, \\ [1ex] \z^\top \bm f_l^- \leq 1, ~ \forall l \in \cL, \\ [1ex] \z^\top \bm f_l^- \leq \z^\top \bm f^+_k, ~ \forall k \in \cK, \forall l \in \cL
		\end{array} }.
	\end{align*}
\end{lemma}
We conclude this section by examining the case of strong hierarchy constraints. Unlike the previous two cases, the set $\Delta_1$ does not immediately admit a totally unimodular representation. Therefore, we utilize Lemma~\ref{lem:conv:Delta:1} to characterize $\conv(\Delta_1)$ for strong hierarchy constraints.

\begin{lemma}
	\label{lem:conv:Z:strong:hierarchy}
	Suppose $\cZ = \{\z \in \{0,1\}^d:~z_d \leq z_i,\,\forall i\in[d-1] \}$. Then,
	\begin{align*}
		\conv(\Delta_1) = \set{(w,\z) \in [0,1]^{1+d}:~
		\begin{array}{l}
		w \leq \sum_{i \in [d-1]} z_i - (d-2) z_d, \\[1ex]
		z_d \leq z_i, ~ \forall i\in[d-1]
		\end{array}}.
	\end{align*}
\end{lemma}

\subsubsection{General graph} 
\label{sec:applications:rank1:general}
We now examine a general graph $G_{\cZ}$ partitioned into $p \in [d]$ connected subgraphs. Without loss of generality, we assume that the subvector $\z_i$ is associated with the variables in the $i$th partition.
Such indexing allows us to simplify the evaluation of the rank-1 function as
\begin{align}
	\label{eq:indexing}
	h(\a^\top \x) = \sum_{i \in [p]} h(\a_i^\top \x_i)
\end{align}
because it is not possible to have two indices $j, k \in [d]$ with $z_j = z_k = 1$ from two different subgraphs. Our result relies on the following assumption.

\begin{assumption}
	\label{asmp:general}
	The set $\cX = \R^d$. The set $\cZ \subseteq \set{0,1}^d$. The graph $G_{\cZ}$ associated with $\cZ$ is partitioned into $p \in [d]$ connected subgraphs, while the subvector $\z_i$ corresponds to the variables in the $i$th partition. For any partition $i \in [p]$, there exists a subvector $\z_i \neq \bm 0$. The vector $\a \in \R^d$ satisfies $a_i \neq 0$ for all $i \in [d]$. The function $h: \R \to \overline \R$ is proper, lower semicontinuous, and convex with $h(0) = 0$.
\end{assumption} 

\begin{theorem}
	\label{thm:general}
	Under Assumption~\ref{asmp:general}, the set $\cT$ defined in~\eqref{eq:T} satisfies 
	\begin{align*}
		\cl\conv(\cT) \!=\! \set{ (\tau, \x, \z) \!\in\! \R \!\times\! \R^d \!\times\! \R^d \!:\! \exists \w \!\in\! \R^p \text{\,s.t.\,}
			\begin{array}{l}
				\sum_{i \in [p]} h^\pi(\a_i^\top \x_i, w_i) \!\leq\! \tau \\ [1ex]
				(\w, \z) \in \conv(\Delta_p)
			\end{array} 
		},
	\end{align*}
	where $\Delta_p = \set{(\w, \z) \in \set{0,1}^p \times \cZ: w_i \leq \ones^\top \z_i, ~ \forall i \in [p]}$.
\end{theorem}

\citet[Thereom~2]{wei2022ideal} give an extended formulation for $\conv(\cT)$ that involves $2p(d+1)$ additional variables. Their formulation relies on having a description of $\conv(\cZ_i)$, where $ \cZ_i := \set{\z \in \cZ: \z_j = \bm 0, ~\forall j \neq i}$, as the system of linear inequalities $\bC_i \z \leq \bm c_i$ for every $i \in [p]$. In particular, they suppose that for every $i\in[p]$ we have access to a finite set $\cF_i \subset \R^d$ satisfying 
\begin{align*}
	\conv(\cZ_i \backslash \set{0}) = \conv(\cZ_i) \cap \set{\z \in \set{0,1}^d:~ \bm \gamma_i^\top \z \geq 1, ~ \forall \bm \gamma_i \in \cF_i},
\end{align*}
and based on this \citet[Theorem~2]{wei2022ideal} provide $\cl\conv(\cT)$ as the set
\begin{align*}
	\Big\{ &(\tau, \x, \z) \in \R \times \R^d \times [0,1]^d : \exists \tilde \x \in \R^{dp}, \tilde \z \in [0,1]^{dp}, \t \in \R^p, \w \in [0,1]^p \text{~s.t.~} \\
	& \textstyle \sum_{i \in [p]} \tilde \x^i = \x, \, \sum_{i \in [p]} \tilde \z^i = \z, \, \ones^\top \t = \tau, \, \ones^\top \w = 1, \, \bC_i \tilde \z^i \leq \bm c_i w_i, ~\forall i \in [p] \\
	& h^\pi(\a^\top \tilde \x^i, w_i) \leq t_i, \, h^\pi(\a^\top \tilde \x^i, \bm \gamma_i^\top \z) \leq \tau, ~ \forall \bm \gamma_i \in \cF_i, ~\forall i \in [p] 
	\Big\}.
\end{align*}
In contrast, Theorem~\ref{thm:general} provides a compact extended formulation for $\cl\conv(\cT)$ by introducing a single binary vector in $\w \in \R^p$ and more importantly replacing the complexity of having explicit descriptions of convex hulls of several sets such as $\conv(\cZ_i)$ and $\conv(\cZ_i \backslash \set{0})$ with the complexity of a single set $\Delta_p$ that is obtained from $\cZ$ by adding $p$ additional binary variables and $p$ additional linear constraints.

Inspired by the set $\cT$, we introduce the mixed-binary set 
\begin{align}
	\label{eq:T:tilde}
	\widetilde \cT := \set{(\x, \z, \s, \w, \t) \!\in\! \R^{d+d+p+p+p} : 
		\begin{array}{l}
			h(s_i) \leq t_i, ~ \a_i^\top \x_i = s_i, ~ \forall i \in [p] \\ [1ex]
			s_i (1 - w_i) = 0, ~ \forall i \in [p] \\ [1ex]
			(\w, \z) \in \Delta_p 
		\end{array}
	}.
\end{align}
Theorem~\ref{thm:general} relies on the following auxiliary results, whose proofs are omitted for brevity because they follow the same path as those in Section~\ref{sec:applications:rank1:connected}. Additionally, the proof of Theorem~\ref{thm:general} is relegated to Appendix~\ref{appendix:A}.
 
\begin{lemma}
	\label{lem:rec:T:general}
	Under Assumption~\ref{asmp:general}, if $(\tau, \x, \z) \!\in\! \cl\conv (\cT)$, then for any $\bar \x \!\in\! \R^n$ satisfying $\a_i^\top \bar \x_i = 0$ for all $i \in [p]$, we have $(\tau, \x + \bar \x, \z) \in \cl\conv(\cT)$.
\end{lemma}

\begin{proposition}
	\label{prop:general}
	Under Assumption~\ref{asmp:general}, we have
	\begin{align*}
		\cl\conv(\cT) = \cl\conv \left(\set{(\tau, \x, \z): \exists \s,  \w, \t \text{~s.t.~} (\x, \z, \s, \w. \t) \in \widetilde \cT, ~ \ones^\top \t = \tau
		} \right),
	\end{align*}
	where $\widetilde \cT$ is as defined in~\eqref{eq:T:tilde}.
\end{proposition}
We provide a characterization for $\Delta_p$ in Lemma~\ref{lem:conv:Delta:p} in Appendix~\ref{appendix:A}.

\subsection{Rank-1 functions with sign-constrained continuous variables} \label{sec:applications:rank1:nonnegative}
Recall the mixed-binary set
\begin{align*}
	\cT = \set{ (\tau, \x, \z) \in \R \times \cX \times \cZ:~ h (\a^\top \x ) \leq \tau, ~ x_i (1 - z_i) = 0, ~ \forall i \in [d] }.
\end{align*} 
In this section we consider the case where $\cX = \{\x \in \R^d : x_i \geq 0, \forall i \in \cI \}$ for some $\cI \subseteq [d]$ and $\cZ\subseteq\{0,1\}^d$.  
This type of set appears as a substructure in fixed-charge network problems~\citep{wolsey1989submodularity}, smooth signal estimation~\citep{han2021compact}, outlier detection~\citep{gomez2021outlier}, and nonnegative least squares regression. 
Our result relies on the following assumption.

\begin{assumption}
	\label{asmp:nonnegative}
	The set $\cX = \{\x \in \R^d : x_i \geq 0, \, \forall i \in \cI \}$ for some set $\cI \subseteq [d]$. For any $i,j \in [d]$, there exists a point $\z \in \cZ \subseteq \set{0,1}^d$ such that $z_i = z_j = 1$. The vector $\a \in \R^d$ satisfies $a_i \neq 0$ for all $i \in [d]$. The function $h: \R \to \overline \R$ is nonlinear, proper, lower semicontinuous, and convex with $h(0) = 0$.
\end{assumption} 

Examples of boolean sets $\cZ$ that satisfy Assumption~\ref{asmp:nonnegative} include cardinality constraint set with parameter $\kappa \in [d]\backslash\{1\}$, weak hierarchy set, and strong hierarchy set.
Note that when the parameter of the cardinality constraint is set $\kappa = 1$, we have $\cZ = \{\z \in \set{0,1}^d: \ones^\top \z \leq 1 \}$. This case is simple as $\cZ$ is totally unimodular, and the logical constraint $x_i (1 - z_i) = 0, \forall i \in [d]$ along with $\ones^\top \z \leq 1$ enforces that at most one $x_i$ can be nonzero. Since $h(0) = 0$, we infact have $h(\a^\top \x) = \sum_{i \in [d]} h(a_i x_i)$. Thus, Remark~\ref{rem:disjoint:TU} is applicable and we arrive at
\begin{align*}
	\conv(\cT) = \set{ (\tau, \x, \z) \in \R \times \cX \times [0, 1]^d:~ \sum_{i \in [d]} h^\pi(a_i x_i, z_i) \leq \tau, ~ \sum_{i \in [d]} z_i \leq 1 }.
\end{align*}
In the following we consider more challenging binary sets $\cZ$.

\begin{theorem}
	\label{thm:nonnegative}
	Under Assumption~\ref{asmp:nonnegative}, the set $\cT$ defined in~\eqref{eq:T} satisfies
	\begin{align*}
		\cl\conv(\cT) \!=\! \set{ (\tau, \x, \z) \!\in\! \R \!\times\! \cX \!\times\! \R^d \!:~ \exists \s, \w \text{~s.t.~}
			\begin{array}{l}
				\sum_{i \in [d]} h^\pi(a_i s_i, w_i) \leq \tau, \\ [1ex]
				0 \leq s_i \leq x_i, ~ \forall i \in \cI, \\ [1ex]
				\a^\top \s = \a^\top \x, \\ [1ex] 
				(\w, \z) \in \conv(\Omega) 
		\end{array}},
	\end{align*}
	where $\Omega := \{(\w, \z) \in \set{0,1}^d \times \cZ:~ \ones^\top \w \leq 1, ~\w \leq \z\}$.
\end{theorem}

\begin{remark}
	\label{rmk:han}
	When $\cZ = \set{0,1}^d$ and $h$ is real-valued, \citet[Proposition~3]{han2021compact} give an extended formulation for $\cl\conv(\cT)$. Their proof involves two steps. The first step is based on \citep[Theorem 1]{han2021compact}, which employs a support function argument and a disjunctive programming method to obtain a lifted description of $\cl\conv(\cT)$ with $2d^2 + 6d + 4$ additional variables. The second step employs the Fourier-Motzkin elimination method to reduce the number of additional variables to $2d$.
	Extending this proof technique to include combinatorial constraints on $\cZ$, however, is not straightforward (if possible at all) as both steps heavily rely on the properties of the unconstrained set $\set{0,1}^d$. For example, even if $\cI = \emptyset$, the support function argument requires the underlying graph of $\cZ$ to be connected, as shown in \citep[Theorem 1]{wei2022ideal}. In contrast, Theorem~\ref{thm:nonnegative} introduces $2d$ additional variables $\w,\s$ and leverages the relation between the original  binary variables $\z$ and the newly introduced the binary variables $\w$ through the set $\Omega$. As a result,  Theorem~\ref{thm:nonnegative} reduces the complexity of characterizing $\cl\conv(\cT)$ to the complexity of characterizing $\conv(\Omega)$.
	\qed
\end{remark}

To prove Theorem~\ref{thm:connected}, we first make an observation about the recessive direction of $\cT$. 

\begin{lemma}
	\label{lem:rec:T:nonnegative}
	Under Assumption~\ref{asmp:nonnegative}, if $(\tau, \x, \z) \in \cl\conv (\cT)$, then for any $\bar \x \in \cX$ satisfying $\a^\top \bar \x = 0$ we have $(\tau, \x + \bar \x, \z) \in \cl\conv(\cT)$.
\end{lemma}
\begin{proof}
	Based on $\bar \x$, we introduce the index sets $\cK_{>} = \{ k \in [d] : a_k \bar x_k > 0 \}$ and $\cK_{<} = \{ k \in [d] : a_k \bar x_k < 0 \}$. If $\cK_{>} = \cK_{<} = \emptyset$, then $\bar \x = \bm 0$ and the claim trivially holds. In the following we assume that $\cK_{>}$ and $\cK_{<}$ are both~nonempty (as $\a^\top \bar \x = 0$ either one of $\cK_{>}$ and $\cK_{<}$ being nonempty implies the other is also nonempty).
	
	For every $k \in \cK_{>}$ and $l \in \cK_{<}$, we construct the vector 
	\begin{align*}
		\x^{kl} := \left( \frac{a_l \bar x_l \bar x_k}{\sum_{j \in \cK_{<}} a_j \bar x_j} \right) \e_k + \left( \frac{a_k \bar x_k \bar x_l}{\sum_{q \in \cK_{>}} a_q \bar x_q} \right) \e_l.
	\end{align*}
	By construction, $\x^{kl}$ satisfies 
	\begin{align*}
		\textstyle \a^\top \x^{kl} = (a_k a_l \bar x_l \bar x_k)/(\sum_{j \in \cK_{<}} a_j \bar x_j) + (a_k a_l \bar x_l \bar x_k)/(\sum_{q \in \cK_{<}} a_q \bar x_q) = 0
	\end{align*}
	since $\sum_{j \in \cK_{<}} a_j \bar x_j = - \sum_{q \in \cK_{>}} a_q \bar x_q$ as $\a^\top \bar \x = 0$. We also have $\x^{kl} \in \cX$ as $\sign(x^{kl}_k) = \sign(\bar x_k)$ and $\sign(x^{kl}_l) = \sign(\bar x_l)$ and $\bar\x\in\cX$. Furthermore,  by Assumption~\ref{asmp:nonnegative}, there exits a binary vector $\z^{kl} \in \cZ$ such that $\z^{kl}_k = \z^{kl}_l = 1$. Thus, the support of $\x^{kl}$ is covered by the binary vector $\z^{kl}$.
	Finally, note~that 
	\begin{align*}
		\sum_{k \in \cK_{>}} \sum_{l \in \cK_{<}} \x^{kl} 
		&= \sum_{k \in \cK_{>}} \sum_{l \in \cK_{<}} \left( \frac{a_l \bar x_l \bar x_k}{\sum_{j \in \cK_{<}} a_j \bar x_j} \right) \e_k + \!\! \sum_{k \in \cK_{>}} \sum_{l \in \cK_{<}} \left( \frac{a_k \bar x_k \bar x_l}{\sum_{q \in \cK_{>}} a_q \bar x_q} \right) \e_l \\
		&=  \sum_{k \in \cK_{>}} \bar x_k \e_k + \sum_{l \in \cK_{<}} \bar x_l \e_l = \sum_{i \in [d]} \bar x_i \e_i = \bar \x.
	\end{align*}
	Since $h(\a^\top \x^{kl}) = h(0) = 0$ and the support of $\x^{kl}$ is covered by $\z^{kl}$, we conclude $(0, \lambda \x^{kl}, \z^{kl}) \in \cT$ for every $\lambda \in \R$, $k \in \cK_{>}$, and $l \in \cK_{<}$.
	Define $\kappa := | \cK_{>} | \cdot | \cK_{<} |$. Then, as $\cl\conv(\cT)$ is both closed and convex, we have
	\begin{align*}
		(\tau, \x + \bar \x, \z) \!=\! \lim_{\lambda \to +\infty} \! \left[\frac{\lambda - \kappa }{\lambda} (\tau, \x, \z) + \sum_{k \in \cK_{>}} \sum_{l \in \cK_{<}} \frac{1}{\lambda}(0, \lambda \x^{kl}, \z^{kl})\right] \!\!\in\! \cl\conv(\cT).
	\end{align*} 
	Hence, the claim follows.
	\qed
\end{proof}

We next introduce the auxiliary mixed-binary set
\begin{align*}
	\overline \cT_s := \set{(\tau, \x, \z): \exists \s, \w, \t \text{~s.t.~} (\x, \s, \z, \w, \t) \in \widehat \cT, \a^\top \x = \a^\top \s, ~\ones^\top \t = \tau},
\end{align*}
where 
\begin{align}
	\label{eq:T:hat}
	\widehat \cT \!\!:=\!\! \set{ \! (\x, \z, \s, \w, \t) \!\in\! \cX \!\times\! \R^d \!\times\! \cX \!\times\! \R^d \!\times\! \R^d \!:\!
		\begin{array}{l}
			h(a_i s_i) \leq t_i ,  \, \forall i \in [d], \\ [1ex] 
			s_i (1 - w_i) = 0, \, \forall i \in [d], \\ [1ex]
			(\w, \z) \!\in\! \Omega, ~ s_i \leq x_i, \, \forall i \!\in\! \cI
		\end{array} \!
	},
\end{align}
and $\Omega$ is as defined in Theorem~\ref{thm:nonnegative}. We first establish that the closed convex hull of $\cT$ can be obtained from the set $\overline \cT_s$.
\begin{proposition}
	\label{prop:nonnegative}
	Under Assumption~\ref{asmp:nonnegative}, we have $\cl\conv(\cT) = \cl\conv(\overline \cT_s)$.
\end{proposition}

\begin{proof}
	From Lemma~\ref{lem:rec:T:nonnegative}, we deduce that the set 
	\begin{align*}
		\cR := \set{(\bar \tau, \bar \x, \bar \z) \in \R \times \cX \times \R^m: \bar \tau = 0, ~ \bar \z = 0, ~ \a^\top \bar \x = 0}
	\end{align*}
	is contained in $\rec(\cT)$. Thus, $\cl\conv(\cT) = \cl\conv(\cT + \cR)$.
	We prove the proposition by showing that $\cl\conv(\cT) \subseteq \cl\conv(\overline \cT_s)$ and  $\cT + \cR \supseteq \overline \cT_s$ (which immediately implies $\cl\conv(\cT + \cR) \supseteq  \cl\conv(\overline \cT_s)$). 
	
	$\bm{(\subseteq)}$ Let $(\tau, \x, \z) \in \cT$. If $\a^\top \x = 0$, then by setting $\s = \w = \t = \bm 0$ we see that $(\tau, \x, \z) \in \overline \cT_s$. Thus, we next focus on the case of $\a^\top \x \neq 0$. In this case, $(\tau, \x, \z)$ may not necessarily be in $\overline \cT_s$. Nonetheless, we will show that $(\tau, \x, \z) \in \conv(\overline \cT_s)$. To this end, we introduce the set 
	$\cK_{>} := \{k \in [d]: (a_k x_k) (\a^\top \x) > 0 \}$, and the vector $\bm \delta := \sum_{k \in \cK_{>}} x_k \e_k$. We next define the points
	\begin{align*}
		(\hat \tau^k, \hat \x^k, \hat \z^k) := \left(\tau, \x - \bm \delta + \frac{\a^\top \bm \delta}{a_k} \e_k , \z \right), \quad \forall k \in \cK_{>}.
	\end{align*}
	Note here that by definition of $\cK_{>}$, we have $a_k\neq 0$ and $x_k\neq 0$ for all $k\in\cK_{>}$. 
	By construction, we have $\hat \x^k \circ (1 - \hat \z^k) = 0$ as $(\tau, \x, \z) \in \cT$, and also
	\begin{align*}
		\a^\top \hat \x^k = \a^\top \x - \a^\top \bm \delta + \frac{\a^\top \bm \delta}{a_k} \a^\top \e_k = \a^\top \x.
	\end{align*}
	Thus, we have $h(\a^\top \hat\x^k) = h(\a^\top \x) \leq \tau=\hat \tau^k$ and $(\hat \tau^k, \hat \x^k, \hat \z^k) \!\in\! \cT$ for all $k \in \cK_{>}$. 
	Moreover, we claim that $(\hat \tau^k, \hat \x^k, \hat \z^k) \in \overline \cT_s$ for all $k \in \cK_{>}$.  
	To see this, define $\hat \s^k := ((\a^\top \hat \x^k) / a_k) \e_k$, $\t^k := (h(\a^\top \hat \x^k)) \e_k $ and $\hat \w^k := \e_k$. Recall that $k\in\cK_{>}$ and so $x_k\neq0$ and by definition $\hat\z^k=\z$. Then, as $(\tau, \x, \z) \in \cT$ we have $x_k(1-z_k)=0$ which implies $z_k=1$. Thus, we deduce $\hat \w^k\leq \hat \z^k$. Then, by construction of $\hat \s^k, \hat \w^k, \hat \t^k$, we immediately conclude that the following conditions  
	\begin{align*}
		(\hat \x^k, \hat \s^k, \hat \z^k, \hat \w^k, \hat \t^k) \in \widehat \cT, ~ \a^\top \hat \x^k = \a^\top \hat \s^k, ~\ones^\top \hat \t^k \leq \hat \tau^k
	\end{align*}
	are satisfied.
	Furthermore, we have $(\tau, \x, \z) = \sum_{k \in \cK_{>}} \lambda_k (\hat \tau_k, \hat \x_k, \hat z_k),$ where $\lambda_k = (a_k x_k) / (\a^\top \bm \delta)$ (as we assumed that $\a^\top \x \neq 0$, the set $\cK_>$ is nonempty and $\a^\top \gd \neq 0$, recall also the definition of $\cK_>$ implies $(a_k x_k) / (\a^\top \bm \delta)>0$ for all $k$) with $\lambda_k > 0$ and $\sum_{k \in \cK_>} \lambda_k = 1$. 
	Thus, the point $(\tau, \x, \z) \in \conv(\overline \cT_s)$, which establishes $\cT\subseteq \conv(\overline \cT_s)$, as desired. 
	Since $\conv(\cT)$ is the smallest convex set containing $\cT$, we conclude $\conv(\cT) \subseteq \conv(\overline \cT_s)$. Taking closure of both side proves that $\cl\conv(\cT) \subseteq \cl\conv(\overline \cT_s)$.
	
	$\bm{(\supseteq)}$ Let $(\tau, \x, \z) \in \overline \cT_s$. Then, there exists $\s \in \cX$ with $s_i \leq x_i$ for all $i \in \cI$ and $\w \in \set{0,1}^n$ with $\w \leq \z$, $s_i(1 - w_i) = 0$ for all $i \in [d]$, and $\ones^\top \w \leq 1$ such that $\a^\top \x = \a^\top \s$ and $\sum_{i \in [d]} h(a_i s_i) \leq \tau$. We will next show that $(\tau, \x, \z) \in \cT + \cR$.
	From $\w \in \set{0,1}^n$ and $\ones^\top \w \leq 1$ we deduce that at most one element of $\w$ is equal to 1. Then, through the constraints  $s_i(1 - w_i) = 0$ for all $i \in [d]$ and $\w \leq \z$, we deduce that $\s$ has at most one nonzero element and $s_i(1 - z_i) = 0$ for all $i \in [d]$ as well. Hence, the constraint $\sum_{i \in [d]} h(a_i s_i) \leq \tau$ implies that $(\tau, \s, \z) \in \cT$. Moreover, the vector $\bar \x = \x - \s$ satisfies $\bar \x \in \cX$ (as $s_i \leq x_i$ for all $i \in \cI$) and $\a^\top \bar \x = 0$ (as $\a^\top \x = \a^\top \s$). Thus, we have $(0, \bar \x, \bm 0) \in \cR$. Therefore, $(\tau, \x, \z) = (\tau, \s, \z) + (0, \bar \x, \bm 0) \in \cT + \cR$, as required. This implies that $\cT + \cR \supseteq \overline \cT_s$. 
	This completes the proof.
	\qed
\end{proof}

We next show that $\conv(\overline \cT_s)$ can be obtained from $\conv(\widehat \cT)$. 
\begin{lemma}
	\label{lem:conv}
	Under Assumption~\ref{asmp:nonnegative}, we have
	\begin{align*}
		\conv(\overline \cT_s) = \set{(\tau, \x, \z): \exists \s, \w, \t \text{~s.t.~} 
			\begin{array}{l}
				(\x, \s, \z, \w, \t) \in \conv(\widehat \cT) \\ [1ex]
				\ones^\top \t = \tau,~ \a^\top \s = \a^\top \x
			\end{array}
		},
	\end{align*}
	where $\widehat \cT$ is as defined in~\eqref{eq:T:hat}.
\end{lemma}

\begin{proof} 
	Let $\cU := \R \times \widehat \cT,~ \cV_1 := \{(\tau, \x, \s, \z, \w, \t\,) \in \R^{5d+1} : \a^\top \x = \a^\top \s \}$, and $\cV_2 := \{(\tau, \x, \s, \z, \w, \t\,) \in \R^{5d+1} : \ones^\top \t \leq \tau \}$. By definition, we have 
	\begin{align*}
		\conv(\overline \cT_s) = \conv(\Proj_{\tau, \x, \z} (\cU \cap \cV_1 \cap \cV_2)) = \Proj_{\tau, \x, \z}(\conv (\cU \cap \cV_1 \cap \cV_2)).
	\end{align*}
	Hence, the claim will follow if we prove that $\conv(\cU \cap \cV_1 \cap \cV_2) = \conv(\cU) \cap \cV_1 \cap \cV_2$.
	From Lemma~\ref{lem:tools}\ref{lem:separable}, we have $\conv(\cU \cap \cV_1 \cap \cV_2) = \conv(\cU \cap \cV_1) \cap \cV_2$. Therefore, in the sequel we will show that $\conv(\cU \cap \cV_1) = \conv(\cU) \cap \cV_1$.
	
	As $\cV_1$ is convex, it is sufficient to show $\conv(\cU) \cap \cV_1 \subseteq \conv(\cU \cap \cV_1)$.
	Take a point $(\bar \tau, \bar \x, \bar \s, \bar \z, \bar \w, \bar \t) \in \conv(\cU) \cap \cV_1$. Since this point is in $\cV_1$, we have $\a^\top \bar \x = \a^\top \bar \s$. Additionally, as the point is in $\conv(\cU)$, by Carath\'{e}odory's theorem, we have 
	$(\bar \tau, \bar \x, \bar \s, \bar \z, \bar \w, \bar \t) = \sum_{k \in [q]}  \lambda_k (\tau^k, \x^k, \s^k, \z^k, \w^k, \t^k)$
	for some finite $q$, $(\tau^k, \x^k, \s^k, \z^k, \w^k, \t^k) \in \cU$, and $\lambda_k > 0$ with $\sum_{k \in [q]} \lambda_k = 1$. Define the~sets 
	\begin{align*}
		\cK_{<} &:= \set{k \in [q]: ~ \a^\top \s^k < \a^\top \x^k }, ~ \cK_{=} := \set{k \in [q]: \a^\top \s^k = \a^\top \x^k }, \\
		&\qquad\qquad\qquad \cK_{>} := \set{k \in [q]: \a^\top \s^k > \a^\top \x^k }.
	\end{align*}
	We consider two possible scenarios.
	
	\textbf{(i)} If $|\cK_{=}| = q$, then $|\cK_{<}| = |\cK_{>}| = 0$ and $(\tau^k, \x^k, \s^k, \z^k, \w^k, \t^k) \in \cV_1$ for every $k \in [q]$. Thus, $(\tau^k, \x^k, \s^k, \z^k, \w^k, \t^k) \in \cU \cap \cV_1$ for every $k \in [q]$, which implies that $(\bar \tau, \bar \x, \bar \s, \bar \z, \bar \w, \bar \t) \in \conv(\cU \cap \cV_1)$. 
	
	\textbf{(ii)} If $|\cK_{=}| < q$, then $|\cK_{<}| > 0$ and $|\cK_{>}| > 0$ because $\a^\top \bar \x = \a^\top \bar \s$. In the following we will use a rounding scheme that iteratively replace the points $(\tau^k, \x^k, \s^k, \z^k, \w^k, \t^k)$ with new points $(\hat \tau^k, \hat \x^k, \hat \s^k, \hat \z^k, \hat \w^k, \hat \t^k)$.
	Pick an index $j \in \cK_{<}$ and an index $l \in \cK_{>}$, and consider the following construction 
	\begin{align*}
		(\hat \tau^k, \hat \x^k, \hat \s^k, \hat \z^k, \hat \w^k, \hat \t^k) = \left\{
		\begin{array}{l@{\quad}l} 
			(\tau^j, \s^j, \s^j, \z^j, \w^j, \t^j) & k = j \\ [1ex]
			(\tau^l, \x^l + (\lambda_j / \lambda_l) (\x^j - \s^j), s^l, \z^l, w^l) & k = l \\ [1ex]
			(\tau^k, \x^k, \s^k, \z^k, \w^k, \t^k) & k \notin \set{j, l}.
		\end{array} \right.
	\end{align*}
	By construction, we have 
	\begin{align*}
		\sum_{k \in [q]}  \lambda_k (\tau^k, \x^k, \s^k, \z^k, \w^k, \t^k) = \sum_{k \in [q]}  \lambda_k (\hat \tau^k, \hat \x^k, \hat \s^k, \hat \z^k, \hat \w^k, \hat \t^k).
	\end{align*}
	Moreover, from $(\tau^j, \x^j, \s^j, \z^j, \w^j, \t^j) \in \cU$, we deduce $0 \leq s^j_i \leq x^j_i$ for every $i \in \cI$. Then, $0 \leq s^l_i \leq x^l_i \leq x^l_i + (\lambda_j / \lambda_l) (x^j_i - s^j_i) = \hat x^l_i$ for every $i \in \cI$. Thus, we conclude that $(\hat \tau^k, \hat \x^k, \hat \s^k, \hat \z^k, \hat \w^k, \hat \t^k) \in \cU$ for every $k \in [q]$.
	Defining the index sets $\hat \cK_{=} = \set{k \in [q]: \a^\top \hat \s^k = \a^\top \hat \x^k}$, one can show that $| \hat \cK_{=} | > | \cK_{=} |$ because the index $j$ now satisfies the condition $\a^\top \hat \s^j = \a^\top \hat \x^j $. We next replace the points $(\tau^k, \x^k, \s^k, \z^k, \w^k, \t^k)$ with the new points $(\hat \tau^k, \hat \x^k, \hat \s^k, \hat \z^k, \hat \w^k, \hat \t^k)$, and repeat the same rounding scheme. In this way, after at most $q$ iterations, we will obtain a set of points $(\hat \tau^k, \hat \x^k, \hat \s^k, \hat \z^k, \hat \w^k, \hat \t^k)$, $k \in [q]$, for which $|\hat \cK_{=}| = q$. Hence, $(\hat \tau^k, \hat \x^k, \hat \s^k, \hat \z^k, \hat \w^k, \hat \t^k) \in \cU \cap \cV_1$ for every $k \in [q]$ and we conclude that $(\bar \tau, \bar \x, \bar \s, \bar \z, \bar \w, \bar \t) \in \conv(\cU \cap \cV_1)$. This completes the proof.
	\qed
\end{proof}

Given Lemmas~\ref{lem:epi:perspective} and~\ref{lem:conv}, and Proposition~\ref{prop:nonnegative}, we are now ready to prove Theorem~\ref{thm:nonnegative}.
\begin{proof} \!\emph{of Theorem~\ref{thm:nonnegative}~}
	Recall the definition of $\widehat \cT$ and $\Omega$. By letting 
	\begin{align*}
		\begin{array}{c}
			\gb_i := (a_i s_i, x_i), ~ \bm \gamma := \t, ~ \gd := (\w, \z), ~ \Delta := \Omega, \\ [1ex]
			\bC_i = \begin{bmatrix} 0 & \mathds{1}_{\cI}(i) \\[1ex]  \mathds{1}_{\cI}(i) / a_i & 0 \\[1ex] -\mathds{1}_{\cI}(i) / a_i & \mathds{1}_{\cI}(i) \end{bmatrix}, ~ \C_i = \R_+^3, ~ \forall i \in [d],
		\end{array}
	\end{align*}
	we can represent the set $\widehat \cT$ as an instance of the set $\cW$ defined as in~\eqref{eq:W}. Then, Proposition~\ref{prop:separable} yields
	\begin{align*}
		\cl\conv(\widehat \cT) \!=\!
		\set{ (\x, \z, \s, \w, \t) \!\in\! \R^{d+d+d+d+d} \!: 
			\begin{array}{l}
				h^\pi(a_i s_i, w_i) \leq t_i, ~ \forall i \in [d] \\ [1ex]
				\s, \x \in \cX, ~ s_i \leq x_i, \, \forall i \in \cI \\ [1ex]
				(\w, \z) \in \conv(\Omega)
			\end{array}
		}.
	\end{align*}
	In the following we characterize $\cl\conv(\cT)$ in terms of $\cl\conv(\widehat \cT)$. From Proposition~\ref{prop:nonnegative} and Lemma~\ref{lem:conv}, we deduce	that $\cl\conv(\cT)$ coincides with
	\begin{align*}
		\cl \Big( \Proj_{\tau, \x, \z} \big( \{(\tau, \x, \z, \s, \w, \t) \in \R \times \conv(\widehat \cT) : \ones^\top \t = \tau,~ \a^\top \s = \a^\top \x\} \big) \Big).
	\end{align*}
	By letting 
	\begin{align*}
		\begin{array}{c}
			\gm := (\tau, \x, \z), ~ \bm \eta := (\tau, \x, \z, \s, \w, \t), ~ \cU := \R \times \conv(\widehat \cT), \\ [1ex]
			\bA := [\bI_{1+2d}, ~ \bm 0], ~ \bB := \begin{bmatrix} -1 & \bm 0^\top & \bm 0^\top & \bm 0^\top & \bm 0^\top & \ones^\top \\ 0 & \a^\top & \bm 0^\top & -\a^\top & \bm 0^\top & \bm 0^\top \end{bmatrix}^\top, ~ \bm b := \bm 0,
		\end{array}
	\end{align*}
	we observe that 
	\begin{align*}
		\Proj_{\tau, \x, \z} & \big( \{(\tau, \x, \z, \s, \w, \t) \in \R \times \conv(\widehat \cT) : \ones^\top \t = \tau,~ \a^\top \s = \a^\top \x\} \big) \\
		&= \{ \gm : \exists \bm \eta \in \cU \text{ s.t. } \bA \bm \eta = \gm, ~ \bB \bm \eta = \bm b \}
	\end{align*}
	as in Lemma~\ref{lem:tools}\,\ref{lem:cl}. Note also that the first requirement of Lemma~\ref{lem:tools}\,\ref{lem:cl} is trivially satisfied as the variable $\tau$ is free to choose from the set $\cU$ and the variable $\x$ is linearly dependent to the variable $\s$. In addition, we have
	\begin{align*}
		&\rec(\{\bm \eta \!\in\! \cl(\cU) \!:\! \bA \bm \eta \!=\! \bm 0, \bB \bm \eta \!=\! \bm b\}) \\
		&=\rec\left( \set{(0, \bm 0, \bm 0, \s, \w, \bm t) \!:\! 
			\begin{array}{l}
				h^\pi(a_i s_i, w_i) \leq t_i, ~ \forall i \in [d] \\ [1ex]
				\s, \in \cX, ~ s_i \leq 0, \, \forall i \in \cI \\ [1ex]
				(\w, \bm 0) \in \conv(\Omega)\\ [1ex]
				\a^\top \s = 0, ~ \ones^\top \t = 0
				\end{array}
				} \right)\\
		&=\rec\left( \set{(0, \bm 0, \bm 0, \s, \bm 0, \bm t) \!:\! 
			\begin{array}{l}
				h^\pi(a_i s_i, 0) \leq t_i, ~ \forall i \notin \cI \\ [1ex]
				s_i =0, \, 0\leq t_i,\, \forall i \in \cI \\ [1ex]
				\a^\top \s = 0, ~ \ones^\top \t = 0
				\end{array}
				} \right),
	\end{align*}
	where the first equation holds by the definition of $\conv(\widehat \cT)$ and the fact that $\bA \bm \eta = \bm 0$ implies $(\tau,\x,\z)=(0,\bm 0,\bm 0)$. The second equation follows from $(\w,\z)\in\conv(\Omega)$, which implies $\bm 0\leq\w\leq\z$ and as $\z=0$ we deduce $\w=0$. Moreover, $\x=0$ implies $0\leq s_i\leq x_i$ for all $i\in\cI$ and so $s_i=0$ for all $i\in\cI$. This also results in $h^\pi(a_i s_i, 0)=h^\pi(0,0)=0\leq t_i$ for all $i\in\cI$.
	Since the function $h$ is nonlinear, proper, lower semicontinuous and convex, the set $\{(u, v): h^\pi(u, 0) \leq v \}$ is a convex closed pointed cone thanks to Lemma~\ref{lem:epi:perspective}. Consequently, the origin is an extreme point of the set, meaning that $\sum_{i \in [d]} (a_i s_i, t_i) = (0, 0)$ only if $s_i = t_i = 0$. Hence, the second requirement of Lemma~\ref{lem:tools}\,\ref{lem:cl} also follows, and we can apply Lemma~\ref{lem:tools}\,\ref{lem:cl} to conclude that
	\begin{align*}
		\cl\conv(\cT) = \set{ (\tau, \x, \z): \exists \s, \w, \t ~\text{~s.t.~}~
			\begin{array}{l}
				(\x, \z, \s, \w, \t) \in \cl\conv(\widehat \cT), \\ [1ex]
				\ones^\top \t = \tau, ~ \a^\top \x = \a^\top \s 
			\end{array}
		}.
	\end{align*}  
	The proof concludes by projecting out the variable $\t$ using the Fourier-Motzkin elimination approach.
	\qed
\end{proof}

We conclude this section by characterizing $\conv(\Omega)$ where $\Omega = \{(\w, \z) \in \set{0,1}^d \times \cZ:~ \ones^\top \w \leq 1, ~\w \leq \z\}$ for some simple integral sets $\cZ$. The proofs of the subsequent results are provided in Appendix~\ref{appendix:A}.
We start with the case when $\cZ$ is defined by a cardinality constraint. In this case, the resulting set $\Omega$ admits an immediate totally unimodular representation.

\begin{lemma}
	\label{conv:Omega:cardinality}
	Suppose $\cZ = \{ \z \in \{0,1\}^d : \ones^\top \z \leq \kappa \}$, where $\kappa \in [d] \backslash \{1\}$. Then
	\begin{align*}
		\conv(\Omega) = \{(\w, \z) \in [0,1]^{d+d}: ~ \ones^\top \w \leq 1, ~\w \leq \z,~ \ones^\top \z \leq \kappa \}.
	\end{align*}
\end{lemma}

We next characterize $\conv(\Omega)$ for the weak and strong hierarchy sets. Unlike the previous case, $\Omega$ does not immediately admit a totally unimodular representation. 
Nonetheless, it turns out that the set $\Omega \backslash \{ \bm 0 \}$ is totally unimodular for weak hierarchy constraints. Using Lemma~\ref{lem:conv:Omega}\,\ref{lem:conv:Omega:0} in Appendix~\ref{appendix:A}, which provides a description of $\conv(\Omega)$ based on $\conv(\Omega \backslash {\bm 0})$, the following lemma analyzes the weak hierarchy constraints.

\begin{lemma}
	\label{conv:Omega:weak:hierarchy}
	Suppose $\cZ=\{\z\in\{0,1\}^d:~z_d\leq\sum_{i\in[d-1]} z_i\}$. Then,
	\begin{align*}
		\textstyle
		\conv(\Omega)= \set{(\w,\z) \in[0,1]^{d+d} :~
			\begin{array}{l}
				\ones^\top \w \leq 1, ~ \w \leq \z, ~z_d \leq \sum_{i\in[d-1]} z_i \\ [1ex]
				\ones^\top \w \leq \sum_{i\in[d-1]} z_i
			\end{array}
		}.
	\end{align*}
\end{lemma} 

We conclude this section by examining the strong hierarchy constraints. In this case, neither the set $\Omega$ nor $\Omega \backslash \set{\bm{0}}$ admits an immediate representation with totally unimodular matrices. Using Lemma~\ref{lem:conv:Omega}\,\ref{lem:conv:Omega:O_i} in Appendix~\ref{appendix:A}, the following lemma analyzes the strong hierarchy constraints.

\begin{lemma}
	\label{conv:Omega:strong:hierarchy}
	Suppose $\cZ = \{\z \in \{0,1\}^d:~z_d \leq z_i,\,\forall i\in[d-1] \}$. Then,
	\begin{align*}
		\conv(\Omega) \!=\! 
		\set{(\w, \z) \!:\! \exists \tilde \z^0, \dots, \tilde \z^d \text{ s.t. } 
			\begin{array}{l}
				\w \in \R_+^d, ~\tilde \z^0_d \geq 0, ~\ones^\top \w \leq 1, \\ [1ex]
				\z = \tilde \z^0 + \sum_{i \in [d]} \tilde \z^i \\ [1ex]
				\tilde \z^0_d \leq \tilde z^0_j \leq 1 - \ones^\top \w, ~ \forall j \in [d-1] \\ [1ex]
				\tilde z^i_i = w_i, ~ \forall i \in [d-1], \\ [1ex] 
				0 \leq \tilde \z^i_d \leq \tilde z^i_j \leq w_i, ~ \forall i,j \in [d-1] \\ [1ex]
				\tilde z^d_j = w_d, ~ \forall j \in [d]
			\end{array}
		}.
	\end{align*}
\end{lemma} 
\section{Numerical Results}
\label{sec:numerical}
In this section we study the numerical performance of our conic formulations on a nonlinear logistic regression problem with quadratic features. The resulting exponential cone programs are modeled with JuMP~\citep{lubin2015computing} and solved with MOSEK~10 on a MacBook Pro with a $2.80$ GHz processor and $16$GB RAM. In these experiments, we set the time limit to $300$ seconds and the number of threads of the solver to $4$.

In the nonlinear logistic regression problem with quadratic features, for an input data $\bm \phi \in \R^p$, we construct the lifted feature vector 
\[
\a = ( (\phi_k)_{k \in [p]}, (\phi_{k}^2)_{k \in [p]}, (\phi_{k} \phi_{l})_{k,l \in [p]: l > k} ) \in \R^d,
\] 
where $d = p(p+3) / 2$. We denote the coefficients of the nonlinear classifier and its support by the vectors $\x \in \R^d$ and $\z \in \{ 0, 1 \}^d$. With slight abuse of notation, we use the notation $\x = ( (x_k)_{k \in [p]}, (x_{kk})_{k \in [p]}, (x_{kl})_{k,l \in [p]: l > k}) )$ to refer to the elements of $\x$. In a similar fashion, we use the notation $\z = ( (z_k)_{k \in [p]}, (z_{kk})_{k \in [p]}, (z_{kl})_{k,l \in [p]: l > k}) )$ for $\z$. We examine the following sparse logistic regression problem
\begin{align}
	\label{eq:logistic}
	\begin{array}{c@{\quad}l@{\quad}l}
		\displaystyle \min & \displaystyle \frac{1}{N} \sum_{j \in [N]} \log(1 + \exp(-b_j \a_j^\top \x)) + \lambda \sum_{i \in [d]} z_i + \mu \| \x \|_2^2 \\ [4ex]
		\text{s.t.} & \x \in \R_+^d, ~ \z \in \cZ, ~ x_i (1 - z_i) = 0, ~ \forall i \in [d],
	\end{array}
\end{align} 
where $(\a_j, b_j)_{j \in [N]}$ represents the (nonlinear) feature-label pairs constructed from the input vector $(\bm \phi_j)_{j \in [N]}$ in the training data, and $\lambda, \mu \in \R_+$ denote the regularization coefficients. We use the set 
\[
\cZ = \set{ \z \in \{ 0, 1\}^d :~ z_{kk} \leq z_k,~ z_{kl} \leq z_k,~ z_{kl} \leq z_l, ~ \forall k,l \in [p] \text{~s.t.} ~ l > k }
\] 
to capture strong hierarchy constraints.
We consider various reformulations of~\eqref{eq:logistic} using the convex hull results presented in Section~\ref{sec:applications}.
Namely, based on Theorem~\ref{thm:separable}, we introduce the \emph{separable} reformulation as
\begin{align}
	\tag{separable}
	\label{eq:separable:logloss}
	\begin{array}{c@{\quad}l@{\quad}l}
		\displaystyle \min & \displaystyle \frac{1}{N} \sum_{j \in [N]} t_i + \log(2) + \lambda \sum_{i \in [d]} z_i + \mu \sum_{i \in [d]} r_i \\ [4ex]
		\text{s.t.} & \x \in \R_+^d, ~ \z \in \cZ, ~ \bm u, \bm v \in \R^N,  \\[1ex]
		& (r_i, z_i, x_i) \in \cK_{\text{rsoc}}, & \forall i \in [d], \\ [1ex]
		& (v_j, 1, -t_j), (u_j, 1,-b_j \a_j^{\top} \x - t_j ) \in \cK_{\exp}, & \forall j \in [N], \\[1ex]
		& u_j + v_j \leq 2, & \forall j \in [N].
	\end{array}
\end{align}
where $\cK_{\exp} = \{(\xi_1, \xi_2, \xi_3): \xi_1 \geq \xi_2 e^{\frac{\xi_3}{\xi_2}},\, \xi_2 > 0 \} \cup \{(\xi_1, 0, \xi_3):\xi_1 \geq 0,\, \xi_3 \leq 0 \}$ and $\cK_{\text{rsoc}} = \{(\xi_1, \xi_2, \xi_3): \xi_1 \xi_2 \geq \xi_3^2,\, \xi_1, \xi_2 \geq 0 \} $ denote the exponential and the rotated second-order cones, respectively.
Furthermore, based on Theorem~\ref{thm:connected}, we introduce the \emph{$\text{rank}_1$} reformulation as 
\begin{align}
	\tag{$\rank_1$}
	\label{eq:rank-one:logloss}
	\begin{array}{c@{\quad}l}
		\displaystyle \min & \displaystyle \frac{1}{N} \sum_{j \in [N]} t_i + \log(2) + \lambda \sum_{i \in [d]} z_i + \mu \sum_{i \in [d]} r_i \\ [4ex]
		\text{s.t.} & \x, \bm r \in \R_+^d, ~ (\w, \z) \in \Delta_b, ~ \bm u, \bm v, \bm t \in \R^N, \\[1ex]
		& (r_i, z_i, x_i) \in \cK_{\text{rsoc}}, \hspace{12.8em} \forall i \in [d], \\ [1ex]
		& (v_j, w_j, -t_j), (u_j, w_j,-b_j \a_j^{\top} \x - t_j ) \in \cK_{\exp}, \qquad \forall j \in [N], \\ [1ex]
		& u_j + v_j \leq 2, \hspace{15.4em} \forall j \in [N]. 
	\end{array}
\end{align}
where $\Delta_b := \{ (\w, \z) \in  \{0, 1\}^N \times \cZ: w_j \leq \sum_{i \in [d]: a_{ji} \neq 0} z_i, ~ \forall j \in [N] \}$.
Note that Lemma~\ref{lem:conv:Z:strong:hierarchy} implies that the set
\begin{align*}
	\Delta_v = \set{ (\w, \z) \in \R^N \times \R^d: \begin{array}{l} w_j \leq \sum_{k,l \in [p]: l > k,  a_{jk}, a_{jl} \neq 0} (z_k + z_l - z_{kl}), \\ [1ex] w_j \leq \sum_{k \in [p]: a_{jk} \neq 0} z_k, \forall j \in [N]  \end{array}}
\end{align*}
generates valid cuts for the binary set $\Delta_b$.
Finally, based on Theorem~\ref{thm:nonnegative}, we introduce the \emph{$\text{rank}_1^+$} formulation as 
\begin{align}
	\tag{$\rank_1^+$}
	\label{eq:rank-one-plus:logloss}
	\begin{array}{c@{\quad}l}
		\displaystyle \min & \displaystyle \frac{1}{N} \sum_{j \in [N]} t_i + \log(2) + \lambda \sum_{i \in [d]} z_i + \mu \sum_{i \in [d]} r_i \\ [4ex]
		\text{s.t.} & \x, \bm r \in \R_+^d, ~ (\w, \z) \in \Omega_b, ~ \bm u, \bm v, \t, \bm s \in \R^{N \times d},  \\[1ex]
		& (r_i, z_i, x_i) \in \cK_{\text{rsoc}}, \hspace{6.6em} \quad \forall i \in [d], \\ [1ex]
		& \sum_{i \in [d]} a_{ji} s_{ji} = \a_j^\top \x, \hspace{6.5em} \forall j \in [N], \\ [1ex]
		& 0 \leq s_{ji} \leq \mathds{1}_{\{a_{ji} \neq 0\}} \, x_i,  \hspace{6.3em} \forall i \in [d], \forall j \in [N], \!\! \\ [1ex]
		&(v_{ji}, w_{ji}, -t_{ji}) \in \cK_{\exp}, \hspace{5.9em} \forall i \in [d], \forall j \in [N], \!\! \\[1ex]
		& (u_{ji}, w_{ji},-b_j \a_{ji} s_{ji} - t_{ji} ) \in \cK_{\exp}, \hspace{1.2em} \forall i \in [d], \forall j \in [N], \!\! \\[1ex]
		& u_{ij} + v_{ij} \leq 2, \hspace{9.8em} \forall i \in [d], \forall j \in [N]. \!\!\!
	\end{array}
\end{align}
where $\Omega_b \!:=\! \{ (\w, \z) \!\in\! \{0, 1\}^{N \times d} \!\times\!\cZ \!:\! \w \, \ones \!\leq\! \ones, w_{ji} \!\leq\! \mathds{1}_{\{a_{ji} \neq 0\}} \, z_i, \forall i \!\in\! [d], \! \forall j \!\in\! [N] \}$. It is easy to verify that the set 
\begin{align*}
	\Omega_v = \set{ (\w, \z) \in \R^{N \times d} \times \R^d: \begin{array}{l} w_{j,k} + w_{j, l} + w_{j, kl} \leq z_k + z_l - z_{kl}, \\ [1ex] w_{j,k} + w_{j, kk} \leq z_k, \\ [1ex]  \forall j \in [N], \, \forall k,l \in [p] \text{~s.t.~} l > k, a_{jk}, a_{jl} \neq 0 \end{array} }
\end{align*}
generates valid cuts for the binary set $\Omega_b$.

We examine the relaxation quality and branch and bound (B\&B) performance of different reformulations of~\eqref{eq:logistic} in terms of the optimality gap, solution time, and number of B\&B nodes.
Namely, we examine the separable, rank-one, and rank-one-plus relaxations obtained by relaxing the integrality restrictions of the boolean sets involved. For example, $\relx(\Omega_b)$ corresponds to the continuous relaxation of the set $\Omega_b$, etc. Note that all of the additional inequalities obtained from the sets  $\Delta_v$ or $\Omega_v$ are directly included as constraints in the corresponding formulations. Thus, our implementation does not rely on specific MOSEK functionalities like cut generation or adding cuts on the fly, and these additional inequalities are utilized both in the convex relaxations at the root node and throughout the nodes of the branch-and-bound tree.

Inspired by \citep[Section~6]{wei2022convex}, we conduct a numerical experiment in which the input data $(\bm \phi_j)_{j \in [N]}$ is sparse. Specifically, we randomly assign $\phi_{ji}$ to either a point sampled from a standard Gaussian distribution with a mean of zero and variance of one with probability $\pi$ or set it to zero with probability $1 - \pi$, using a threshold value $\pi \in (0, 1]$. We then randomly  generate a true coefficient vector $\x_0 \in [-1, 1]^d$. Using the vector $\x_0$, we finally generate the label $b_j \in \{-1,1\}$ by sampling from a Bernoulli distribution with $\P(b_j = 1 | \a_j) = 1 / (1 + \exp(- \a_j^\top \x_0))$.

We examine the quality of following relaxations:
\begin{itemize}
	\item In \emph{natural relaxation}, we drop the complementary constraints in~\eqref{eq:logistic}. It is easy to see that the relaxed problem is solved by $\z^\star = 0$.
	\item In \emph{separable relaxation}, we replace $\cZ$ in \eqref{eq:separable:logloss} with $\relx(\cZ)$.
	\item In \emph{$\text{rank}_1\!$ relaxation}, we replace $\Delta_b$ in \eqref{eq:rank-one:logloss} with $\relx(\Delta_b)$.
	\item In \emph{$\text{rank}_{1,v}\!$ relaxation}, we replace $\Delta_b$ in \eqref{eq:rank-one:logloss} with $\relx(\Delta_b \cap \Delta_v)$.
	\item In \emph{$\text{rank}_1^+\!$ relaxation}, we replace $\Omega_b$ in \eqref{eq:rank-one-plus:logloss} with $\relx(\Omega_b)$.
	\item In \emph{$\text{rank}_{1,v}^+\!$ relaxation}, we replace $\Omega_b$ in \eqref{eq:rank-one-plus:logloss} with $\relx(\Omega_b \cap \Omega_c)$.
\end{itemize}
Note that \citep[Theorem~1]{wei2022ideal} cannot be applied directly when a complete description for $\conv(\cZ \backslash \{ \bm 0 \})$ is not available. Nonetheless, the suggested valid inequalities from set $\Delta_v$ can still be employed to obtain a convex relaxation. As a result, we propose the subsequent relaxation as a substitute for \citep[Theorem~1]{wei2022ideal}
\begin{align*}
	\label{eq:rank-one:wei}
	\tag{$\rank_{1,w}$}
	\begin{array}{c@{\quad}l}
		\displaystyle \min & \displaystyle \frac{1}{N} \sum_{j \in [N]} t_i + \log(2) + \lambda \sum_{i \in [d]} z_i + \mu \sum_{i \in [d]} r_i \\ [4ex]
		\text{s.t.} & \x, \bm r \in \R_+^d, ~ (\w, \z) \in \relx(\Delta_w), ~ \bm u, \bm v \in \R^{4 \times N}, ~ \bm t \in \R^N, \\[1ex]
		& (r_i, z_i, x_i) \in \cK_{\text{rsoc}}, \hspace{12.9em} \forall i \in [d], \\ [1ex]
		& (v_{kj}, w_{kj}, -t_j), (u_{kj}, w_{kj},-b_j \a_j^{\top} \x - t_j ) \in \cK_{\exp}, ~ \forall j \in [N], \forall k \in [4], \hspace{-3em} \\ [1ex]
		& u_{ij} + v_{ij} \leq 2, \hspace{15.0em} \forall j \in [N], \forall k \in [4], \hspace{-3.2em}
	\end{array}
\end{align*}
where the set 
\begin{align*}
	\Delta_w = \set{ (\w, \z) \in \R^{4 \times N} \times \{0,1\}^d: \begin{array}{l} w_{1j} = 1, ~ w_{2j} = \sum_{i \in [d]: a_{ji} \neq 0} z_i \\ [1ex] w_{3j} = \sum_{k,l \in [p]: l > k,  a_{jk}, a_{jl} \neq 0} z_k + z_l - z_{kl}, \\ [1ex] w_{4j} = \sum_{k \in [p]: a_{jk} \neq 0} z_k, ~ \forall j \in [N] \end{array} }.
\end{align*}
Note that there are two main differences between $\text{rank}_{1,v}$ and $\text{rank}_{1,w}$  formulations. First and foremost, through Theorem~\ref{thm:connected} we have the variables $\w$ in $\text{rank}_{1,v}$ restricted to be binary whereas in $\text{rank}_{1,w}$ they are continuous and indeed are simply explicit functions of other variables. And second, while the strengthening via valid inequalities in both formulations involves the same set of valid inequalities, in $\text{rank}_{1,v}$ this strengthening is done in the lifted space in a linear form, and in $\text{rank}_{1,w}$ it is essentially done in the original space of the variables through nonlinear inequalities.

In the first experiment we set $\lambda = 10^{-2}, \mu = 10^{-4}, p = 50, d = 1325, \pi = 0.1$ and $N \in \{ 10, 50, 100, 200 \}$. 
In Figure~\ref{fig:change:N} we report the optimality gap and solution time for different convex relaxations. In determining the optimality gap, we compare the objective value of the given relaxation against the best known feasible solution for the instance (among the ones found by the B\&B method applied to any of the formulations of~\eqref{eq:logistic} presented in \eqref{eq:separable:logloss}, \eqref{eq:rank-one:logloss} and \eqref{eq:rank-one-plus:logloss}. This solution corresponds to the optimal solution if the time limit has not been reached in the B\&B algorithm, or the best feasible solution reported by MOSEK if the time limit has been reached. 

The results in Figure~\ref{fig:change:N} suggest that the qualities of the convex relaxations based on the \eqref{eq:rank-one-plus:logloss} formulation are the best. In particular, $\text{rank}_1^+\!$ relaxation attains an average gap smaller than $5\%$. Moreover, adding valid inequalities to the sets $\Delta_b$ and $\Omega_b$ significantly improve the quality of the $\text{rank}_1\!$ and $\text{rank}_1^+\!$ relaxations. For example, the $\text{rank}_{1,v}^+\!$ relaxation attains the average gap smaller than $1\%$, which is $5$ times smaller than the gap attained by $\text{rank}_{1}^+\!$. However, adding these valid inequalities comes with a computational downside. Specifically, the optimization problems now involve more constraints, which result in longer solution times. It is also worth to note that the optimality gap of the $\text{rank}_1\!$ relaxations is significantly superior to that of the separable and natural relaxations, albeit at the expense of relatively longer solution times. 

Finally, we highlight that the optimality gap of the continuous relaxations from $\text{rank}_{1,v}\!$ and $\text{rank}_{1,w}\!$, with the latter being inspired by \citep[Theorem~1]{wei2022ideal}, are identical. This is due to the fact that when the variables $\w$ are relaxed to be continuous the projection of $\text{rank}_{1,v}\!$ in the original space leads to precisely the same inequalities as in $\text{rank}_{1,w}\!$. Nonetheless, it takes roughly twice as long to solve the $\text{rank}_{1,w}\!$ relaxation than the $\text{rank}_{1,v}\!$ relaxation. This is expected as the $\text{rank}_{1,w}\!$ relaxation introduces a considerably larger number of constraints and variables compared to the $\text{rank}_{1,v}\!$ relaxation. It is also important to note that a complete implementation of \citep[Theorem~1]{wei2022ideal} requires using a characterization of $\conv(\cZ \backslash \{ \bm 0 \})$ which may possibly involve an exponential number of constraints. In contrast, $\text{rank}_{1,v}\!$ relaxation handles this complexity through the use of binary variables $\w$, making the $\text{rank}_{1,v}\!$ relaxation much more applicable in practice. 

\begin{figure}
	\centering
	\begin{subfigure}[b]{0.95\textwidth}
		\centering
		\includegraphics[width=\textwidth]{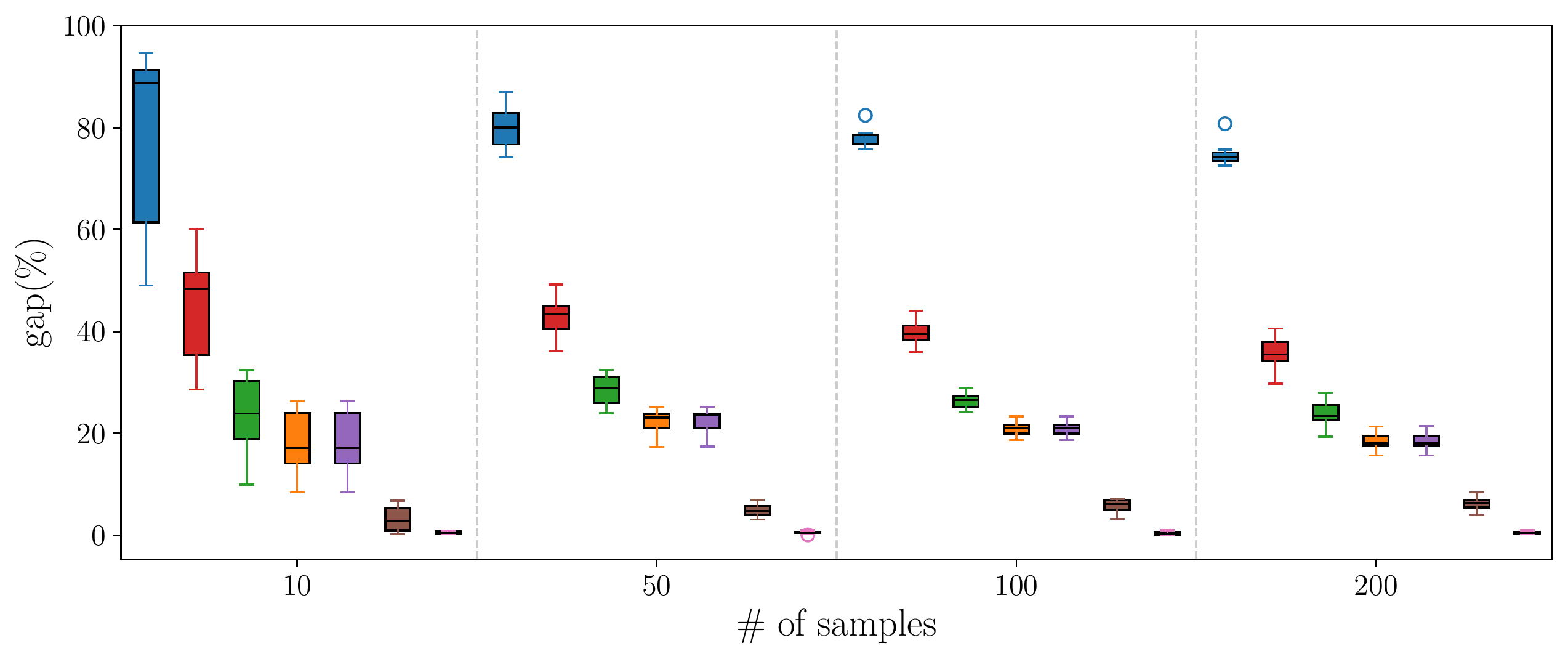}
	\end{subfigure} \\
	\begin{subfigure}[b]{0.95\textwidth}
		\centering
		\includegraphics[width=\textwidth]{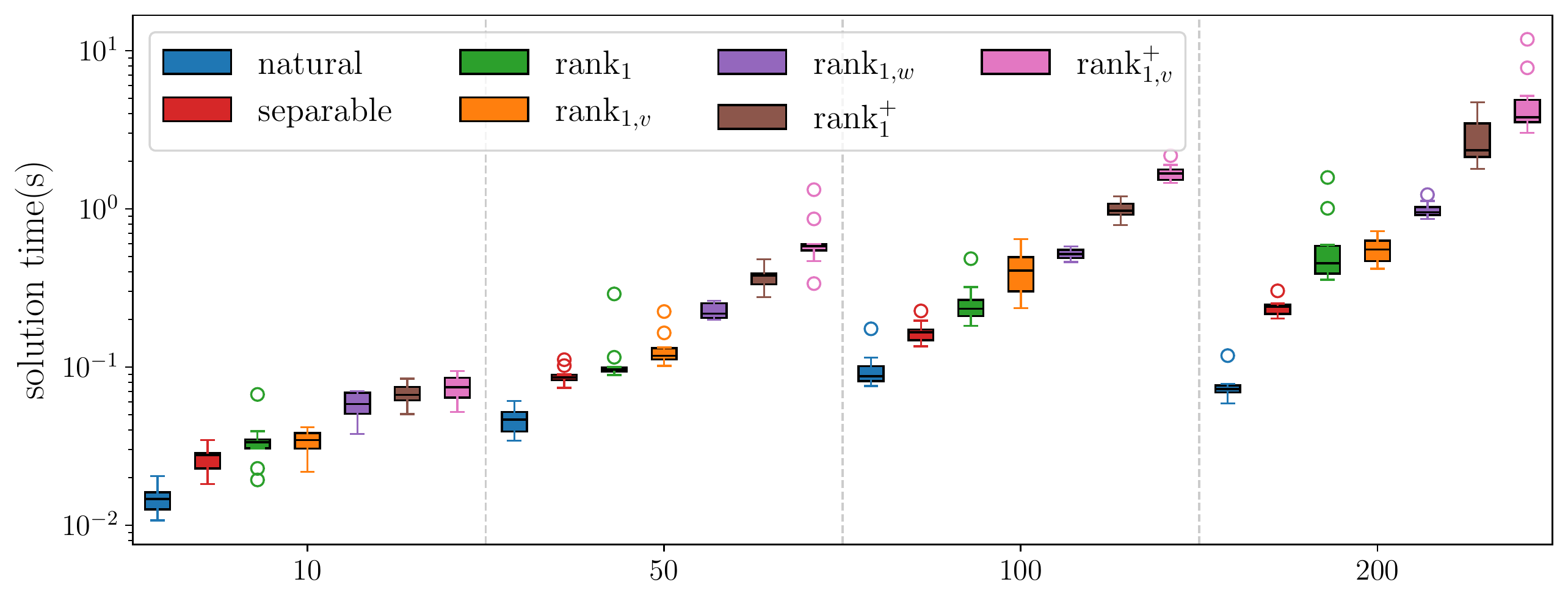}
	\end{subfigure}
	\caption{Comparison of different continuous relaxations for $p = 50$ and $\pi = 0.1$ as $N$ varies.}
	\label{fig:change:N}
\end{figure}

We next examine the B\&B performance of these alternative formulations of~\eqref{eq:logistic} in which we always keep the variables $\z$ as binary but we create two variants for each of the $\text{rank}_{1}\!$ and $\text{rank}_{1}^+\!$ formulations based on whether the variables $\w$ are kept as binary or $\w$ are relaxed to be continuous:
\begin{itemize}
	\item In \emph{separable reformulation}, we consider \eqref{eq:separable:logloss}.
	\item In \emph{$\text{rank}_1\!$ reformulation}, we consider \eqref{eq:rank-one:logloss}.
	\item In \emph{$\text{rank}_{1,r}\!$ reformulation}, we replace $\Delta_b$ in \eqref{eq:rank-one:logloss} with $\relx_{\w}(\Delta_b)$.
	\item In \emph{$\text{rank}_{1,v}\!$ reformulation}, we replace $\Delta_b$ in \eqref{eq:rank-one:logloss} with $\Delta_b \cap \Delta_v$.
	\item In \emph{$\text{rank}_{1,r,v}\!$ reformulation}, we replace $\Delta_b$ in \eqref{eq:rank-one:logloss} with $\relx_{\w}(\Delta_b \cap \Delta_v)$.
	\item In \emph{$\text{rank}_{1,w}\!$ reformulation}, we consider \eqref{eq:rank-one:wei}.
	\item In \emph{$\text{rank}_1^+\!$ reformulation}, we consider \eqref{eq:rank-one-plus:logloss}.
	\item In \emph{$\text{rank}_{1,r}^+\!$ reformulation}, we replace $\Omega_b$ in \eqref{eq:rank-one-plus:logloss} with $\relx_{\w}(\Omega_b)$.
	\item In \emph{$\text{rank}_{1,v}^+\!$ reformulation}, we replace $\Omega_b$ in \eqref{eq:rank-one-plus:logloss} with $\Omega_b \cap \Omega_v$.
	\item In \emph{$\text{rank}_{1,r,v}^+\!$ reformulation}, we replace $\Omega_b$ in \eqref{eq:rank-one-plus:logloss} with $\relx_{\w}(\Omega_b \cap \Omega_c)$.
\end{itemize}
Figure~\ref{fig:change:N:bb} reports the true optimality gap computed using the best known heuristic solution (in our experiments this was usually obtained by a variant of \eqref{eq:rank-one:logloss} formulation) as discussed earlier, the optimality gap reported by the solver, the number of B\&B nodes, and the solution time. 

\begin{figure}
	\centering
	\begin{subfigure}[b]{0.95\textwidth}
		\centering
		\includegraphics[width=\textwidth]{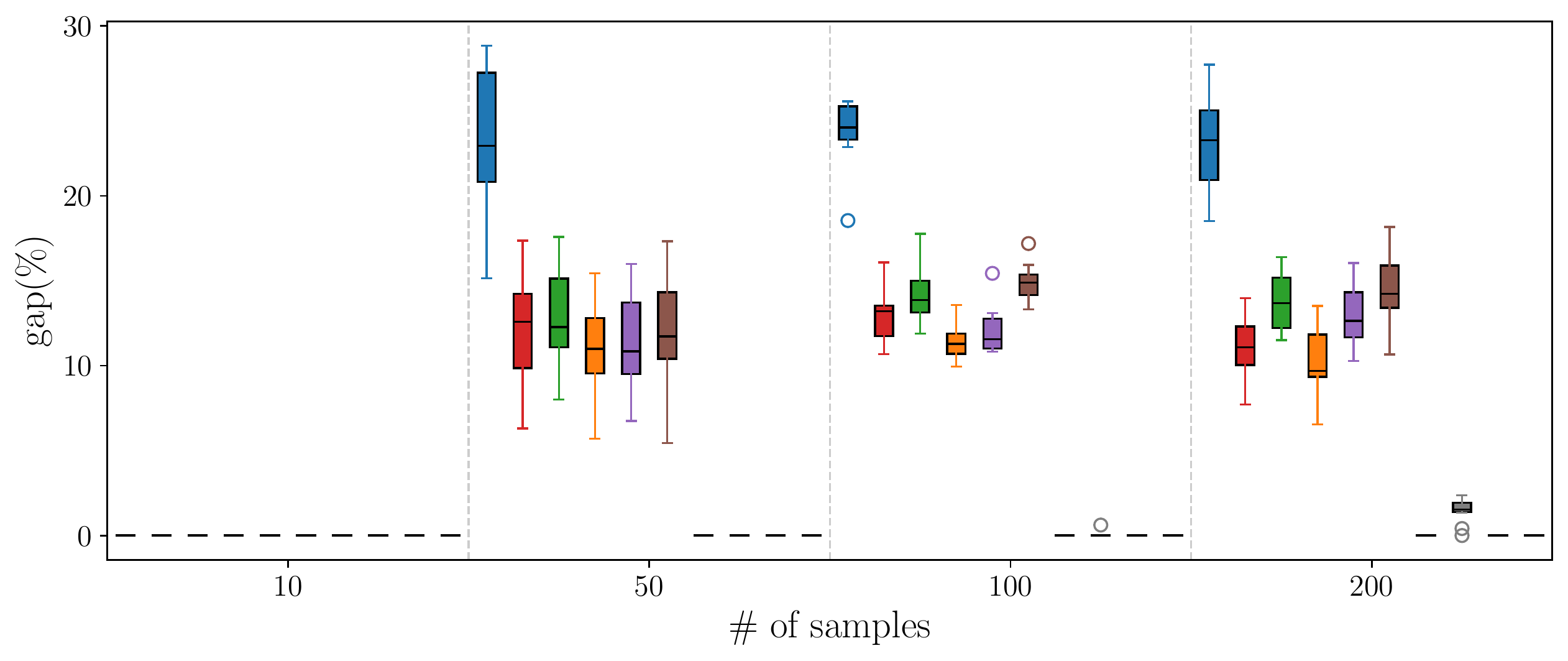}
	\end{subfigure} \\
	\begin{subfigure}[b]{0.95\textwidth}
		\centering
		\includegraphics[width=\textwidth]{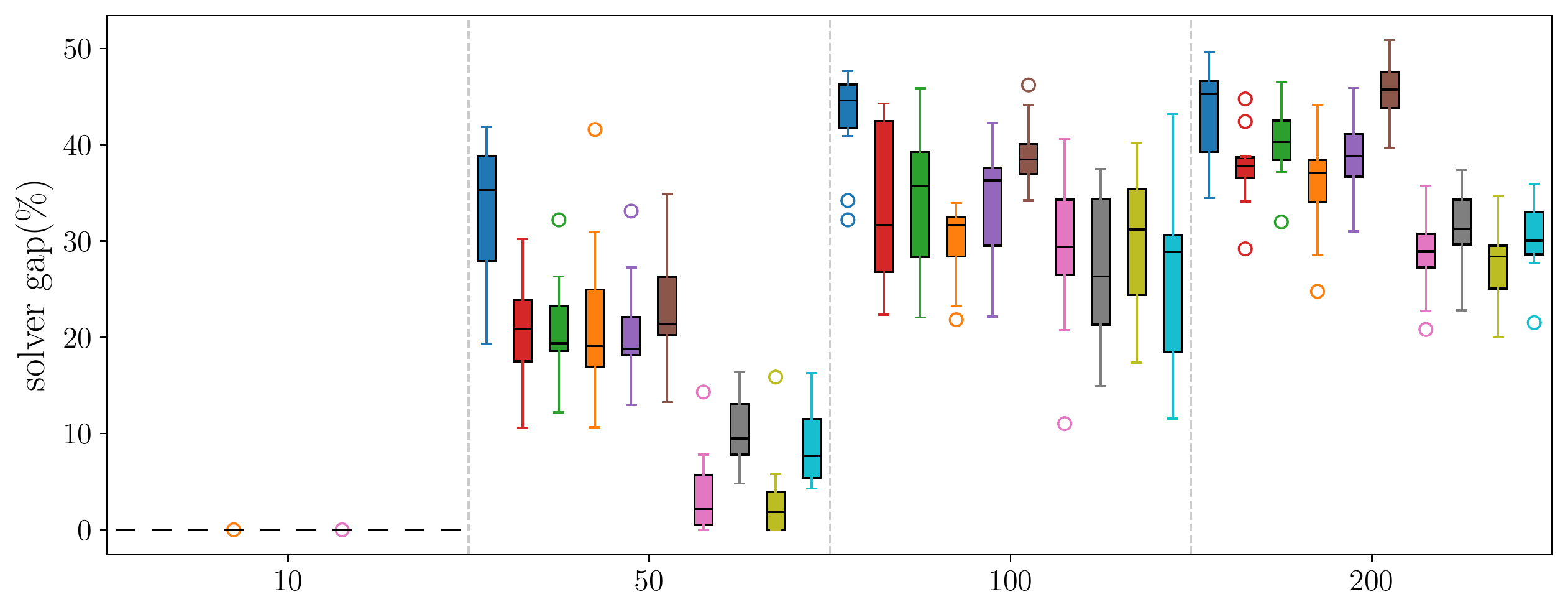}
	\end{subfigure} \\
	\begin{subfigure}[b]{0.95\textwidth}
		\centering
		\includegraphics[width=\textwidth]{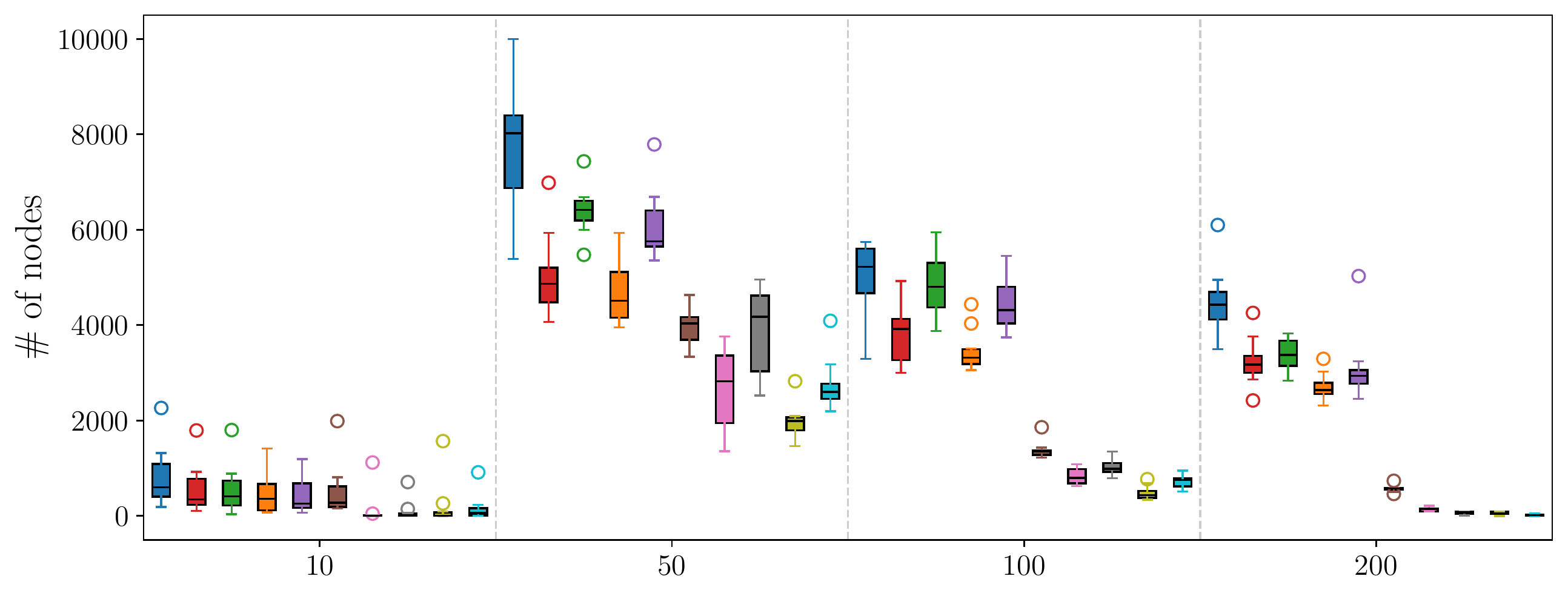}
	\end{subfigure} \\
	\begin{subfigure}[b]{0.95\textwidth}
		\centering
		\includegraphics[width=\textwidth]{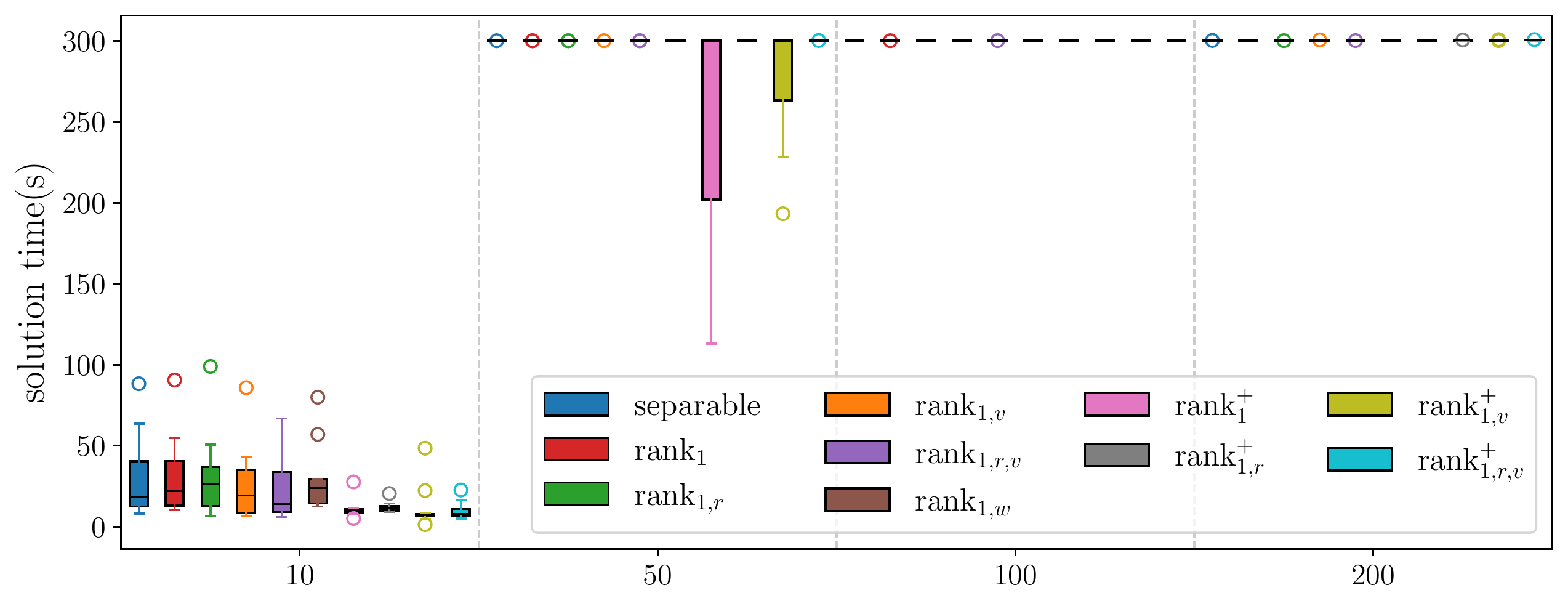}
	\end{subfigure}
	\caption{The B\&B performance for $p = 50$ and $\pi = 0.1$ as $N$ varies.}
	\label{fig:change:N:bb}
\end{figure}

We start by analyzing the true optimality gap and solver gap in Figure~\ref{fig:change:N:bb}. As also seen in Figure~\ref{fig:change:N}, in terms of true optimality gap, the formulations based on \eqref{eq:rank-one-plus:logloss} consistently outperform those based on \eqref{eq:rank-one:logloss} in terms of the true optimality gap. This is primarily because the relaxations based on \eqref{eq:rank-one-plus:logloss} can produce high-quality lower bounds, even though these formulations take longer to solve and thus result in significantly fewer number of nodes explored in B\&B. 
Among $\text{rank}_{1}\!$ and $\text{rank}^+_{1}\!$ variants, the ones where the variables $\w$ are kept as binary perform better in terms of the optimality gaps than the ones where $\w$ are relaxed to be continuous even though having $\w$ binary results in fewer B\&B nodes. This is because when $\w$ are binary, the solver is able to leverage the structure of the sets $\Delta_b$ and $\Omega_b$ to generate further cuts and results in higher quality relaxations. 
Among $\text{rank}_{1}\!$ variants, the performance of $\text{rank}_{1,v}\!$ seems to be best and $\text{rank}_{1,w}\!$ seems to be the worst.  As the continuous relaxation of $\text{rank}_{1,w}$ takes longer to solve compared to those based on \eqref{eq:rank-one:logloss}, its B\&B can explore only a smaller number of nodes, and therefore, results in a worse optimality gap.

When we examine the gaps reported by the solver, we still observe that the same phenomena, but this time the associated gaps reported by the solver are considerably larger than the true optimality gaps. This is because for this class of problems, the heuristic methods utilized by the solver are not very advanced, and the good quality feasible solutions of a formulation are often found at integral nodes in the B\&B tree, so essentially by chance. 
Therefore, the B\&B procedure for the formulations that admit quick to solve node relaxations often results in high quality feasible solutions. Consequently, in the case of expensive lifted formulations such as  \eqref{eq:rank-one-plus:logloss} while the actual optimality gaps are very close to zero, the solver is unable to report this gap due to the inferior quality of the feasible solution found in the associated B\&B tree. To address this issue, BARON introduces the concept of ``relaxation-only constraints'' \citep{sahinidis1996baron,sahinidis2013baron}. These constraints are active during the relaxation step, but are disregarded during the local search step. To the best of our knowledge, MOSEK 10 does not support this feature, and thus in our figures we report both the true optimality gap and the solver reported gap for the formulations tested.

In the second experiment we set $\lambda = 10^{-2}, \mu = 10^{-4}, N = 100, \pi = 0.1$ and $p \in \{ 10, 20, 30, 40, 50 \}$, which translates to $d \in \{ 65, 230, 495, 860, 1325 \}$.
Figure~\ref{fig:change:p} and Figure~\ref{fig:change:p:bb} in Appendix~\ref{appendix:B} compare the quality of different continuous relaxations and also report the performance of the B\&B algorithm, respectively. 
The observations from Figure~\ref{fig:change:p} are very similar to the ones in Figure~\ref{fig:change:N}; thus we omit this discussion for brevity.
Despite the similarity between the observations in Figure~\ref{fig:change:p:bb} and Figure~\ref{fig:change:N:bb}, it is worth noting that when $p=10$, the B\&B performance is slightly different. In particular, when $p=10$, all methods except $\text{rank}_{1,r}^+\!$ and $\text{rank}_{1,r,v}^+\!$ can solve the optimization problem in less than $20$ seconds. This implies that the integer programs may be relatively simple to solve when the dimension is small. As a result, stronger relaxations may not be needed when dealing with small scale instances.

In the last experiment we set $\lambda = 10^{-2}, \mu = 10^{-4}, p = 50, d = 1325, N = 100$ and $\pi \in \{ 0.1, 0.3, 0.5, 0.7, 0.9 \}$; see
Figure~\ref{fig:change:prob} and Figure~\ref{fig:change:prob:bb} in Appendix~\ref{appendix:B} for the quality of different continuous relaxations and the B\&B performance. In all of these instances time limit was reached in B\&B, so the solution time is not reported in Figure~\ref{fig:change:prob:bb}.
As the value of $\pi$ increases, we notice that the optimality gap of the $\text{rank}_{1}\!$ relaxation gets closer to that of separable relaxation.  
This is expected as the binary variable $w_j$ models whether $\a_j^\top \x = 0$ or not. When $\pi$ is large, the probability of such event is low. As a result, $w_j$ is assigned a value of $1$ with high probability, which make the $\text{rank}_{1}\!$ relaxations much less effective.
As a final observation, we note that the value of $\pi$ seems to not affect the quality of the $\text{rank}_{1}^+\!$ relaxations; in particular these relaxations continue to be of high quality even for high values of $\pi$.

\paragraph{Acknowledgements} This research is supported by Early Postdoc Mobility Fellowship SNSF grant P2ELP2\_195149 and AFOSR grant FA9550-22-1-0365.

\bibliographystyle{myabbrvnat}
\bibliography{mybibfile}

\newpage
\renewcommand\thesection{\Alph{section}}
\setcounter{equation}{0}
\setcounter{section}{0}
\renewcommand{\theequation}{A.\arabic{equation}}
\setcounter{lemma}{0}
\renewcommand{\thelemma}{A.\arabic{lemma}}
\setcounter{proposition}{0}
\renewcommand{\theproposition}{A.\arabic{proposition}}

\section{Additional Proofs}
\label{appendix:A}

\begin{proof}\!\emph{of Lemma~\ref{lem:tools}~}
	As for assertion~\ref{lem:separable}, let $\overline \cU \!:=\! \R \times \cU$ and $\overline \cV := \{ (\tau, \gm, \t \hspace{0.1em}) : \ones^\top \t = \tau \}$. 
	By construction, $\cV$ is the projection of $\overline \cU \cap \overline \cV$ onto $(\tau, \gm)-$space. Hence, $\conv(\cV) = \conv(\Proj_{\tau, \gm}(\overline \cU \cap \overline \cV)) = \Proj_{\tau, \gm}(\conv(\overline \cU \cap \overline \cV))$. Therefore, in the remainder of the proof we characterize convex hull of $\overline \cU \cap \overline \cV$.
	As $\overline \cV$ is convex, to prove $\conv(\overline \cU\cap\overline \cV) = \conv(\overline \cU)\cap\overline \cV$ it suffices to show that $\conv(\overline \cU) \cap \overline \cV \subseteq \conv( \overline \cU \cap \overline \cV)$. 
	
	Take a point $(\bar \tau, \bar \gm, \bar \t \hspace{0.1ex})$ from $\conv(\overline \cU) \cap \overline \cV$. Since $(\bar \tau, \bar \gm, \bar \t \hspace{0.1ex}) \in \overline \cV$, we have $\bar \tau = \ones^\top \bar \t$. 
	On the other hand, since $(\bar \tau, \bar \gm, \bar \t \hspace{0.1ex}) \in \conv(\overline \cU)$, we can always express it as a convex combination of a finite number of points in $\overline \cU$, that is, $(\bar \tau, \bar \gm, \bar \t \hspace{0.1ex}) = \sum_{k \in [q]} \lambda_k (\tau^k, \gm^k, \t^k)$ for some finite $q$, 
	$(\tau^k, \gm^k, \t^k) \in \overline \cU$, and $\lambda_k > 0$ for all $k\in[q]$ satisfy $\sum_{k \in [q]} \lambda_k = 1$.
	Notice that there is no restriction on $\tau$ in the definition of $\overline \cU$. Hence, we can take $\tau^k := \ones^\top \t^k$ for all $k \in [q]$. Therefore, we have $(\tau^k, \gm^k, \t^k) \in \overline \cV$ for all $k \in [q]$. 
	We thus deduce $(\tau^k, \gm^k, \t^k) \in \overline \cU \cap \overline \cV$ for all $k \in [q]$.
	Moreover, because $\bar \t = \sum_{k \in [q]} \lambda_k \t^k$, we have $\sum_{k \in [q]} \lambda_k \tau^k = \sum_{k \in [q]} \lambda_k (\ones^\top t^k) = \ones^\top(\sum_{k \in [q]} \lambda_k t^k) = \ones^\top \bar \t = \bar \tau$. Hence, this proves that any point $(\bar \tau, \bar \gm, \bar \t \hspace{0.1ex})$ can be written as a convex combination of points from $\overline \cU \cap \overline \cV$ as desired. 
	Therefore, $\conv(\overline \cU \cap \overline \cV) = \conv(\overline \cU) \cap \overline \cV$, and the claim follows by projecting onto $(\tau, \gm)-$space.
	
	As for assertion~\ref{lem:cl}, note that the set $\cV$ is the affine transformation of the convex set $\widetilde \cU = \{ \bm \eta \in \cU: \bB \bm \eta = \bm b \}$, that is, $\cV = \bA \, \widetilde \cU$. We then have
	\begin{align*}
		\cl(\cV) 
		= \cl(\rint(\cV)) 
		= \cl(\bA(\rint(\widetilde \cU))) 
		&= \cl(\bA(\rint(\cl(\widetilde \cU)))) \\
		&= \cl(\rint(\bA(\cl(\widetilde \cU)))) 
		= \cl(\bA(\cl(\widetilde \cU))).
	\end{align*}
	The first equality holds since a convex set and its relative interior have the same closure. The second equality holds as the linear transformation and the relative interior operators are interchangeable for convex sets (see \citep[Theorem~6.6]{rockafellar1970convex}); thus, we have $\rint(\cV) = \rint(\bA \, \widetilde \cU) = \bA(\rint(\widetilde \cU))$. The third equality holds as the relative interior of a convex set and the relative interior of its closure are the same. The fourth equality follows from the convexity of $\cl(\widetilde \cU)$, which allows us to interchange the linear transformation and the relative interior operators. Finally, the last inequality holds as the closure of the relative interior of a convex set equals the closure of the set. Since we assumed that there exists a point $\bm \eta^\star \in \rint(\cU)$ satisfying the condition $\bB \bm \eta^\star = \bm b$, we have $\cl(\widetilde \cU) = \{ \bm \eta \in \cl(\cU): \bB \bm \eta = \bm b \}$ thanks to \citep[Corrolay~6.5.1]{rockafellar1970convex}. Thus, we showed that
	\begin{align*}
		\cl(\cV) = \cl(\bA(\cl(\widetilde\cV))) = \cl ( \{ \gm : \exists \bm \eta \in \cl(\cU) \text{ s.t. } \bA \bm \eta = \gm, \bB \bm \eta = \bm b \} ).
	\end{align*}
	As we assumed that $\bm \eta = \bm 0$ is the only $\bm \eta \in \rec(\cl(\widetilde\cV))$ with $\bA \bm \eta = \bm 0$, by \citep[Theorem~9.1]{rockafellar1970convex}, we have $\cl(\bA(\cl(\widetilde\cV))) \!=\! \bA(\cl(\widetilde\cV))$. This completes the~proof.\!
	\qed
\end{proof}

\begin{proof}\!\emph{of Lemma~\ref{lem:conv:Z:cardinality}~}
	Note that the matrix
	\begin{align*}
		\bA = \begin{bmatrix}
			-1 & ~ \ones^\top \\[0.5ex]
			0 & ~ \ones^\top
		\end{bmatrix}
	\end{align*}
	is totally unimodular. Recall that if $A$ is totally unimodular, then the matrix $[\bA^\top,~ \bI,~ -\bI]^\top$ is also totally unimodular. 
	By definition, $\Delta_1 = \{(w,\z)\in\{0,1\}^{1+d}:~w\leq \ones^\top \z, ~\ones^\top \z \leq \kappa \}$. 
	As the feasible set of $\Delta_1$ will be represented by the matrix $[\bA^\top,~ \bI,~ -\bI]^\top$ and an integer vector, the set $\Delta_1$ is totally unimodular, and $\conv(\Delta_1)$ is thus given by the continuous relaxation of $\Delta_1$.\!\!\!\!
	\qed
\end{proof}

\begin{proof}\!\emph{of Lemma~\ref{lem:conv:Z:weak:hierarchy}~}
	Note that the matrix
	\begin{align*}
		\bA = \begin{bmatrix}
			1 & -\ones_{d-1}^\top & 0 \\[0.5ex]
			0 & -\ones_{d-1}^\top & 1
		\end{bmatrix}
	\end{align*}
	is totally unimodular as every square submatrix of $A$ has determinant $0$, $+1$, or $-1$.
	Moreover, it is easy to see that 
	\begin{align*}
		\Delta_1 & \textstyle= \{(w,\z) \in \{0,1\}^{1+d}:~ w \leq \ones^\top \z, ~z_d \leq \sum_{i\in[d-1]} z_i\} \\
		& \textstyle= \{(w,\z) \in \{0,1\}^{1+d}:~ w \leq \sum_{i\in[d-1]} z_i, ~z_d \leq \sum_{i\in[d-1]} z_i\},
	\end{align*}
	which implies that the feasible set of $\Delta_1$ can be represented by the matrix $[\bA^\top,~ \bI,~ -\bI]^\top$ and an integer vector. Thus, the new representation of $\Delta_1$ is totally unimodular, and $\conv(\Delta_1)$ is given by its continuous relaxation. 
	\qed
\end{proof}

\begin{proof}\!\emph{of Lemma~\ref{lem:conv:Delta:1}~}	
	The set $\Delta_1$ can be written as 
	\begin{align*}
		\Delta_1 
		&= \left( \set{0} \times \set{\bm 0} \right) \cup \left( \set{0,1} \times \cZ\backslash\set{\bm 0} \right).
	\end{align*}
	Since $\conv(A \times B) = \conv(A) \times \conv(B)$ and $\conv(A \cup B) = \conv(\conv(A) \cup \conv(B))$, we arrive at
	\begin{align*}
		\conv(\Delta_1) 
		&= \conv \left( \left( \set{0} \times \set{\bm 0} \right) \cup \left( [0,1] \times \conv(\cZ \backslash \{ \bm 0 \} ) \right) \right) \\[1ex]
		&= \set{ (w, \z) \in \R^{1+d}: \exists \lambda \in [0,1] \text{~s.t.~} \begin{array}{l} \bF^0 \z \geq \bm 0, ~ 0 \leq w \leq \lambda \\ [1ex] \z^\top \bm f^+_k \geq \lambda, ~ \forall k \in \cK, \\ [1ex] \z^\top \bm f_l^- \leq \lambda, ~ \forall l \in \cL \end{array} }
	\end{align*}
	The proof concludes by projecting out the variable $\lambda$ using the Fourier-Motzkin elimination approach.	
	\qed
\end{proof}

\begin{proof}\!\emph{of Lemma~\ref{lem:conv:Z:strong:hierarchy}~}
	By \citep[Lemma~3]{wei2022ideal}, we have
	\begin{align}
		\label{eq:strong:hierarchy}
		\textstyle
		\conv(\cZ\backslash\set{\bm 0}) = \set{ \z \in [0,1]^d: ~ \begin{array}{l}
				1 \leq \sum_{i \in [d-1]} z_i - (d-2) z_d, \\[1ex]
				z_d \leq z_i, ~ \forall i\in[d-1]
		\end{array}}.
	\end{align}
	The proof concludes by applying Lemma~\ref{lem:conv:Delta:1}.
	\qed
\end{proof}

\begin{proof} \!\emph{of Theorem~\ref{thm:general}~}
	Recall the definition of $\widetilde \cT$ and $\Delta_p$. By letting 
	\begin{align*}
		\gb_i := (s_i, \x_i), \, \bm \gamma := \t, \, \gd := (\w, \z), \, \Delta := \Delta_p, \, \bC_i = [-1, \a_i^\top], \, \C_i = \set{0}, \, \forall i \in [p],
	\end{align*}
	we can represent the set $\widetilde \cT$ as an instance of the set $\cW$ defined as in~\eqref{eq:W}. Then, Proposition~\ref{prop:separable} yields
	\begin{align*}
		\cl\conv(\widetilde \cT) \!=\!
		\set{(\x, \z, \s, \w, \t) \!\in\! \R^{d+d+p+p+p} : 
			\begin{array}{l}
				h^\pi(s_i, w_i) \leq t_i, ~ \forall i \in [p] \\ [1ex]
				\a_i^\top \x_i = s_i, ~ \forall i \in [p] \\ [1ex]
				(\w, \z) \in \conv(\Delta_p) 
			\end{array}
		}.
	\end{align*}
	In the following we characterize $\cl\conv(\cT)$ in terms of $\cl\conv(\widetilde \cT)$. 
	Let $\cT' := \{(\tau, \x, \z): \exists \s,  \w, \t \text{~s.t.~} (\x, \z, \s, \w. \t) \in \widetilde \cT,  \ones^\top \t = \tau \}$.
	By Proposition~\ref{prop:connected}, we have
	\begin{align*}
		\cl\conv(\cT) 
		= \cl \big(\conv(\Proj_{\tau, \x, \z} (\cT'))\big)
		= \cl \big(\Proj_{\tau, \x, \z}(\conv (\cT'))\big).
	\end{align*}
	Moreover, applying Lemma~\ref{lem:tools}\,\ref{lem:separable} yields 
	$$\conv(\cT') \!=\! \{ (\tau, \x, \z): \exists \s, \w, \t \text{ s.t. } (\x, \z, \s, \w, \t) \in \conv(\widetilde \cT), \, \ones^\top \t = \tau \}.$$
	By letting 
	\begin{align*}
		\begin{array}{c}
			\gm := (\tau, \x, \z), ~ \bm \eta := (\tau, \x, \z, \s, \w, \t), ~ \cU := \R \times \conv(\widetilde \cT), \\ [1ex]
			\bA := [\bI_{1+2d}, ~ \bm 0], ~ \bB := (-1, \bm 0, \bm 0, \bm 0, \bm 0, \ones_p)^\top, ~ \bm b := 0,
		\end{array}
	\end{align*}
	we observe that $\Proj_{\tau, \x, \z}(\conv (\cT')) = \{ \gm : \exists \bm \eta \in \cU \text{ s.t. } \bA \bm \eta = \gm, ~ \bB \bm \eta = \bm b \}$ as in Lemma~\ref{lem:tools}\,\ref{lem:cl}. 
	Moreover, the first requirement of Lemma~\ref{lem:tools}\,\ref{lem:cl} is trivially satisfied as the variable $\tau$ is free in the set $\cU$. In addition, we have the set 
	\begin{align*}
		&\set{\bm \eta \in \cl(\cU):~ \bA \bm \eta = \bm 0,~ \bB \bm \eta = \bm b}\\
		= &\set{(0, \bm 0, \bm 0, \s, \w,\t):~
			\begin{array}{l}
				(\w, \bm 0) \in \conv(\Delta_p), ~ \ones ^\top \t=0 \\ [1ex]
				s_i = \a_i^\top \x_i,\, h^\pi(s_i,w_i)\leq t_i,\,\forall i\in[p]
			\end{array}
		} \\
		= &\set{(0, \bm 0, \bm 0, \s, \w,\t):  ~\w= \bm 0,~ \s =\bm 0,~ h^\pi(0,0)\leq t_i,\,\forall i\in[p], ~\ones ^\top \t=0} = \set{\bm 0},
	\end{align*}
	where the first equation holds by the definition of $\cl\conv(\widetilde \cT)$ and the fact that $\bA \bm \eta = \bm 0$ implies $(\tau,\x,\z)=(0,\bm 0,\bm 0)$, and the second equation holds because, by definition of $\Delta_p$ we have $0 \leq w_i \leq \ones^\top \z_i$ which enforces $w_i=0$ as $\z_i=\bm 0$ for all $i\in[p]$, and the last equation holds as the function $h$ satisfies $h(0)=0$, which along with $\ones ^\top \t=0$ implies $\t=\bm 0$. 
	Thus, we can apply Lemma~\ref{lem:tools}\,\ref{lem:cl} to conclude that
	\begin{align*}
		\cl\conv(\cT) = \{ (\tau, \x, \z): \exists \s, \w, \t \text{ s.t. } (\x, \z, \s, \w, \t) \in \cl\conv(\widetilde \cT), ~ \ones^\top \t = \tau \}.
	\end{align*}  
	The proof concludes by projecting out the variable $\t$ using the Fourier-Motzkin elimination approach.
	\qed
\end{proof}

The convex hull of $\Delta_p$ relies on the set 
$$\cZ_i := \{ \z_i \in \R^{d_i} : \exists \tilde \z \in \cZ ~ \text{s.t.} ~ \tilde \z_i = \z_i, ~ \tilde \z_j = \bm 0,  \forall j \neq i \}$$ 
for every $i \in [p]$. In particular, given any $i \in [p]$, we assume that 
\begin{align}
	\label{eq:Zi:0}
	\conv(\cZ_i \backslash \set{\bm 0}) = \set{ \z_i \in \R^{d_i} : \bF_i^0 \z_i \geq \bm 0, ~ \begin{array}{l} \z_i^\top \bm f^+_{ik} \geq 1, ~ \forall k \in \cK_i, \\ [1ex] \z_i^\top \bm f_{il}^- \leq 1, ~ \forall l \in \cL_i
	\end{array} }.
\end{align}
\begin{lemma}
	\label{lem:conv:Delta:p}
	Given the representation of $\conv(\cZ_i \backslash \set{\bm 0})$ as in~\eqref{eq:Zi:0} for any $i \in [p]$, then
	\begin{align*}
		\conv(\Delta_p)= \set{ (\w, \z) \in \R^{p+d} : \begin{array}{l} \bm 0 \leq \w \leq \bm 1, ~ \ones^\top \w \leq 1, ~ \bF_i^0 \z_i \geq \bm 0, ~ \forall i \in [p], \\ [1ex] \z_i^\top \bm f^+_{ik} \geq w_i, ~ \forall i \in [p], \forall k \in \cK_i, \\ [1ex] \z_i^\top \bm f_{il}^- \leq 1, ~ \forall i \in [p], \forall l \in \cL_i, \\ [1ex] \z_i^\top \bm f_{il}^- \leq \z_i^\top \bm f^+_{ik}, ~ \forall i \in [p],\forall k \in \cK, \forall l \in \cL
		\end{array} }.
	\end{align*}
\end{lemma}

\begin{proof}\!\emph{of Lemma~\ref{lem:conv:Delta:p}~}
	The set $\Delta_p$ can be written as the union of $p+1$ sets. Namely,
	\begin{align*}
		\Delta_p 
		&= \left( \set{\bm 0} \times \set{\bm 0} \right) \bigcup\limits_{i \in [p]} \left( \set{\bm 0, \e_i} \times \left( \times_{k \in [i-1]} \set{\bm 0} \times \cZ_i\backslash\set{\bm 0} \times_{k \in \{i+1, \dots, p\}} \set{\bm 0} \right) \right).
	\end{align*}
	Since $\conv(A \times B) = \conv(A) \times \conv(B)$ and $\conv(A \cup B) = \conv(\conv(A) \cup \conv(B))$, we arrive at
	\begin{align*}
		\conv(\Delta_p) 
		= \set{ (\w, \z) \in \R^{p+d} : \exists \bm \lambda\in\R^p \text{~s.t.~} \begin{array}{l} \bm 0 \leq  \bm \lambda \leq \ones, ~\ones^\top \bm \lambda \leq 1, ~ \w \leq \bm \lambda,  \\ [1ex] \bF_i^0 \z_i \geq \bm 0, ~ \forall i \in [p] \\ [1ex] \z_i^\top \bm f^+_{ik} \geq \lambda_i, ~ \forall i \in [p], \forall k \in \cK_i, \\ [1ex] \z_i^\top \bm f_{il}^- \leq \lambda_i, ~ \forall i \in [p], \forall l \in \cL_i\end{array} }
	\end{align*}
	The proof concludes by projecting out the variable $\bm \lambda$ using the Fourier-Motzkin elimination approach. 
	\qed
\end{proof}

\begin{proof}\!\emph{of Lemma~\ref{conv:Omega:cardinality}~}
	Note that the matrix
	\begin{align*}
		\bA = \begin{bmatrix}
			\ones^\top & ~ \phantom{-}\bm 0^\top \\
			\!\!\bI & ~ -\bI \\
			\bm 0^\top & \phantom{-}\ones^\top
		\end{bmatrix}
	\end{align*}
	is totally unimodular. To see this, we partition the rows of $\bA$ into the two disjoint sets $\{1\}$ and $\{2, \dots, d+2\}$. As such, $\bA$ is totally unimodular because every entry in $\bA$ is $0$, $+1$, or $-1$, every column of $\bA$ contains at most two non-zero entries, and two nonzero entries in a column of $\bA$ with the same signs belong to either $\{1\}$ or $\{2, \dots, d+2\}$, whereas two nonzero entries in a column of $\bA$ with the opposite signs belong to $\{2, \dots, d+2\}$. Hence, by \citep[Theorem~2]{heller1956extension}, the matrix $\bA$ is totally unimodular. 
	As the feasible set of $\Omega$ is represented by the matrix $[\bA^\top,~ \bI,~ -\bI]^\top$ and an integer vector, the set $\Omega$ is totally unimodular, and $\conv(\Omega)$ is given by its continuous relaxation.
	\qed
\end{proof}

We next provide ideal and lifted descriptions of $\conv(\Omega)$ for general~$\cZ$. Lemma~\ref{lem:conv:Omega}\,\ref{lem:conv:Omega:0} is particularly useful when the set $\Omega \backslash \set{\bm 0}$ is totally unimodular. 
In this case, we can provide an ideal description for $\conv(\Omega)$. On the other hand, Lemma~\ref{lem:conv:Omega}\,\ref{lem:conv:Omega:O_i} is useful when the set $\cZ$ is totally unimodular.
\begin{lemma}
	\label{lem:conv:Omega}
	The following holds.
	\begin{enumerate}[label=(\roman*)]
		\item \label{lem:conv:Omega:0} Suppose that $\conv(\Omega \backslash \set{\bm 0})$ admits the following ideal representation
		\begin{align*}
			\conv(\Omega \backslash \set{\bm 0}) = \set{ \bm \delta \in \R^{2d} : \bF^0 \bm \delta \geq \bm 0, ~ \begin{array}{l} \bm \delta^\top \bm f^+_k \geq 1, ~ \forall k \in \cK, \\ [1ex] \bm \delta^\top \bm f_l^- \leq 1, ~ \forall l \in \cL
			\end{array} }.
		\end{align*}
		Then, we have
		\begin{align*}
			\conv(\Omega) \!=\! \set{ \bm \delta \in \R^{d+d} : \bF^0 \bm \delta \geq \bm 0, ~ \begin{array}{l} \bm \delta^\top \bm f^+_k \geq 0, ~ \forall k \in \cK, \\ [1ex] \bm \delta^\top \bm f_l^- \leq 1, ~ \forall l \in \cL, \\ [1ex] \bm \delta^\top \bm f_l^- \leq \bm \delta^\top \bm f^+_k, ~ \forall k \in \cK, \forall l \in \cL
			\end{array} }.
		\end{align*}
	
		\item \label{lem:conv:Omega:O_i} Let $\cO_i = \{ \z \in \R^d: z_i = 0 \}$. Then, we have
		\begin{align*}
			\conv(\Omega) = \set{(\w, \z) \!:\! \exists \tilde \z^0, \dots, \tilde \z^d \text{ s.t. } 
				\begin{array}{l}
					\w \in \R_+^d, ~\ones^\top \w \leq 1, \\ [1ex]
					\z = \tilde \z^0 + \sum_{i \in [d]} \tilde \z^i \\ [1ex]
					\tilde \z^0 \in (1 - \ones^\top \w) \cdot \conv(\cZ) \\ [1ex]
					\tilde \z^i \in w_i \cdot \conv(\cZ \backslash \cO_i), ~\forall i \in [d]
				\end{array}
			}.
		\end{align*}
	\end{enumerate}
\end{lemma}

\begin{proof} 
	For case~\ref{lem:conv:Omega:0}, notice that the set $\Omega$ can be written as the union of two sets $\Omega = \left( \set{\bm 0} \times \set{\bm 0} \right) \cup \left( \Omega \backslash\set{\bm 0} \right)$. We thus have
	\begin{align*}
		\conv(\Omega) 
		&= \conv \left( \left( \set{\bm 0} \times \set{\bm 0} \right) \cup \left( \conv(\Omega \backslash \{ \bm 0 \} ) \right) \right) \\[1ex]
		&= \set{ \bm \delta \in \R^{d+d}: \exists \lambda \in [0,1] \text{~s.t.~} \bF^0 \bm \delta \geq \bm 0, ~\begin{array}{l} \bm \delta^\top \bm f^+_k \geq \lambda, ~ \forall k \in \cK, \\ [1ex] \bm \delta^\top \bm f_l^- \leq \lambda, ~ \forall l \in \cL \end{array} } .
	\end{align*}
	The proof concludes by projecting out the variable $\lambda$ using the Fourier-Motzkin elimination approach.
	
	For case~\ref{lem:conv:Omega:O_i}, note that the set $\Omega$ can be written as the union of $d+1$ sets, that is, $\Omega = \left( \set{\bm 0} \times \cZ \right) \bigcup \left(\cup_{i \in [d]} \left( \set{\e_i} \times (\cZ \backslash \cO_i) \right)\right)$.
	We thus have
	\begin{align*}
		&\conv(\Omega) \\
		=& \conv \left( \left( \set{\bm 0} \times \conv(\cZ) \right) \bigcup \left(\cup_{i \in [d]}\set{\e_i} \times \conv(\cZ \backslash \cO_i) \right) \right) \\[1ex]
		=& \set{(\w, \z) : 
			\begin{array}{l}
				\exists \w^0, \dots, \w^d \\
				\exists \z^0, \dots, \z^d \\
				\exists \bm \lambda \in \R_+^d
			\end{array} \text{ s.t. } 
			\begin{array}{l}
				\ones^\top \bm \lambda \leq 1, \\ [1ex]
				\w = (1 - \ones^\top \bm \lambda) \w^0 + \sum_{i \in [d]} \lambda_i \w^i \\ [1ex]
				\z = (1 - \ones^\top \bm \lambda) \z^0 + \sum_{i \in [d]} \lambda_i \z^i \\ [1ex]
				(\w^0, \z^0) \in \set{\bm 0} \times \conv(\cZ) \\ [1ex]
				(\w^i, \z^i) \in \set{\e_i} \times \conv(\cZ \backslash \cO_i ), ~\forall i \in [d]
			\end{array}
		}.
	\end{align*}
	The proof concludes by introducing the new variables $\tilde \z^0 = (1 - \ones^\top \bm \lambda) \z^0$ and $\tilde \z^i = \lambda_i \z^i$ for every $i \in [d]$, and then projecting out the variable $\bm \lambda$ using the observation that $\w = \bm \lambda$. 
	\qed
\end{proof}

Armed with Lemma~\ref{lem:conv:Omega}, we are ready to provide characterizations for $\conv(\Omega)$ in the cases of weak and strong hierarchy constraints.

\begin{proof}\!\emph{of Lemma~\ref{conv:Omega:weak:hierarchy}~}
	First, note that 
	\begin{align*}
		\Omega \backslash \{ \bm 0 \} & \textstyle= \{(\w,\z) \in \{0,1\}^{d+d}:~ \w \leq \z, ~\ones^\top \w \leq 1, ~ \sum_{i\in[d-1]} z_i \geq 1 \}.
	\end{align*}
	Based on this let us consider the matrix
	\begin{align*}
		\bA = \begin{bmatrix}
			\bI_d \quad & -\bI_d \\[0.5ex]
			\ones_{d}^\top \quad & \bm 0_{d}^\top \\[0.5ex]
			\bm 0_{d}^\top \quad & (-\ones_{d-1}, 0)^\top
		\end{bmatrix}.
	\end{align*}
	This matrix is totally unimodular as we can partition rows of $\bA$ into the subsets $[d]$ and $\{ d+1, d+2 \}$ such that they satisfy the requirements of \citep[Theorem~2]{heller1956extension}. 
	Thus, the set $\Omega \backslash \{ \bm 0 \}$ is represented by constraints based on the matrix $[\bA^\top,~ \bI,~ -\bI]^\top$ and an integer vector. Thus, $\Omega \backslash \{ \bm 0 \}$ admits a totally unimodular representation, and $\conv(\Omega \backslash \{ \bm 0 \})$ is given by its continuous relaxation. The proof concludes by applying Lemma~\ref{lem:conv:Omega}\,\ref{lem:conv:Omega:0}.
	\qed
\end{proof}

\begin{proof}\!\emph{of Lemma~\ref{conv:Omega:strong:hierarchy}~}
	Note that $\cZ$ is totally unimodular. Hence, the continuous relaxation of $\cZ$ gives $\conv(\cZ)$.
	For any $i \in [d-1]$, we also have $\cZ \backslash \cO_i = \{ \z \in \{0,1\}^d: z_i = 1, ~ z_d \leq z_j, ~ \forall j \in [d-1]  \}.$
	It is easy to see that the continuous relaxation of $\cZ \backslash \cO_i$ gives its convex~hull. Besides, we have $\cZ \backslash \cO_d = \{ \bm 1 \}$. The proof follows by applying Lemma~\ref{lem:conv:Omega}\,\ref{lem:conv:Omega:O_i}.
	\qed
\end{proof}

\newpage
\renewcommand\thesection{\Alph{section}}
\setcounter{figure}{0}
\renewcommand{\thefigure}{B.\arabic{figure}}

\section{Additional Numerical Results}
\label{appendix:B}

\begin{figure}[h]
	\centering
	\begin{subfigure}[b]{0.95\textwidth}
		\centering
		\includegraphics[width=\textwidth]{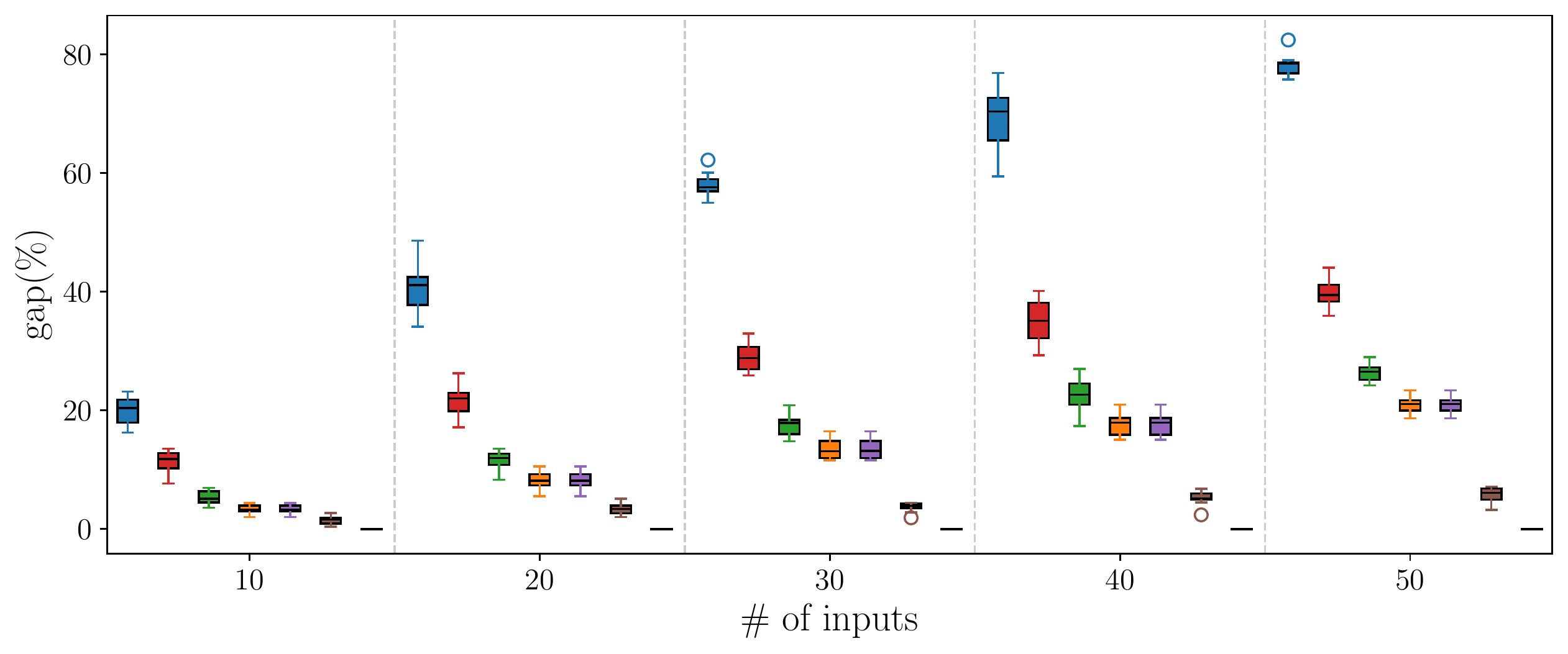}
	\end{subfigure} \\
	\begin{subfigure}[b]{0.95\textwidth}
		\centering
		\includegraphics[width=\textwidth]{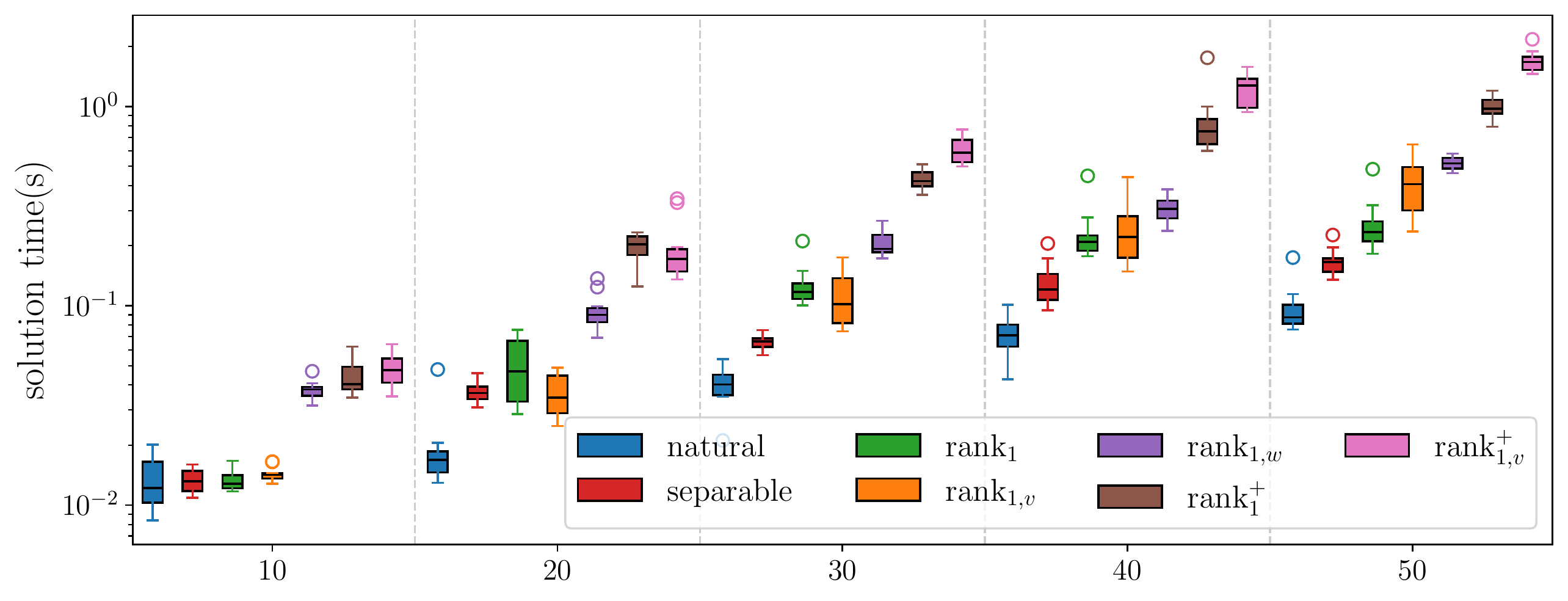}
	\end{subfigure}
	\caption{Comparison of different continuous relaxations for $N = 100$, $\pi = 0.1$ as $p$ (and thus $d$) varies.}
	\label{fig:change:p}
\end{figure}

\begin{figure}
	\centering
	\begin{subfigure}[b]{\textwidth}
		\centering
		\includegraphics[width=\textwidth]{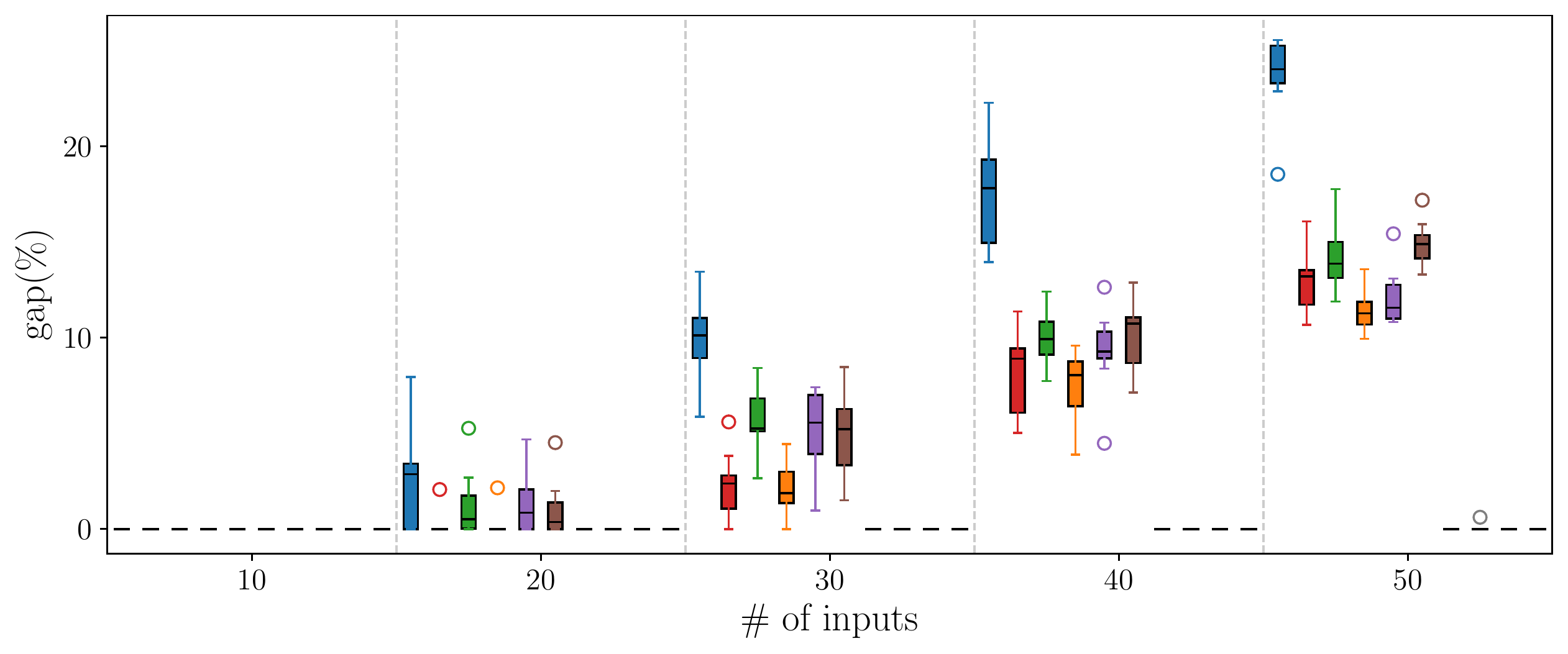}
	\end{subfigure} \\
	\begin{subfigure}[b]{\textwidth}
		\centering
		\includegraphics[width=\textwidth]{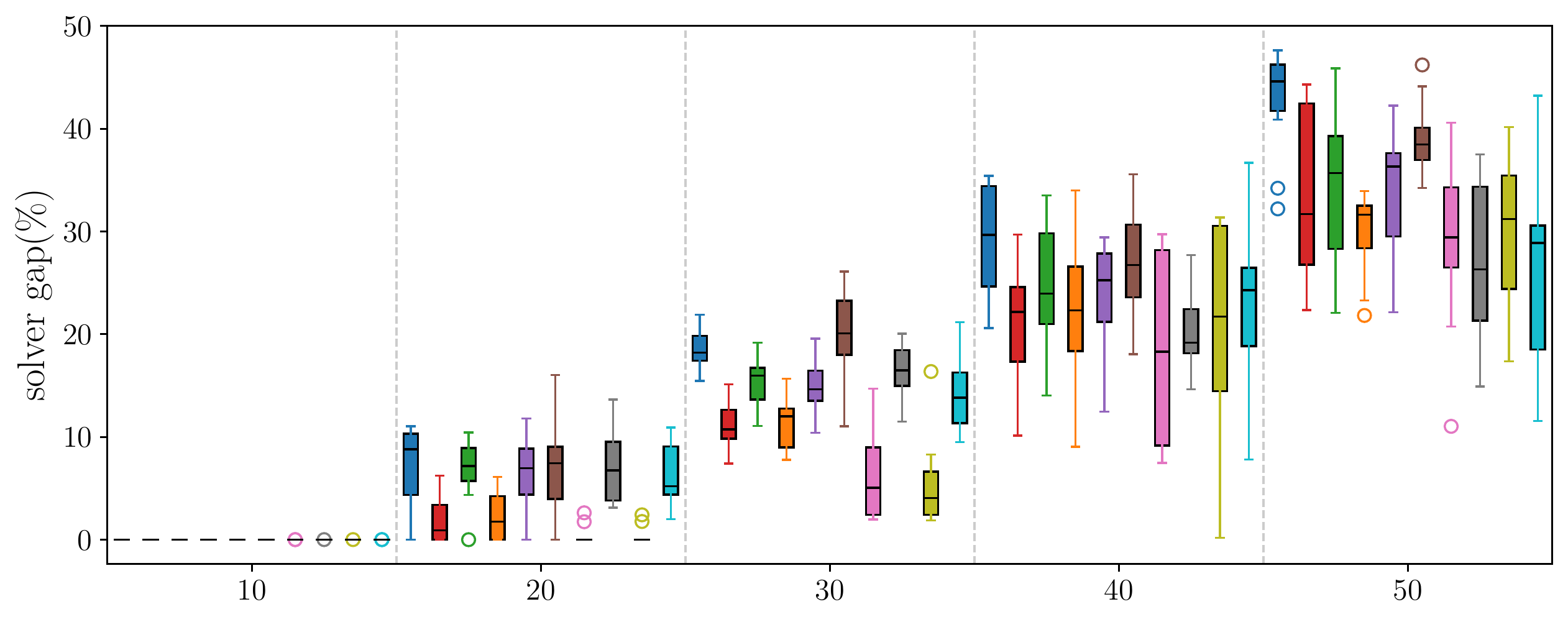}
	\end{subfigure} \\
	\begin{subfigure}[b]{\textwidth}
		\centering
		\includegraphics[width=\textwidth]{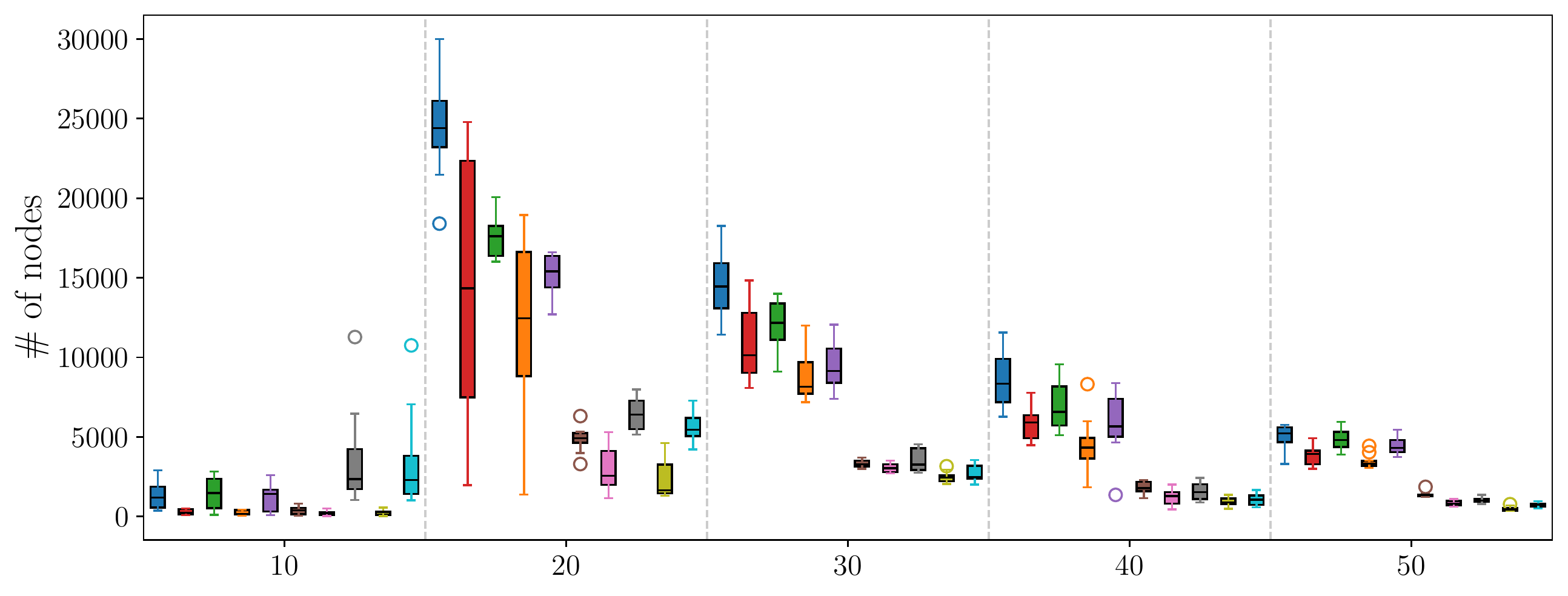}
	\end{subfigure} \\
	\begin{subfigure}[b]{\textwidth}
		\centering
		\includegraphics[width=\textwidth]{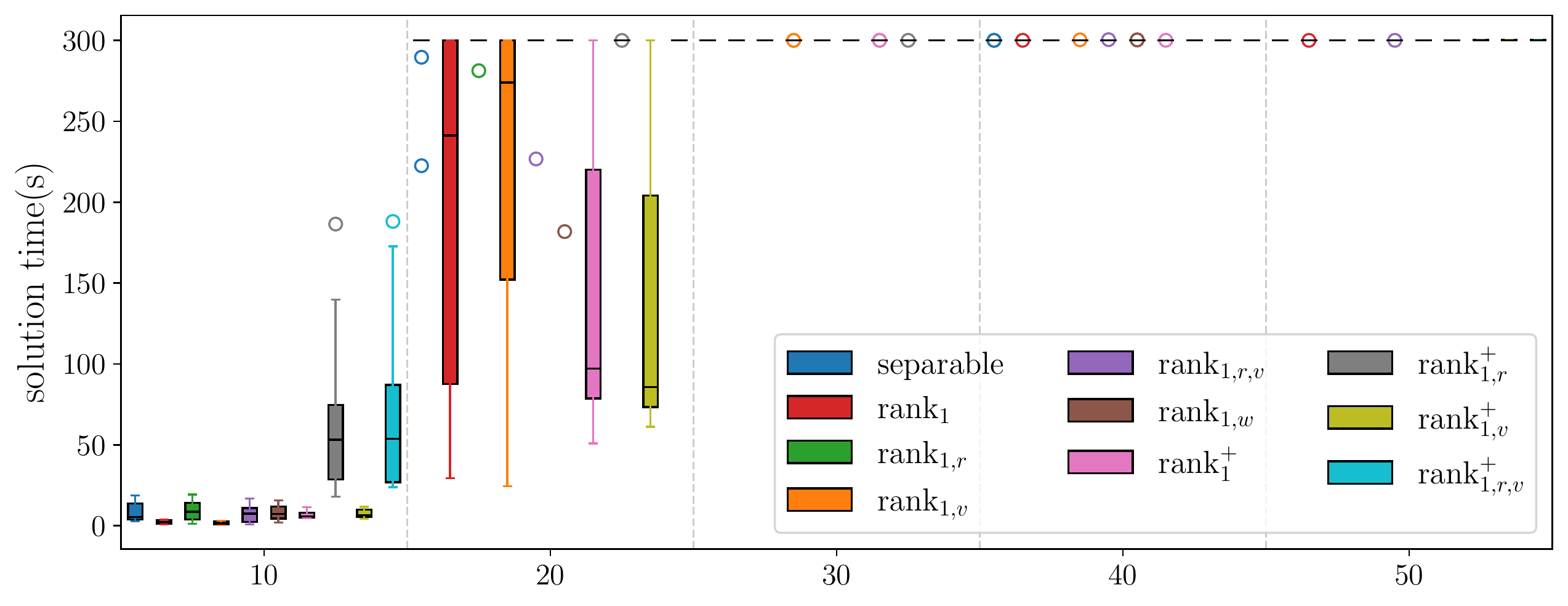}
	\end{subfigure}
	\caption{Performance of the  B\&B algorithm for $N = 100$, $\pi = 0.1$ as $p$  (and thus $d$) varies.}
	\label{fig:change:p:bb}
\end{figure}

\begin{figure}
	\centering
	\begin{subfigure}[b]{\textwidth}
		\centering
		\includegraphics[width=\textwidth]{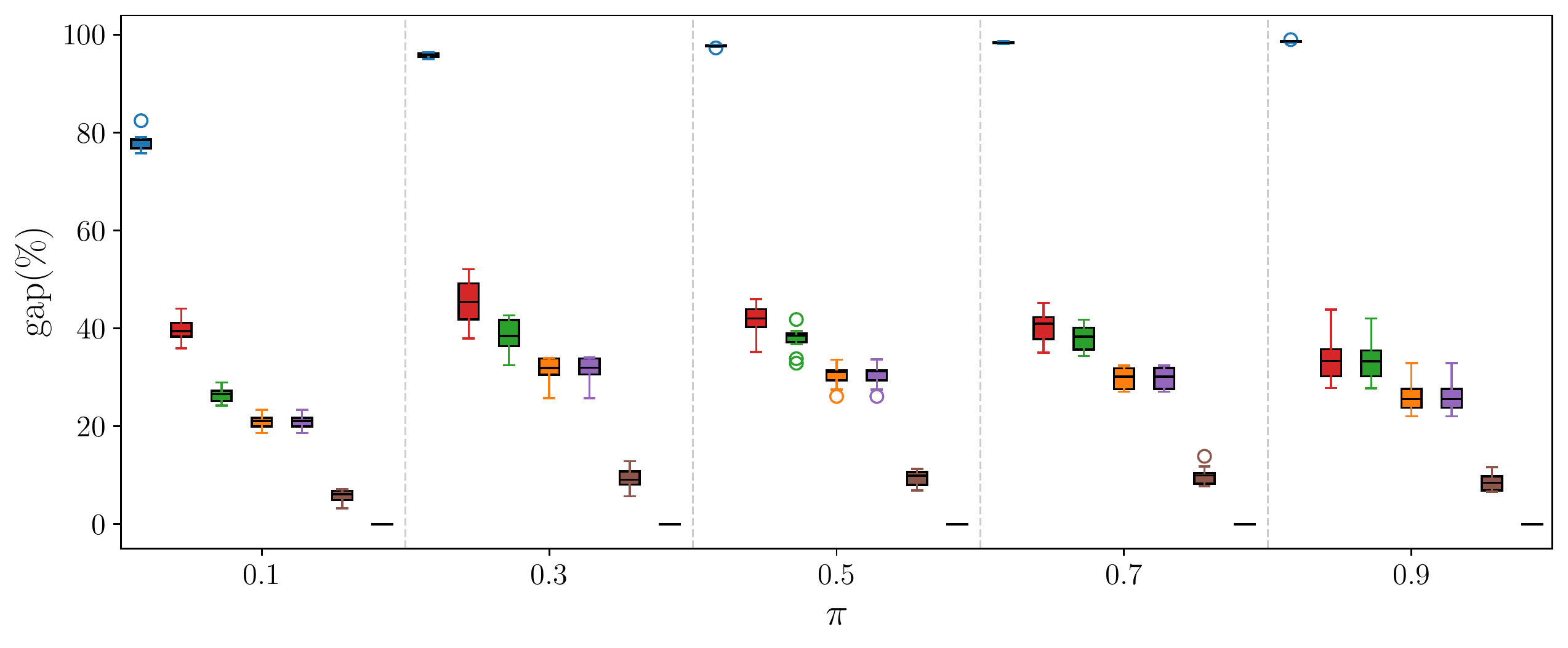}
	\end{subfigure} \\
	\begin{subfigure}[b]{\textwidth}
		\centering
		\includegraphics[width=\textwidth]{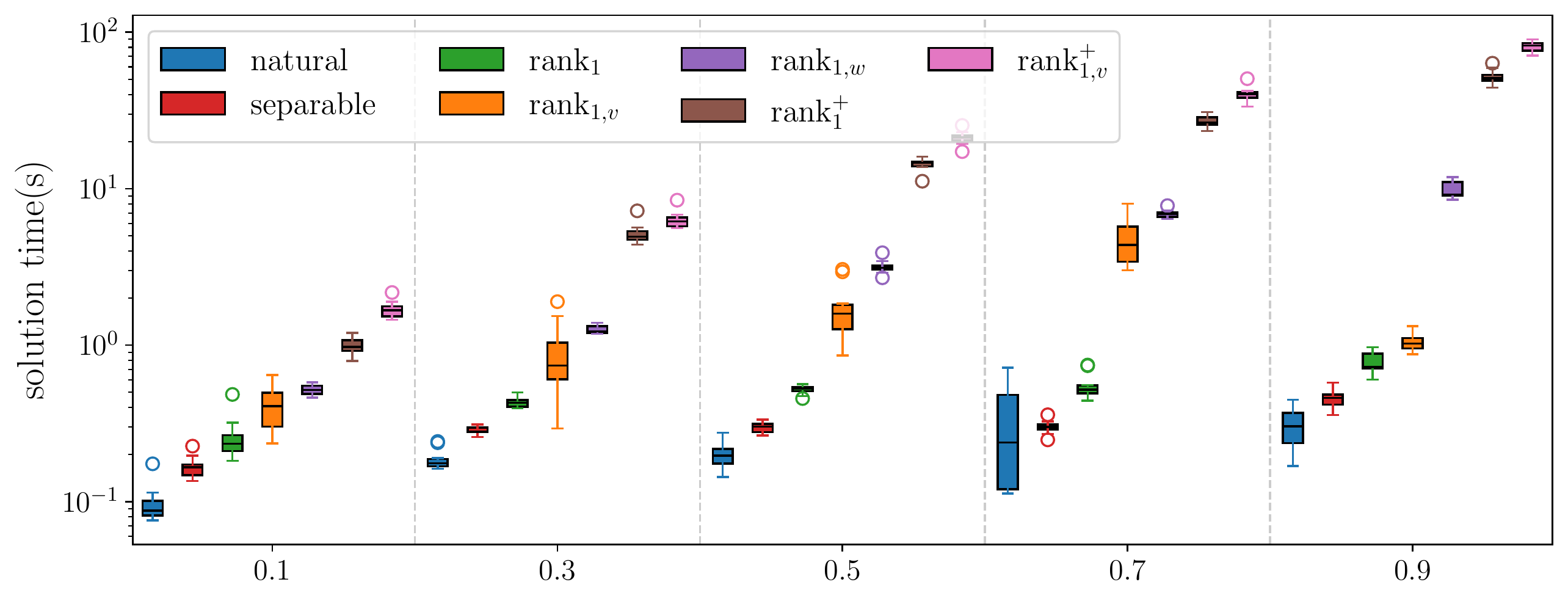}
	\end{subfigure}
	\caption{Comparison of different continuous relaxations for $(p, N) \!=\! (50, 100)$ as $\pi$ varies.}
	\label{fig:change:prob}
\end{figure}

\begin{figure}
	\centering
	\begin{subfigure}[b]{\textwidth}
		\centering
		\includegraphics[width=\textwidth]{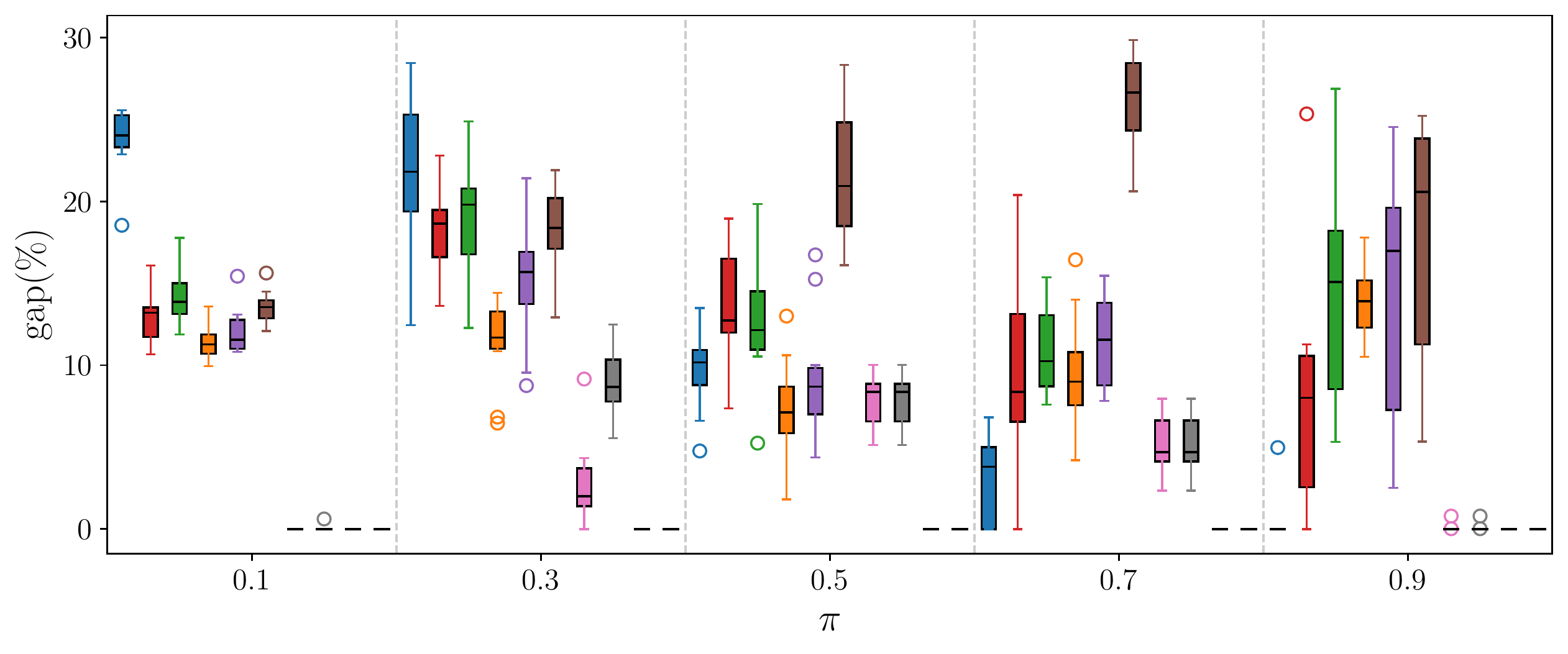}
	\end{subfigure} \\
	\begin{subfigure}[b]{\textwidth}
		\centering
		\includegraphics[width=\textwidth]{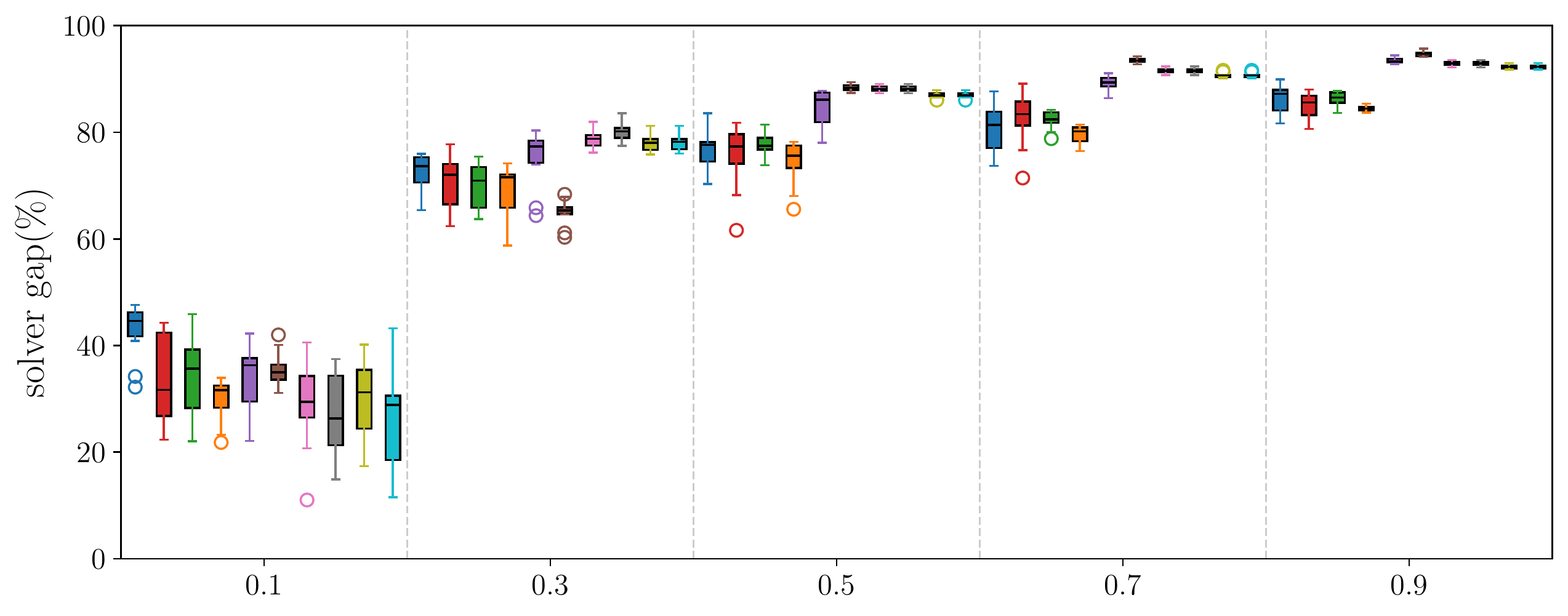}
	\end{subfigure} \\
	\begin{subfigure}[b]{\textwidth}
		\centering
		\includegraphics[width=\textwidth]{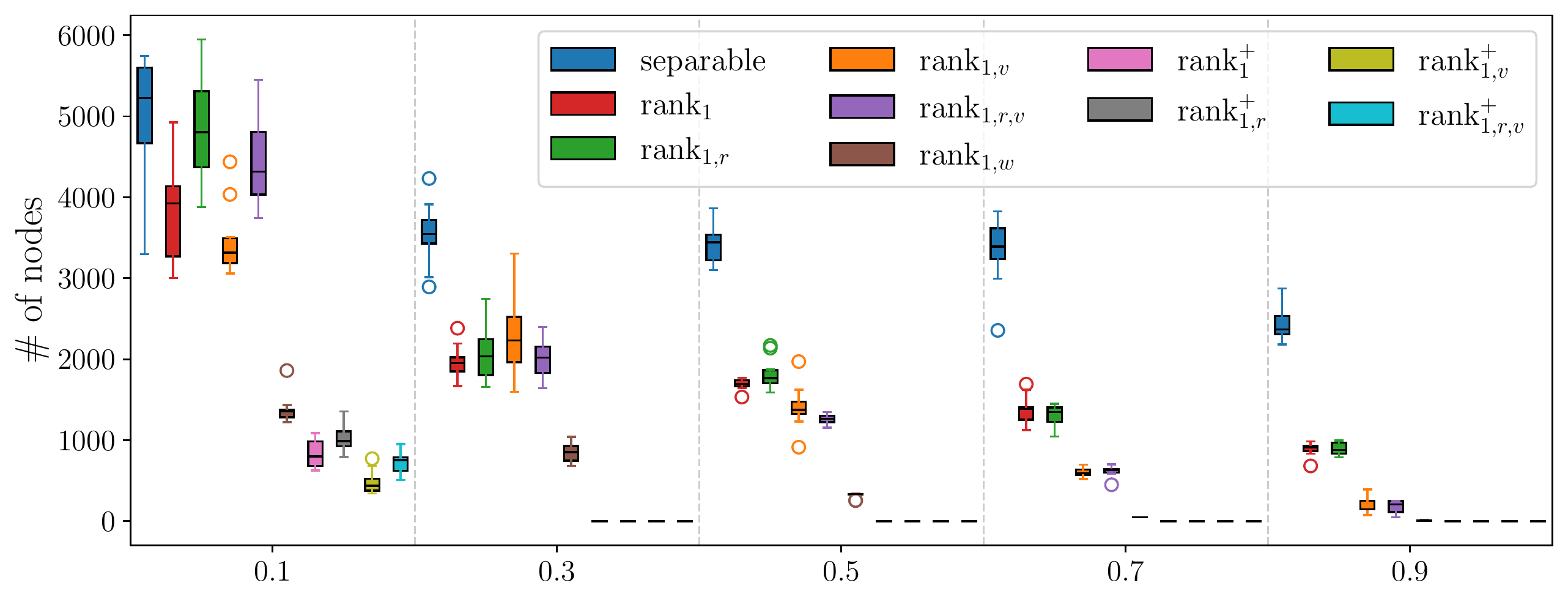}
	\end{subfigure}
	\\
	\begin{subfigure}[b]{\textwidth}
		\centering
		\includegraphics[width=\textwidth]{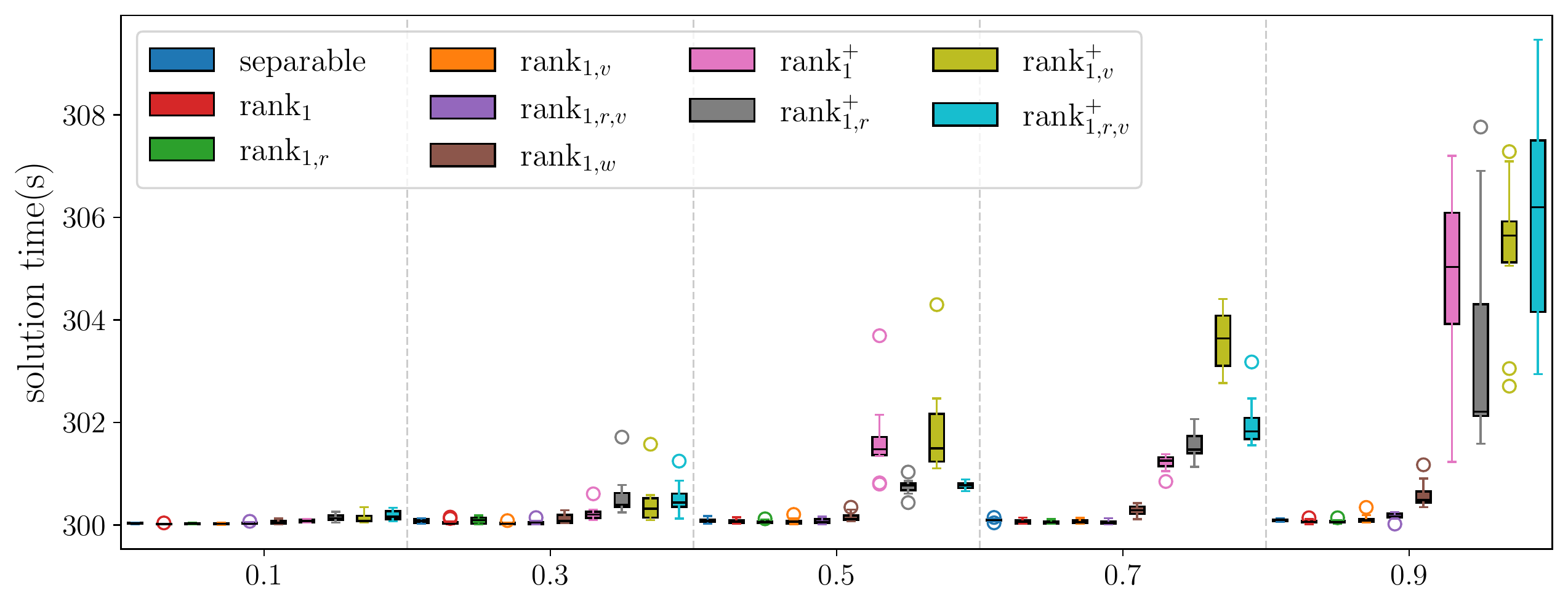}
	\end{subfigure}
	\caption{Performance of the B\&B algorithm for $(p, N) \!=\! (50, 100)$ as $\pi$ varies.}
	\label{fig:change:prob:bb} 
\end{figure}

\end{document}